\documentclass[a4paper,11pt]{amsart}
\usepackage[backref=page]{hyperref}
\usepackage[all]{xy}
\usepackage{latexsym}
\usepackage{amsmath}
\usepackage{enumerate}
\usepackage{eufrak}
\usepackage{mathtools}
\usepackage{amssymb}
\usepackage{pgfplots}
\pgfplotsset{compat = newest}
\usepgfplotslibrary{statistics}
\usepgfplotslibrary{fillbetween}
\usepackage{tikz}
\usetikzlibrary{calc}
\usetikzlibrary{patterns}
\usetikzlibrary{arrows.meta}
\usetikzlibrary{positioning}
\usetikzlibrary{cd}
\usepackage{xcolor}
\usetikzlibrary{shadows}
\usetikzlibrary{external}
\usepackage{tkz-euclide}
\usepackage{scalerel}
\usepackage{bm}

\usepackage{cancel}
\usepackage{color}

\newtheorem{theorem}{Theorem}[section]
\newtheorem{proposition}[theorem]{Proposition}
\newtheorem{corollary}[theorem]{Corollary}
\newtheorem{lemma}[theorem]{Lemma}
\newtheorem{example}[theorem]{Example}

\newtheorem{definition}[theorem]{Definition}

\newtheorem{problem}[theorem]{Problem}
\newtheorem{remark}[theorem]{Remark}
\newtheorem{hypothesis}[theorem]{Hypothesis}

\newcommand{\pb}{\ar@{}[dr]|{\text{\pigpenfont J}}}

\newcommand{\RMod}{R \mbox{-} {\rm Mod}}
\newcommand{\RPInj}{R \mbox{-} {\rm PInj}}
\newcommand{\RInj}{R{\rm \mbox{-Inj}}}
\newcommand{\RFlat}{R \mbox{-}{\rm Flat}}

\newcommand{\RPProj}{R \mbox{-} {\rm PProj}}

\newcommand{\RCotor}{R \mbox{-} {\rm Cotor}}
\newcommand{\Rmod}{R{\rm \mbox{-mod}}}

\newcommand{\Ab}{{\rm {\bf \mbox{Ab}}}}

\newcommand{\Filt}{\rm Filt}

\newcommand{\Add}{{\rm Add}}
\newcommand{\Prod}{{\rm {\bf Prod}}}
\newcommand{\add}{{\rm add}}

\newcommand{\Arr}{{\rm Arr}}

\newcommand{\Char}{{\rm char}}

\newcommand{\Cofilt}{{\rm \mbox{Cofilt}}}
\newcommand{\Coker}{{\rm Coker}}

\newcommand{\Cont}{{\rm Cont}}

\newcommand{\Ext}{{\rm Ext}}

\newcommand{\Hom}{{\rm \operatorname{Hom}}}

\newcommand{\Mod}{{\rm \mbox{Mod}}}

\newcommand{\Ob}{{\rm Ob}}
\newcommand{\On}{{\rm On}}
\newcommand{\cophantom}{{\rm cophantom}}
\newcommand{\op}{{\rm {\footnotesize op}}}

\newcommand{\Proj}{{\rm Proj}}

\newcommand{\Sum}{{\rm Sum}}

\newcommand{\Tor}{{\rm \mbox{Tor}}}

\newcommand{\gra}{\alpha}
\newcommand{\grb}{\beta}
\newcommand{\grg}{\gamma}

\newcommand{\grl}{\lambda}

\newcommand{\grD}{\Delta}

\newcommand{\grS}{\Sigma}
\newcommand{\grO}{\Omega}

\newcommand{\gro}{\omega}

\newcommand{\mcS}{{\mathcal S}}

\newcommand{\mcA}{\mathcal{A}}

\newcommand{\mcE}{\mathcal{E}}

\newcommand{\mcF}{\mathcal{F}}
\newcommand{\mcC}{\mathcal{C}}
\newcommand{\mcI}{\mathcal{I}}
\newcommand{\mcJ}{\mathcal{J}}

\newcommand{\mcM}{\mathcal{M}}

\newcommand{\mcB}{\mathcal{B}}

\newcommand{\mcT}{\mathcal{T}}

\newcommand{\bC}{{\bf C}}

\newcommand{\sfA}{{\sf A}}
\newcommand{\sfB}{{\sf B}}
\newcommand{\sfC}{{\sf C}}
\newcommand{\sfJ}{{\sf J}}
\newcommand{\sfX}{{\sf X}}
\newcommand{\sfXi}{\mathsf{\Xi}}
\newcommand{\sfgrS}{\mathsf{\Sigma}}

\newcommand{\lneg}{\mbox{\rm -}}
\newcommand{\dis}{\displaystyle}

\newcommand{\lb}{\linebreak}

\begin{document}

\footskip30pt

\date{}

\title{Powers of Ghost ideals}

\author{S.\ Estrada}
\address{Departamento de Matem\'aticas, Universidad de Murcia, SPAIN}
\email{sestrada@um.es}

\author{X.H.\ Fu}
\address{School of Mathematics and Statistics, Northeast Normal University, Changchun, CHINA}
\email{fuxianhui@gmail.com}

\author{I.\ Herzog}
\address{The Ohio State University at Lima, Lima, Ohio, USA}
\email{herzog.23@osu.edu}

\author{S.\ Odaba\c{s}\i}
\address{Departamento de Matem\'aticas, Universidad de Murcia, SPAIN}
\email{sinem.odabasi@um.es}

\thanks{S.E. and S.O. were supported by Grant PID2020-113206GB-I00 funded
by MICIU/AEI/10.13039/501100011033 and by Grant 22004/PI/22 funded by Fundación Séneca-
Agencia de Ciencia y Tecnología de la Región de Murcia.\  X.H.F. was partially supported by the National Natural Science Foundation of China, Grant No. 12071064.\ I.H. was partially supported by NSF Grant DMS 12-01523.}



\begin{abstract}
Let $(\mathcal{A};\mathcal{E})$ be an exact category and $\mathcal{S}$ a set of objects in $\mathcal{A}.$ A morphism $f:X\to Y$ is called an $\mathcal{S}$-ghost map if $\Ext(S,f)=0$ for each $S\in\mathcal{S}.$ A theory of ordinal powers of the ideal $\mathfrak{g}_{\mathcal{S}}$ of $\mathcal{S}$-ghost morphisms is developed by introducing for every ordinal $\lambda$, the $\lambda$-th inductive power $\mathcal{J}^{(\lambda)}$ of an ideal $\mathcal{J}.$ The Generalized $\grl$-Generating Hypothesis ($\grl$-GGH) for an ideal $\mcJ$ of an exact category $\mathcal{A}$ is the proposition that  the $\lambda$-th inductive power ${\mathcal{J}}^{(\lambda)}$ is an object ideal.

It is shown that under mild conditions every inductive power of a ghost ideal is an object-special preenveloping ideal. When $\lambda$ is infinite, the proof is based on an ideal version of Eklof's Lemma, which states that if an object $C \in \mathcal{A}$ admits a $\lambda$-$^{\perp}\mathcal{J}$-filtration, then $C$ is left orthogonal to $\mathcal{J}^{(\lambda)}.$ When $\lambda$ is an infinite regular cardinal, the Generalized $\grl$-Generating Hypothesis is established for the ghost ideal $\mathfrak{g}_{\mathcal{S}}$ for the case when $\mcA$ a locally $\lambda$-presentable Grothendieck category and $\mathcal{S}$ is a set of $\lambda$-presentable objects in $\mcA$ such that $^\perp (\mathcal{S}^\perp)$ contains a generating set for $\mcA.$

As a consequence of $\gro$-GGH for the ghost ideal $\mathfrak{g}_{\Rmod}$ in the category of modules $\RMod$ over a ring, it is shown that if the class of pure projective left $R$-modules is closed under extensions, then every left FP-projective module is pure projective. A restricted version $n$-GGH($\mathfrak{g}(\mathbf{C}(R))$) for the ghost ideal in $\mathbf{C}(R))$ is also considered and it is shown that $n$-GGH($\mathfrak{g}(\mathbf{C}(R))$) holds for $R$ if and only if the $n$-th power of the ghost ideal in the derived category $\mathbf{D}(R)$ is zero if and only if the global dimension of $R$ is less than $n.$ If $R$ is coherent, then the Generating Hypothesis holds for $R$ if and only if $R$ is von Neumann regular.  	
\end{abstract}

\maketitle
\setcounter{tocdepth}{1}
\tableofcontents

\section{Introduction}
In algebraic topology and its categorical counterpart, the theory of triangulated categories, a \emph{ghost} is a general term for a morphism that belongs to the kernel of some particular functor. For example, a map $f \colon X \to Y$ in the homotopy category of spectra is called a {\rm ghost} if the induced maps of homotopy groups $\pi(f):\pi(X)\to \pi(Y)$ are zero. It is the subject of a longstanding conjecture.\vspace{10pt}

\noindent {\bf Freyd's Generating Hypothesis (GH).}~\cite[\S 9]{Fre66} \emph{Every ghost map between finite spectra is trivial.}\vspace{10pt}

\noindent The class of ghost maps in the homotopy category of (finite) spectra forms an ideal so the GH can be restated to say that the ghost ideal in the homotopy category of finite spectra is zero.

Kelly~\cite{Kel65} explored the possibility of when finitely many ghosts compose to $0$ in the homotopy category $\mathbf{K}(\mcA),$ where $\mcA$ is an abelian category with enough projectives. Working, as we will be here, in the exact category $\mathbf{C}(\mcA)$ of chain complexes, this means that the finite composition of ghosts belongs to the ``object'' ideal of morphisms that factor though a null-homotopic chain complex. Christensen~\cite{C} makes explicit reference to the term \emph{ghost} in the derived category $\mathbf{D}(\mcA)$, which is defined to be a chain map $f:X\to Y$ whose induced maps on homology  $\operatorname{H}_i(f):\operatorname{H}_i(X)\to \operatorname{H}_i(Y)$ are zero. In the work of Beligiannis~\cite[Section 1]{Bel08}, a morphism in an abelian category $\mcA$ is called a $T$-ghost, for an object $T \in \mcA,$ if it is invisible to the representable functor $\Hom_{\mcA} (T,-).$


Lockridge~\cite{Loc07} initiated the study of GH in a triangulated category and proved that it holds in the derived category $\mathbf{D}^c (R),$ where $R$ is commutative, if and only if $R$ is von Neumann regular. For a ring $R,$ the GH states that the ghost ideal in the subcategory $\mathbf{D}^c (R)$ of perfect complexes is zero. Hovey, Lockridge and Puninski~\cite{HLP07} generalized Lockridge's result to the noncommutative case.

At around the same time, Benson, Chebolu, Christensen and Miná\v{c}~\cite{BCCM07} formulated the GH in the stable module category $kG$-$\underline{\rm mod}$ of a finite $p$-group. Here a ghost is a map between $kG$-modules that is trivial in Tate cohomology, and GH states that the ghost ideal in the stable category $kG$-$\underline{\rm mod}$ of finitely generated $kG$-modules is zero. They show that GH holds for a non-trivial finite $p$-group if and only if $G$ is $C_2$ or $C_3.$ If $G$ is a $p$-group not equal to $C_2$ or $C_3,$ then Chebolu, Christensen and Miná\v{c}~\cite{CCM08b} showed that the ghost ideal in $kG$-$\underline{\rm mod}$ is nilpotent with index of nilpotency less than $|G|.$ The same authors~\cite{CCM08a} also considered the {\bf Generalized Generating Hypothesis (GGH),} which states that the ghost ideal is zero in the ambient stable category $kG$-$\underline{\rm Mod}$ of all $kG$-modules. They showed that it holds if and only if the group is $C_2$ or $C_3.$

Hovey, Lockridge and Puninski suggest the following ring-theoretic proposition as a generalization of GH for a positive integer $n.$\vspace{10pt}

\noindent {\bf The $n$-Fold Generating Hypothesis ($n$-GH).} \cite[p.\ 799]{HLP07} \emph{The $n$-th power of the ghost ideal in $\mathbf{D}^c (R)$ is zero.}\vspace{10pt}

\noindent Hovey and Lockridge~\cite{HL11} proved that $n$-GH holds for a ring if and only if the weak global dimension is less than $n.$ Regarding the ghost ideal in the ambient derived category $\mathbf{D}(R),$ we can combine~\cite[Corollary 8.3]{C} with~\cite{HL} to obtain that the $n$-GGH holds (see below, the generalization of $n$-GH to the ambient category) for a ring if and only if the global dimension of $R$ is less than $n.$ This suggests that one may use the nilpotency index of the ghost ideal in $\mathbf{D}(R)$ to recover the global dimension. \bigskip

In this paper, we will work in an exact category $(\mcA;\mcE).$ If $\mcS$ is a set of objects in $\mcA,$ then the \emph{$\mcS$-ghost ideal} $\mathfrak{g}_{\mathcal{S}}$ (Definition \ref{df-ghost}) consists of morphisms $f \colon X \to Y$ satisfying $\Ext (S,f) = 0$ for any $S \in \mcS.$ This is a unifying notion that includes both the ghosts of Kelly in the homotopy category (Example~\ref{ex:ghost}) as well as those of Benson, Chebolu, Christensen and Miná\v{c} in the stable category (Example~\ref{ex:ghost_tate}).

Inspired by previous work on the (Generalized) GH in the derived category and in modular representation theory we formulate $\lambda$-GGH($\mathfrak{g}_\mcS$) in an exact category as follows. The notation $\mcJ^{(\grl)}$ stands for the $\grl$-th inductive power of an ideal $\mcJ$  (see \S 5.2), and if $\mcC \subseteq \mcA$ is additive, then $\langle \mcC \rangle$ is the ideal of morphisms that factor through some object in $\mcC.$ \bigskip

\noindent {\bf The Generalized $\lambda$-Generating Hypothesis ($\lambda$-GGH($\mathfrak{g}_\mcS$))} \emph{Let $(\mcA; \mcE)$ be an exact category, $\mcS \subseteq \mcA$ a subset and $\grl$ an ordinal. Then
$\mathfrak{g}^{(\lambda)}_{\mathcal{S}} = \langle \mcS^\perp \rangle,$
where $\mcS^\perp = \{X \in \mcA \, | \, \Ext(S,X) = 0 \}.$}
\bigskip

\noindent The main result of the paper may then be stated as follows.
\bigskip

\noindent {\bf Theorem~\ref{EFHS}.} \emph{Let $\lambda$ be an infinite regular cardinal and $\mcA$ a locally $\lambda$-presentable Grothendieck category. If $\mathcal{S}$ is a set of $\lambda$-presentable objects in $\mcA$ such that $^\perp (\mathcal{S}^\perp)$ contains a generating set for $\mcA$, then $\grl$-{\rm GGH(}$\mathfrak{g}_\mcS)$ holds.}
\bigskip

Essential to the statement of $\grl$-GGH, and instrumental in the proof of Theorem~\ref{EFHS}, is the definition of the inductive ordinal power $\mcJ^{(\gra)}$ of an ideal $\mcJ;$ it is defined as the ideal generated by compositions of postcomposable $\gra$-sequences of morphisms from $\mcJ$ (see \S 5.2). These powers form a descending filtration
\begin{equation} \label{Eq:filt}
\mcJ = \mcJ^{(1)} \supseteq \mcJ^{(2)} \supseteq \cdots \supseteq \mcJ^{(\gra)} \supseteq \mcJ^{(\gra + 1)} \supseteq \cdots
\end{equation}
of ideals, bounded below by the object ideal $\langle \Ob (\mcJ) \rangle.$ The Generalized $\lambda$-Generating Hypothesis for a general ideal $\mcJ$ is the convergence statement $\grl$-GGH($\mcJ$): $\mcJ^{(\lambda)} =  \langle \Ob (\mcJ) \rangle.$

The filtration~(\ref{Eq:filt}) resembles, for example, the descending filtration of powers of the Jacobson radical in the category $\Lambda$-${\rm mod}$ of finitely presented modules over an artin algebra, but in contrast to previously defined filtrations of ordinal powers of ideals (see, for example, \cite{Sto10, MR, Prest}), the definition at a limit ordinal $\grl$ is {\em not} given by $\mcJ^{(\grl)} := \bigcap_{\gra < \grl} \, \mcJ^{(\gra)}.$

We will apply an ideal version of the Eklof Lemma to show (Theorem~\ref{T:OSPE preservation}) that the ordinal inductive powers $\mcJ^{(\gra)}$ inherit from $\mcJ$ the property of being object-special preenveloping: if $\mcJ$ is object-special preenveloping, then so are all the ideals of the filtration~(\ref{Eq:filt}). Recall that an ideal $\mcJ$ is {\em object-special preenveloping} if for every object $A \in \mcA,$ there exists an inflation $$\xymatrix{A \ar[r]^j & B \ar[r] & C}$$ with $j \in \mcJ$ and $C \in \Ob ({^{\perp}}\mcJ).$ In that case, the object $C \in {^{\perp}}\mcJ$ is called a $\mcJ$-cosyzygy of $A$ and a subcategory $\grO^{\lneg 1}(\mcJ) \subseteq \Ob ({^{\perp}}\mcJ)$ is a $\mcJ$-{\em cosysygy category} if it contains a $\mcJ$-cosyzygy for every object $A$ of the category. A crucial property of object-special preenveloping ideals is that the orthogonal ideal ${^{\perp}}\mcJ$ may be characterized~\cite[Proposition 7.3]{FH} as the least left perpendicular object ideal containing some (resp., any) $\mcJ$-cosyzygy category $\grO^{\lneg 1}(\mcJ)$ (see \S 3.1).

The ideal version of Eklof's Lemma is a general statement about ideals $\mcJ$ that relates the objects of the left perpendicular ideal ${^{\perp}}\mcJ$ with those of ${^{\perp}}(\mcJ^{(\gra)}).$ \bigskip

\noindent {\bf Theorem~\ref{TFH4}. {\rm (The Ideal Eklof Lemma)}} \emph{Let $\mcJ$ be an ideal of $\mcA$ and $\gra$ an ordinal. If an object $C$ in $\mcA$ has an $\gra$-filtration in $\operatorname{Ob}({^{\perp}}\mcJ),$ then $C \in {^{\perp}(\mcJ^{(\gra)})}.$}
\bigskip

If we denote by $\gra$-$\Filt (\mcC)$ the subcategory of objects that admit an $\gra$-filtration by objects in $\mcC,$ then the Ideal Eklof Lemma may be stated more concisely as the inclusion
$$\gra \mbox{\rm -} \Filt (\Ob ({^{\perp}}\mcJ)) \subseteq \Ob ({^{\perp}}(\mcJ^{(\gra)})).$$
The technical proof of the the Ideal Eklof Lemma is by induction on the ordinal $\gra.$ The successor case is treated in \S 4; the limit case in \S 5. The successor case elaborates on the adjustment technique originally used in Eklof's proof. The $\Hom$-complex between two conflations in an exact category is used to find a method of downward adjustment (see \S 4.3) for a morphism of trivial conflations. The method identifies conditions under which the morphism, regarded as a conflation of arrows, is itself trivial. This downward adjustment method leads to a general version of the classical Eklof Lemma (Corollary~\ref{Eklof_object}) as well as a condition (Theorem~\ref{T:DAP}), relevant in the limit case, for a conflation of $\grl$-systems that is trivial at every ordinal
$\gra < \grl$ to be globally trivial. We point out that while some arguments in \S 5 clearly take inspiration from those of \v{S}\'{t}ov\'{i}\v{c}ek~\cite[Proposition~5.7]{Sto13} our goals require slight differences in argumentation.

These ideas are brought together in \S 6 to prove the following (see Hypothesis~\ref{hyp1}).
\bigskip

\noindent {\bf Theorem~\ref{T:OSPE preservation}.} \emph{If  $\mcJ$ is an object-special preenveloping ideal with $\Omega^{\lneg 1}(\mcJ)$ a cosyzygy subcategory, then for every ordinal $\gra,$ the $\gra$-th inductive power $\mcJ^{(\gra)}$ is an object-special preenveloping ideal with cosyzygy category $\grO^{\lneg 1} (\mcJ^{(\gra)}) = \gra \lneg \Filt (\Omega^{\lneg 1}(\mcJ)).$}
\bigskip

Applying this scenario to the ghost ideal $\mcJ = \mathfrak{g}_{\mathcal{S}}$ yields a proof by induction that every $\mathfrak{g}_{\mathcal{S}}^{(\gra)}$ is object-special preenveloping. The base case $\gra = 1$ is given by the following. It also states that $\Sum (\mcS) = \grO^{\lneg 1}(\mathfrak{g}_{\mathcal{S}})$ is a $\mathfrak{g}_{\mathcal{S}}$-cosyzygy category; note that $\Sum (\mcS) \subseteq {^{\perp}}(\mathfrak{g}_{\mathcal{S}})$ follows from the definition of an $\mcS$-ghost ideal.
\bigskip

\noindent {\bf Theorem~\ref{key}.} {\em Let $\mcS$ be a set of objects in $\mathcal{A}.$ If $\mathcal{A}$ has exact coproducts, then $\mathfrak{g}_{\mathcal{S}}$ is an object-special preenveloping ideal with the class $\Sum (\mcS)$ as a
$\mathfrak{g}_{\mathcal{S}}$-cosyzygy subcategory. More precisely, for any object $A \in \mathcal{A},$ there is a conflation $$\xymatrix{A \ar[r]^g & G \ar[r] & W}$$ with $g$ an $\mcS$-ghost map, and $W \in \Sum (\mcS).$}
\bigskip

To summarize, let us sketch the proof of Theorem~\ref{EFHS}. For an $\mcS$-ghost ideal $\mathfrak{g}_{\mathcal{S}}$, there is a descending filtration of inductive powers of $\mathfrak{g}_{\mathcal{S}}$
$$\mathfrak{g}_{\mathcal{S}} \supseteq\mathfrak{g}^{2}_{\mathcal{S}}\supseteq\cdots  \supseteq\mathfrak{g}^{(\gra)}_{\mathcal{S}} \supseteq \cdots$$
bounded below by $\Ob (\mathfrak{g}_{\mathcal{S}}) = \langle \mathcal{S}^\perp \rangle.$ Using Theorem~\ref{T:OSPE preservation}, we may consider the ascending filtration of cosyzygy subcategories
$$\Sum(\mathcal{S}) \subseteq 2 \lneg \Filt(\Sum(\mathcal{S})) \subseteq \cdots \subseteq \gra \lneg \Filt(\Sum(\mathcal{S})) \subseteq \cdots $$
of ${^{\perp}}(\mcS^{\perp}).$ When the category is a locally $\lambda$-presentable Grothendieck category, and $\mathcal{S}$ is a set of $\lambda$-presentable objects in $\mcA$, the work of \v{S}\'{t}ov\'{i}\v{c}ek~\cite{Sto} (and Enochs~\cite{E2} for a module category) provides the explicit bound $\lambda$ on the length of a filtration in ${^{\perp}}(\mcS^{\perp}) = \grl \operatorname{-}\Filt(\Sum(\mathcal{S})).$ This implies that the $\lambda$-th inductive power $\mathfrak{g}^{(\lambda)}_{\mathcal{S}}=  \langle \mcS^{\perp} \rangle$ is already an object ideal and that $\grl$-GGH($\mathfrak{g}_{\mathcal{S}}$) holds. In this last step of the proof, the convergence statement $\grl$-GGH is seen as $\Ext$-orthogonal to the statement, established by \v{S}\'{t}ov\'{i}\v{c}ek and Enochs, on the convergence of the associated ascending filtration of subcategories.
\bigskip

The methods employed in this paper are a part of Ideal Approximation Theory, which is devoted to the study of {\em ideal cotorsion pairs,} that is, maximal pairs of $\Ext$-orthogonal ideals. The fundamental theorem is an ideal version of Salce's Lemma~\cite{FGHT}, which states that for an ideal cotorsion pair $(\mcI, \mcJ)$ the ideal $\mcI$ is special precovering if and only if $\mcJ$ is special preenveloping. Such ideal cotorsion pairs are called {\em complete} and we contribute to the theory here by proving that if $({^{\perp}}\mcJ, \mcJ)$ is a complete ideal cotorsion pair, with ${^{\perp}}\mcJ$ an object ideal, then so is $({^{\perp}}(\mcJ^{(\gra)}), \mcJ^{(\gra)}),$ for any ordinal $\gra.$ The special case when $\gra = n$ is finite follows from an ideal version of Christensen's Lemma~\cite[Theorem 8.4]{FH}, which was used to settle the Benson-Gnacadja Conjecture~\cite[Corollary 9.4]{FH} by showing that the ideal of phantom morphisms in the stable module category $kG$-$\underline{\rm Mod}$ of a finite group ring is nilpotent with index of nilpotency bounded by $|G|.$ Theorem~\ref{key} above may be regarded as an ideal version of the Eklof-Trlifaj Lemma~\cite[Theorem 3.2.1]{GT}, whose main consequence is that a cotorsion pair (of subcategories rather than ideals) generated by a set $\mcS$ is complete; the {\em ideal} cotorsion pair generated by $\mcS$ is $({^{\perp}}(\mathfrak{g}_{\mathcal{S}}), \mathfrak{g}_{\mathcal{S}}).$

The main contributions of the paper to the theory are: 1) the introduction of the notion $\mcJ^{(\gra)}$ of an inductive ordinal power of an ideal $\mcJ;$ 2) the further development of Ideal Approximation Theory by using the inductive ordinal powers of an ideal to formulate and prove an ideal variant of the classical version of Eklof's Lemma; and finally 3) an application that establishes the connection with the Generalized Generating Hypothesis and confirms the correctness of the ideal variant.  Let us describe some instances of $\grl$-GGH($\mathfrak{g}_{\mathcal{S}}$) in closer detail.

The proper setting for Ideal Approximation Theory is that of an exact category~\cite{B}. It allows us to not only consider complete ideal cotorsion pairs in categories of modules, but also in more general situations, such as the category $\bC (R)$ of chain complexes in $\RMod$, the category $\operatorname{dg-Proj}$ of $\mbox{dg}$-projective chain complexes in $\mathbf{C}(R)$, etc. As in the derived category $\mathbf{D}(R)$, we call a chain map $f$ in the category $\bC (R)$ of complexes a \emph{ghost} if
$\operatorname{H}_i (f)=0$ for any integer $i.$

It turns out that the ideal $\mathfrak{g}(\mathbf{C}(R))$ of ghost maps in $\bC (R)$ is not only object-special preenveloping, but also object-special precovering. Indeed, it is shown in Proposition~\ref{prop:cotorsion_triple_ghost} that there exists a complete ideal cotorsion triple
$$(\langle \operatorname{CE-Proj} \rangle, \mathfrak{g} (\mathbf{C}(R)), \langle\operatorname{CE-Inj}\rangle),$$
where $\operatorname{CE-Proj}$ (resp., $\operatorname{CE-Inj}$) is the category of Cartan-Eilenberg projective (resp., injective) complexes.

The Enochs-\v{S}\'{t}ov\'{i}\v{c}ek bound on filtrations of Cartan-Eilenberg projective complexes is given by $\gro,$ so that $\gro$-GGH($\mathfrak{g}(\mathbf{C}(R))$) holds and the inductive powers of the ideal $\mathfrak{g}(\mathbf{C}(R))$ already converge to the object ideal $\langle \operatorname{Acyc} \rangle$ at stage $\gro.$ But if the global dimension of $R$ is less than $n$, then it will converge at stage $n.$ This is confirmed by the theory developed in this paper.\bigskip

\noindent {\bf Theorem~\ref{theo:global_dim_ghost_dim}.} {\em Let $n \geq 0.$ The following are equivalent:}
	\begin{enumerate}[(i)]
		\item $\mathfrak{g}(\mathbf{D}(R))^{n+1}=0$;
		\item $\mathfrak{g}(\mathbf{C}(R))^{n+1}=\langle \operatorname{Acyc} \rangle,$ i.e., $(n+1)$-GGH($\mathfrak{g}(\mathbf{C}(R)))$ holds;
		\item $\mathfrak{g}(\operatorname{dg-Proj})^{n+1}=\langle\operatorname{Proj}(\mathbf{C}(R))\rangle$;
		\item $\operatorname{l.gl.dim}(R) \leq n$.
\end{enumerate} \bigskip

That (iv $\Rightarrow$ i) is just \cite[Corollary 8.4]{C}, and (i $\Rightarrow$ iv) appears in~\cite{HL}. Applying the same techniques for the ghost ideal between finitely presented chain complexes over a coherent ring, we also obtain the following result, in which the equivalence (i $\Leftrightarrow$ iv) is just~\cite[Theorem 3.1]{HLP07}.
\bigskip

\noindent {\bf Theorem~\ref{GH}.} {\em Suppose that $R$ is left coherent. Then the following are equivalent:}
	\begin{enumerate}[(i)]
		\item $\mathfrak{g}(\mathbf{D}^c(R))=0$;
		\item  $\mathfrak{g}( \mathbf{C}^{\operatorname{fp}}(R))= \langle \operatorname{acyc}^b\rangle$;
		\item  $ \mathfrak{g}(\mathbf{C}^b(\operatorname{proj}))= \langle \operatorname{proj}(\mathbf{C}(R)) \rangle$; {\em and}
		\item $R$ {\em is von Neumann regular,}
	\end{enumerate}
{\em where $\operatorname{proj}(\mathbf{C}(R))$ denotes the class of finitely generated projective chain complexes.}
\bigskip

The most important example of a cotorsion pair in the category of $R$-modules may very well be the pair $(\RFlat, \RCotor)$ cogenerated by the subcategory $\RPInj \subseteq \RCotor$ of pure injective modules. The completeness property of this pair, better known as The Flat Cover Conjecture, was established by Bican, El Bashir and Enochs~\cite{BBE}. The important special case of a ring $R$ over which every left cotorsion module is pure injective, i.e.\ $\RCotor = \RPInj,$ was considered by Xu~\cite{X} in 1996, who proved that this condition is equivalent to the class of pure injective modules being closed under extension. His proof, which relies on the existence of pure injective envelopes, follows readily from Wakamatsu's Lemma. The absence of pure projective covers, however, does not allow us to dualize Xu's proof, but it still makes sense to conjecture the dual of Xu's result for the cotorsion pair $(\mathcal{FP}, \mathcal{FI})$ generated by the subcategory $\RPProj$ of pure projective left $R$-modules, where $\mathcal{FP}$ is the subcategory of FP-projective modules, and $\mathcal{FI}$ is the subcategory of FP-injective modules. In~\cite[Proposition 10.8]{FH}, the second and third authors reprove Xu's result using the ideal approximation theory developed in that paper, but that proof doesn't permit dualization either. In the last sentence of \cite{FH}, the second and third authors promised that they will develop a theory to prove a special case of the dual. We present that general case in this paper! \medskip

\noindent {\bf Theorem~\ref{TFH10}.} {\em Let $R$ be an associative ring. Every left FP-projective $R$-module is pure projective if and only if $\RPProj \subseteq \RMod$ is closed under extensions.}
\medskip


The proof of this result starts from an analysis of the notion of \emph{FP-ghost} ideal $\Psi$ introduced in Section 9, and heavily relies on Theorem~\ref{EFHS}, which reveals the importance of The Generalized Generating Hypothesis. For artin algebras, we have the following two sufficient conditions for $n$-GGH($\Psi$) to hold. \medskip

\noindent {\bf Corollary~\ref{n-GH-FP-ghost1}}	{\em Let $R$ be an artin algebra, and $n$ a positive integer. Then $n$-{\rm GGH}{\rm ($\Psi$)} holds provided one of the following two conditions is satisfied.}
	\begin{enumerate}[(i)]
		\item {\em $R$ has global dimension less than $n;$}
		\item {\em $R$ has pure global dimension less than $n.$}
	\end{enumerate}
\medskip

\noindent {\bf Proposition~\ref{n-GH-FP-ghost2}.} {\em If $R$ is a left artinian ring with $J^n=0$ then $n$-{\rm GGH}{\rm ($\Psi$)} holds.}
\medskip

The paper ends with some open problems related to the contents of this work.
\bigskip

\section{Preliminaries}
Let $\mcA$ be an additive category. An {\em ideal} $\mcI$ of $\mcA$ is an additive subfunctor of $\Hom: {\mcA}^{\op} \times \mcA \to \Ab$, or equivalently, for every morphisms $f:A\to X$, $h:Y\to B$ in $\mcA$ and  $g:X\to Y$ in $\mcI (X,Y)$,  the composition
$hgf:A\to B$ belongs to $\mcI(A,B)$; and for any two  morphisms $g_1:X\to Y$ and $g_2:X\to Y$ in $\mcI(X,Y)$, $g_1-g_2$ also belongs to $\mcI(X,Y)$.

For given ideals  $\mcI$ and $\mcJ$ of  $\mcA$, {\em the  product ideal} $\mcI \mcJ$ is the smallest ideal containing all compositions $ij$ with $i \in \mcI$ and $j \in \mcJ$ composable. One can easily verify that the product ideal $\mcI \mcJ$ in fact consists of  all those compositions $$ \mcI \mcJ= \{ij \mid\ {} i \in \mcI, {} j \in \mcJ\},$$
whenever the pair $(i,j)$ is composable.

For a given ideal $\mcI$ in $\mcA$, an object $A\in\mcA$ is said to belong to $\mcI,$ denoted by  $A \in \mcI,$ if $1_A \in \mcI (A,A).$ We let  $\Ob(\mcI)$ be the subcategory consisting of objects in $\mcI$. By \cite[Proposition 2.1]{FH}, $\Ob(\mcI)$ is an additive subategory of $\mcA$, that is, $\Ob(\mcI)$ is closed under finite direct sums and direct summands.
We let $\langle \Ob (\mcI )  \rangle$ denote the collection of morphims factoring through an object in $\Ob(\mcI)$, which is clearly  an ideal.
$\mcI$   is said to be an {\em object ideal} if $\mcI= \langle \Ob(\mcI )  \rangle$.

%
%

Throughout this paper, we will work in the context of an exact category. As the reference of exact categories, we heavily rely on \cite{B}, but we use the terminology from Keller \cite{K}. An exact category $(\mcA; \mcE)$ is an additive category $\mcA$ equipped with an exact structure $\mcE$, where the {\em exact structure $ \mcE$}   is an isomorphism-closed collection of distinguished kernel-cokernel pairs  $(i,p)$ in $\mcA$
$$ \Xi : \xymatrix@1{  A \ar[r]^i & B \ar[r]^p & C},$$
where $i$ is called the \textit{inflation} and $p$ the \textit{deflation} of $\Xi$, respectively,
satisfying:
\begin{enumerate}[(i)]
\item For every object $A \in \mcA,$  the identity morphism $1_A$ is both an inflation and a deflation;
\item Inflations and  deflations are closed under composition;
\item The class $\mcE$ is closed under  pushout  and pullback   along any morphism in $\mcA$.
\end{enumerate}

A kernel-cokernel pair  $(i,p)$, or $\Xi$, in $\mcE$ is  called a {\em conflation in } $(\mcA; \mcE)$.  We will also sometimes say that $C$ is the cokernel of $i$ and use the notation $C = B/A;$ $A$ is the kernel of $p,$ which may be denoted by $A \subseteq B$. For given objects $A$ and $C$ in $\mcA$, we will assume that the isomorphism classes of conflations $\Xi$ in $\mathcal{E}$ form a set denoted by $\Ext (C,A)$, and therefore, it is an additive group with respect to the {\em Baer sum} operation. Hence, a conflation $\Xi: \xymatrix@1{A \ar[r]^i & B \ar[r]^p & C}$ in $\mcE$ represents $0$ in $\Ext(C,A)$ if it is \textit{trivial}, that is, the morphism $i$ is a section: there exists a retraction $r: B \to A$ in $\mcA$ such that $ri=1_A$; or equivalently, $p$ is a retraction: there exists a morphism $s:C \rightarrow B$ such that $ps=1_C$; see \cite[Remark 7.4]{B}.

From now on, the category $\mathcal{A}$ is assumed to be equipped with an exact structure $\mathcal{E}$ unless otherwise is stated. Note that the rule $(C,A)\mapsto\Ext(C,A)$ is an additive bifunctor from $\mathcal{A}^{\textrm{op}} \times \mathcal{A}$ to $\Ab$. For a morphism  $f:A\to X$  and  an object $C$ in $\mcA$, the morphism $\Ext(C,f):\Ext(C,A)\to\Ext(C,X)$ sends a conflation $ \xymatrix@1{  A \ar[r]^i & B \ar[r]^p & C}$ in $\Ext(C,A)$  to the conflation $\xymatrix{X \ar[r] & Y \ar[r] & C}$ in $\Ext(X,C)$ appearing in the following pushout diagram
$$\xymatrix{A \ar[r]^i \ar[d]_f & B \ar[r]^p \ar[d] & C\ar@{=}[d]\\ X \ar[r] & Y \ar[r] & C.}$$
Therefore, $\Ext(C,f)=0$ if and only if for any conflation $ \xymatrix@1{  A \ar[r]^i & B \ar[r]^p & C}$ in $\Ext(C,A)$, the conflation $\xymatrix{X \ar[r] & Y \ar[r] & C}$ in the bottom row is trivial.

For a morphism  $g:Y\to C$ and an object $A$ in $\mathcal{A}$, the morphism $\Ext(g,A):\Ext(C,A)\to\Ext(Y,A)$ is defined dually.
We have  $\Ext(f,g)=\Ext(f,A)\Ext(g,Y)=\Ext(g,C)\Ext(f,X):\Ext(C,A)\to\Ext(Y,X).$

The following lemma not only describes the finite version $(\grl = n)$ of Theorem~\ref{TFH4}, but also illustrates the argument behind the inductive step of the transfinite case.

\begin{lemma}\label{finiteghost}
	Let $f=f_nf_{n-1}\cdots f_1$ be a sequence of composable morphisms in $\mathcal{A}$
	$$\xymatrix{X_1\ar[r]^{f_1}&X_2\ar[r]^{f_2}&\cdots\ar[r]&X_{n-1}\ar[r]^{f_n}&X_n}.$$ If $0=C_0\subseteq C_1\subseteq C_2 \subseteq\cdots\subseteq C_n=C$ is a sequence of inflations in $\mathcal{A}$ such that for each  $1 \leq i \leq n$,
$\Ext(C_i/C_{i-1},f_i)=0$, then $\Ext(C,f)=0$.
\end{lemma}
\begin{proof}The proof is by induction. So we only need to prove this result for $n=2$. Applying the functor $\Ext(-,f)=\Ext(-,f_2)\Ext(-,f_1)$ to the conflation $\xymatrix@1{C_1 \ar[r]^{\iota} & C_2 \ar[r]^-{\pi} & C_2/C_1},$ we obtain a commutative diagram $$\xymatrix@C=40pt@R=35pt{\Ext(C_2/C_1,X_1)\ar[r]\ar[d]&\Ext(C_2,X_1)\ar[r]\ar[d]^{\Ext(C_2,f_1)}&\Ext(C_1,X_1)\ar[d]^{\Ext(C_1,f_1)=0}\\
	\Ext(C_2/C_1,X_2)\ar[r]^-{\Ext(\pi,X_2)}\ar[d]_{\Ext(C_2/C_1,f_2)=0}&\Ext(C_2,X_2)\ar[r]^-{\Ext(\iota,X_2)}\ar[d]^{\Ext(C_2,f_2)} &\Ext(C_1,X_2)\ar[d]\\
\Ext(C_2/C_1,X_3)\ar[r]&\Ext(C_2,X_3)\ar[r]&\Ext(C_1,X_3)}$$
with exact rows. Let $a\in\Ext(C_2,X_1)$. Since $\Ext(C_1,f_1)=0$, then from the right upper commutative square, we have that $\Ext(C_2,f_1)(a)$ belongs to the kernel of $\Ext(\iota,X_2)$, and hence belongs to the image of $\Ext(\pi,X_2)$. Let $b\in\Ext(C_2/C_1,X_2)$ such that $\Ext(C_2,f_1)(a)=\Ext(\pi,X_2)(b)$. Then from the lower left commutative square, we obtain that $\Ext(C_2,f)(a)=\Ext(C_2,f_2)\Ext(C_2,f_1)(a)=\Ext(C_2,f_2)\Ext(\pi,X_2)(b)=\Ext(\pi,X_3)\Ext(C_2/C_1,f_2)(b)=0$, as desired.
	\end{proof}

For a given class $\mathcal{S}$ of objects in $\mcA$, we let
$\mathcal{S}^\perp = \{ A \in \mcA \mid \Ext(S,A)=0, \textrm{ for every } S \in \mathcal{S} \}$ and by $^\perp \mathcal{S}  = \{ A \in \mcA \mid \Ext(A,S)=0,  \textrm{ for every } S \in \mathcal{S}\}.$ Then a pair $( \mathcal{F}, \mathcal{C})$ of classes of objects in $\mcA$ is called	 a \textit{cotorsion pair} if $ \mathcal{F} ^\perp= \mathcal{C}$ and $\mathcal{F}= {}^\perp \mathcal{C}$. A cotorsion pair $( \mathcal{F}, \mathcal{C})$ is said to \textit{have enough injectives} if for every object $A $ of $\mcA$, there exists a  conflation in $\mcA$ of the form
$$
A\longrightarrow C \longrightarrow F
$$
with $C \in  \mathcal{C}$ and $F\in \mcF$. That a cotorsion pair $( \mathcal{F}, \mathcal{C})$ is said to be \textit{have enough projectives} is defined dually.  $( \mathcal{F}, \mathcal{C})$ is a
\textit{complete cotorsion pair} if it has both enough projectives and enough injectives. Similarly, for a given class  $\mathcal{M}$ of morphisms in $\mcA$, we let $\mathcal{M}^\perp $ and ${}^\perp \mathcal{M}$  denote the classes of morphisms $f$ in $\mathcal{A}$ such that $\Ext(m,f)=0$ and $\Ext(f,m)=0$  for every  $m$ in $\mathcal{M}$, respectively.  It is clear that both classes $\mathcal{M}^\perp$ and ${}^\perp \mathcal{M}$ are  ideals in $\mcA$.  A pair $(\mcI, \mcJ)$ of ideals in $\mcA$ is said to be an \textit{ideal cotorsion pair} if $ \mcI ^\perp= \mcJ$ and $\mcI= {}^\perp \mcJ$.

A special case of the foregoing is the cotorsion pair $(\mathcal{E}\mbox{-Proj}, \mcA)$ where $\mathcal{E}\mbox{-Proj}$ denotes the subcategory of \textit{projective objects} of $\mcA.$ Thus an object $P \in \mcA$ is projective if $\Ext (P,X)=0$ for every object $X \in \mcA,$ or equivalently, if the functor $\Hom(P,-)$ is exact. The exact category $\mcA$ is said to have \textit{enough projectives} if the cotorsion pair $(\mathcal{E}\mbox{-Proj}, \mcA)$ does. The notion of an \textit{injective object} and the subcategory $\mathcal{E}\mbox{-Inj}$ are defined dually.

Recall that an \textit{exact substructure} of $\mathcal{E}$ is an exact structure $\mcE'$ on $\mcA$ with $\mcE' \subseteq \mcE$. For a given class $\mcS$ of objects in $\mcA$, define $\mathcal{E}_{\mcS}$ to be the class of all conflations $\Xi$ for which the functor $\Hom(S,-) $ preserves the exactness of $\Xi$ for every $S \in \mathcal{S}$. It is readily verified that $\mathcal{E}_{\mcS}$ is an exact substructure of $\mcE$ with  $\mcS \subseteq \mcE_{\mcS} \mbox{-Proj}.$ This exact substructure $\mcE_{\mcS}$ is called \emph{the exact structure projectively generated by} $\mcS.$ The notion of \emph{the exact structure injectively cogenerated by a class $\mcS$} is defined dually.

Let $\mbox{Sum}(\mathcal{S})$ denote the class of coproducts of objects in $\mathcal{S}$ and denote by $\Add (\mcS)$ ($\add(\mcS)$) the smallest subcategory containing $\mcS$ which is closed under (finite) coproduct and direct summands. An object belongs to $\mbox{Add}(\mcS)$ ($\add(\mcS)$) if and only if it is a direct summand of a (finite) coproduct of objects in $\mcS.$ Indeed, by \cite[Corollary 11.4 and  11.7]{B}, we have that $\Add (\mcS) \subseteq \mcE_{\mcS} \mbox{-Proj}.$

A {\em generating set} for an exact category $\mcA$  is a set $\mcS$ of objects in $ \mcA$ such that every object $A \in \mcA$ admits a deflation $\coprod_{S \in \mcS} \; S \to A$ in $\mathcal{E}$ from some coproduct of objects in $\mcS.$ The following idea originates in the work of Stenstr\"{o}m~\cite[Proposition 2.3]{St}.

\begin{proposition}\label{projectivelygenerated}
If $\mcA$ has coproducts, and  $\mcS$ is a generating set for  $\mcA$,  then the exact substructure $\mathcal{E}_{\mcS} \subseteq \mcE$ has enough projectives. Moreover, $\mathcal{E}_{\mcS}\textrm{-}\Proj=\Add(\mcS).$
\end{proposition}
\begin{proof}
	Let $A$ be an object in $\mcA$. By hypothesis, we have a deflation $\coprod_{S \in \mcS} \; S \to A$ in $\mathcal{E}$ from a coproduct of objects in $S.$ The domain $\coprod_{S \in \mcS} \; S$ is $\mcE_{\mcS}$-projective. If every map from an object in $\mcS$ to $A$ is represented in this deflation, then it is an $\mcE_{\mcS}$-deflation. If not, supplement the coproduct by adding factors to the domain so that every map from an object in $\mcS$ to $A$ is represented. By the argument dual to the equivalence (i) $\Leftrightarrow$ (ii) of~\cite[Proposition 2.12]{B}, this map is still a deflation in $\mathcal{E}_{\mathcal{S}}$ and the domain is still $\mcE_{\mcS}$-projective.
	\end{proof}
\bigskip

\section{Ghost ideals in an exact category}
In this section, we introduce the notion of an $\mcS$-ghost ideal in $\mcA$, and study its approximation properties. We start by
showing that a finite power of an object-special preenveloping ideal is still object-special preenveloping.
\subsection{Object-special preenveloping ideals} An ideal $\mcJ$ in $\mathcal{A}$ is said to be  an \emph{object-special preenveloping} ideal if for any object $A$, there is a conflation $$\xymatrix{A\ar[r]^j&J\ar[r]&W}$$ with $j\in\mcJ$, and $W\in {^{\perp}\mcJ}$. The object $W \in {^{\perp}\mcJ}$ is called a $\mcJ$-cosyzygy object of $A$, and is denoted by $\Omega^{-1}_{\mcJ }(A)$. It is easy to check that the morphism $j$ is a $\mcJ$-preenvelope of $A$, that is, any morphism $j':A\to J'\in\mcJ$ factors through $j,$ as in $$\xymatrix{A\ar[d]_{j'}\ar[r]^j&J\ar[ld]\\J'.}$$
A subcategory of $\mcA$ that is closed under finite direct sums is called a $\mcJ$-\emph{cosyzygy subcategory} if it contains a $\mcJ$-cosyzygy $\Omega^{-1}_{\mcJ }(A)$ for every object $A$ in $\mcA$. Such a subcategory will be denoted by $\Omega^{-1}(\mcJ).$

Let $\mathcal{C} \subseteq \mcA$ be a subcategory and $n$ a positive integer. An object $C$ of $\mcA$ is said to admit an \emph{$n$-$\mathcal{C}$-filtration} if there is a sequence of inflations
$C_0=0 \subseteq C_1 \subseteq C_2 \subseteq \cdots \subseteq C_{n-1} \subseteq C_n=C$ with each $C_i/C_{i-1} \in \mathcal{C}$. We denote by $n$-$\Filt (\mathcal{C})$ the class of objects in $\mcA$ possessing an $n$-$\mathcal{C}$-filtration. The following result follows from~\cite[Proposition 6.2 and Corollary 8.2]{FH}.

\begin{proposition}\label{finite1}
	Let $\mcJ$ be an object-special preenveloping ideal in $\mathcal{A}$ with $\Omega^{-1}(\mcJ)$ a cosyzygy subcategory. Then the  $n$-th power $\mcJ^n$ is still an object-special preenveloping ideal with $n$-$\Filt(\Omega^{-1}(\mcJ))$ a cosyzygy category.
\end{proposition}
\begin{proof}
	An object-special $\mcJ^n$-preevelope of $A \in \mcA$ can be constructed in the following way. Firstly, there is a conflation
	$$\xymatrix{A\ar[r]^j&J\ar[r]&W}$$ with $j$ a special $\mcJ$-preenvelope of $A$, and $W = J/A \in\Omega^{-1}(\mcJ)$. For $J_1:=J$, there is a conflation
	$$\xymatrix{J_1\ar[r]^-{j^1_2}&J_2\ar[r]&W^1_2}$$ with $j^1_2$ a special $\mcJ$-preenvelope of $J_1$, and $W^1_2 \in \Omega^{-1}(\mcJ)$. Continuing this process,  there is a conflation
	$$\xymatrix{J_{n-1}\ar[r]^-{j^{n-1}_n}&J_n\ar[r]&W^{n-1}_n}$$ with $j^{n-1}_n$ a special $\mcJ$-preenvelope of $J_{n-1}$, and $W^{n-1}_n\in\Omega^{-1}(\mcJ)$. Composing the morphisms $j$, $j^1_2$, $\cdots, j^{n-1}_n$, we attain the following commutative diagram
	$$\xymatrix{A\ar[r]^j\ar@{=}[d]&J\ar[r]\ar[d]^{j^1_2}&W\ar[d]\\
		A\ar[r]^{j_2} \ar@{=}[d]&J_2\ar[r]\ar[d]^{j^2_3}&W_2\ar[d]\\
		\vdots\ar@{=}[d]&\vdots\ar[d]^{j^{n-1}_n}&\vdots\ar[d]\\
		A\ar[r]^{j_n}&J_n\ar[r]&W_n}.$$
	In the conflation $\xymatrix{A\ar[r]^{j_n}&J_n\ar[r]&W_n}$ on the bottom, $j_n\in\mcJ^n$, and $W_n\in n$-$\mbox{Filt} (\Omega^{-1}(\mcJ))$ since \linebreak $0\to W\to W_2\to\cdots\to W_n$ is an $n$-$\Omega^{-1}(\mcJ)$-filtration. By Lemma~\ref{finiteghost}, $n$-$\mbox{Filt} (\Omega^{-1}(\mcJ)) \subseteq {}^\perp \mathcal{J}^n$. Therefore $j_n$ is an object-special $\mcJ^n$-preenvelope of $A$.
\end{proof}

\subsection{Ghosts}
\begin{definition} \label{df-ghost} Let $\mathcal{S}$ be a set of objects in $\mcA$. A morphism $f:X\to Y$ in $\mathcal{A}$ is called an $\mathcal{S}$-\emph{ghost} if $\Ext(S,f)=0$ for any $S\in\mathcal{S}$.
	\end{definition}
The $\mathcal{S}$-ghost morphisms in $\mathcal{A}$ form an ideal, we call it the $\mathcal{S}$-\emph{ghost} ideal, and denote it by $\mathfrak{g}_{\mathcal{S}}$. It is clear that $({^{\perp}\mathfrak{g}_{\mathcal{S}}},\mathfrak{g}_{\mathcal{S}})$ is an ideal cotorsion pair and that every object in $\Add(\mathcal{S})$ belongs to ${^{\perp}\mathfrak{g}_{\mathcal{S}}}.$ Also $\mcS^{\perp} = \Ob (\mathfrak{g}_{\mathcal{S}})$ so that  $\langle\mcS^{\perp}\rangle \subseteq \mathfrak{g}_{\mathcal{S}}.$

\begin{example}\label{ex:ghost_tate} \rm
	Let $G$ be a finite group and $k$ a field with $\Char (k) \! \mid \! o(G).$ A morphism $f: M \rightarrow N$ of left $kG$-modules is called a \textit{ghost} if it induces zero on Tate cohomology, that is, $\widehat{\operatorname{H}}^n(G,f)=0$ for every $n \in \mathbb{Z}$, see \cite{BCCM07, CCM08a}.
	
	Denote by $P_*$ a complete projective resolution of the trivial left $kG$-module $k$, that is, an acyclic chain complex of projective left $kG$-modules, with group of $0$-cycles $\operatorname{Z}_0(P_*) \cong k,$ that remains exact under the functor $\Hom(-,M)$ for every projective left $kG$-module $M.$ Let $\mcS$ be the set $ \{\Omega^{n}(k) \}_{n \in \mathbb{Z}}$, where $\Omega^{n}(k):=\operatorname{Z}_{n}(P_*)$ is the group of $n$-cycles of $P_*.$ Then a morphism $f:M \rightarrow N$ in $kG$-$\Mod$ is a ghost in the sense of \cite{CCM08a} if and only if it is an $\mcS$-ghost.
	Indeed, for every $n \in \mathbb{Z}$, we have
	$$\widehat{\operatorname{H}}^n(G,-) \cong \operatorname{H}_{-(n+1)}(\Hom(P_*,-)) \cong \Ext( \Omega^{n-1}(k),- );$$
	for further details, see \cite[Section VI-4]{Bro82}.
\end{example}

\begin{example}\label{ex:ghost} \rm
	Let $\mathbf{C}(R)$ and $\mathbf{K}(R)$ denote the category of chain complexes of left $R$-modules and the homotopy category of $\mathbf{C}(R)$, respectively.
	Consider the homology functor
	$$\operatorname{H}_*:  \mathbf{C}(R) \longrightarrow \mathbf{K}(R) \longrightarrow R\operatorname{-Mod.}$$
	Kelly in \cite{Kel65} studied morphisms in $\mathbf{K}(R)$ that are invisible under the homology functor, so we call a morphism $f$ in $ \mathbf{C}(R)$ \textit{ghost} if it is invisible under homology. Let $\mcS$ denote the set $\{ \mbox{S}^n(R)\}_{n \in \mathbb{Z}}$ of spherical chain complexes, where $\mbox{S}^n (R)$ is the chain complex with $R$ in the $n$-th position, and $0$ elsewhere. A morphism $f$ in $ \mathbf{C}(R)$ is ghost if and only if it is $\mcS$-ghost; see Proposition \ref{prop:relative_ghost}. We will investigate this ideal further in Section 8.
\end{example}

Recall that for a given subfunctor $\mcF \subseteq \Ext$,  a morphism $f:X\to Y$ in $\mcA$ is called  an $\mcF$-$\cophantom$ (see \cite{FGHT}) if the pushout  of any conflation in $\mathcal{A}$ along $f$ is a conflation in
$\mcF$. Such morphisms form the $\mcF$-\textit{cophantom ideal} in $\mathcal{A}$.  Every exact substructure $\mcE' \subseteq  \mcE$ can be regarded as a subfunctor of $\Ext,$ in which case the ideal is denoted by $\mcE'$-$\cophantom.$

\begin{proposition} For a given set $\mcS$ of objects in  $ \mcA$, a morphism $f:X\to Y$ is $\mcS$-ghost if and only if it is $\mcE _{\mcS}$-cophantom, i.e., $\mathfrak{g}_{\mathcal{S}} = \mcE _{\mcS}\mbox{-}\cophantom.$
\end{proposition}

\begin{proof} If $f$ is an $\mcE_{\mcS}$-cophantom, then $\Ext (S,f) = 0$ for every $S \in \mcS,$ because $S$ is $\mcE_{\mcS}$-projective. For the converse, let $f \colon X \to Y$ be an $\mcS$-ghost and consider a pushout of a conflation in $\mathcal{E}$ along $f$ as in

	$$\xymatrix{X\ar[r]\ar[d]_f&B\ar[r]\ar[d]&C\ar@{=}[d]\\
		Y\ar[r]&Z\ar[r]&C.}$$
	For $S \in \mcS,$ apply the functor $\Hom(S,-)$ to the pushout diagram
	to obtain the following commutative diagram of exact sequences
	$$\xymatrix{0\ar[r]&\Hom(S,X)\ar[r]\ar[d]&\Hom(S,B)\ar[r]\ar[d]&\Hom(S,C)\ar[r]\ar@{=}[d]&\Ext(S,X)\ar[d]^{\Ext(S,f)}\\0\ar[r]&\Hom(S,Y)\ar[r]&\Hom(S,Z)\ar[r]&\Hom(S,C)\ar[r]&\Ext(S,X)}$$
	By hypothesis, $\Ext(S,f)=0$, so that  $$\xymatrix{0\ar[r]&\Hom(S,Y)\ar[r]&\Hom(S,Z)\ar[r]&\Hom(S,C)\ar[r]&0}$$ is an exact sequence, and hence $\xymatrix{Y\ar[r]&Z\ar[r]&C}$ belongs to $\mcE_{\mcS},$ as required.
\end{proof}

If $\mathcal{A}$ has coproducts, and $\mcS$ is a generating set for $\mcA$, then by Proposition \ref{projectivelygenerated}, the exact structure $\mcE_{\mcS}$ has enough projective objects. If we assume further that $\mcA$ has enough projective objects, then it has been shown in~\cite[Theorem 2]{FGHT} that $\mathfrak{g}_{\mathcal{S}}$ is an object-special preenveloping ideal. However, this result still holds even without these conditions.

\begin{theorem}\label{key}	
Let $\mcS$ be a set of objects in $\mathcal{A}$. If $\mathcal{A}$ has   exact coproducts\footnote{$\mathcal{A}$ is said to \textit{have exact coproducts} if it has coproducts, and a coproduct of any conflations is conflation.}, then  $\mathfrak{g}_{\mathcal{S}}$ is an object-special preenveloping ideal with the class $\Sum (\mcS)$   as a  $\mathfrak{g}_{\mathcal{S}}$-cosyzygy subcategory. More precisely, for any object $A\in\mathcal{A}$, there is a conflation
	$$\xymatrix{A\ar[r]^g&G\ar[r]&W}$$ with $g$ an $\mcS$-ghost map, and $W \in \Sum (\mcS)$.
\end{theorem}

\begin{proof} We follow an argument that appears in \cite{FHHZ} with a subtle modification. Let $S$ be the coproduct of all objects in the set $\mcS$ and let $A$ be an object in $\mathcal{A}$. By assumption, the coproduct of all conflations $\xymatrix{A\ar[r]&B_i\ar[r]&S}$ in $I = \Ext(S,A)$ exists, and is a conflation in $\mathcal{A}$.  Consider the pushout diagram   along the canonical morphism $A^{(I)}\to A$
	$$\xymatrix{A^{(I)}\ar[d]\ar[r]&\coprod_{i\in I}B_i\ar[d]\ar[r]&S^{(I)}\ar@{=}[d]\\
	A\ar[r]^j&G\ar[r]&S^{(I)}}.$$
We shall show that $j:A\to G$ belongs to the ideal $\mathfrak{g}_{\mathcal{S}}$. This is equivalent to proving, that for any conflation $\xymatrix{A\ar[r]&B_i\ar[r]&S}$, the pushout diagram
$$\xymatrix{A\ar[d]_j\ar[r]&B_i\ar[r]\ar[d]&S\ar@{=}[d]\\
	G\ar[r]&Y\ar[r]&S}$$
is null homotopic. The following commutative diagram of conflations
	$$\xymatrix{A\ar[r]\ar[d]&B_i\ar[r]\ar[d]&S\ar[d]\\
		A^{(I)}\ar[r]\ar[d]&\coprod_{i\in I}B_i\ar[d]\ar[r]&S^{(I)}\ar@{=}[d]\\
	A\ar[r]^j&G\ar[r]&S^{(I)}}$$
gives us the map $B_i \to G$ such that the diagram $$\xymatrix{A\ar[r]\ar[d]_-j&B_i\ar[ld]\\G}$$ commutes, as desired.
	\end{proof}

In addition  to the conditions given in Theorem~\ref{key}, if we assume further that  $\mathcal{A}$ has enough projective objects, then by \cite[Theorem 2]{FGHT}, the pair $(^{\perp}\mathfrak{g}_{\mathcal{S}},\mathfrak{g}_{\mathcal{S}})$ is a \emph{complete ideal cotorsion pair} with $^{\perp}\mathfrak{g}_{\mathcal{S}}$ an object ideal. In order to determine the objects of $^{\perp}\mathfrak{g}_{\mathcal{S}}$, we recall that for two subcategories $\mcB$ and $\mcC$ of $\mcA$, the subcategory consisting of objects $X$ that fit into a conflation $B\to X\to C,$ with $B\in\mcB$ and $C\in\mcC,$ is denoted by $\mcB \star \mcC.$ Applying~\cite[Propositions 6.2 and 7.3]{FH}, we obtain the following.

\begin{proposition}\label{key1} 
Let $\mcS \subseteq \mcA$ be a set of objects in $\mathcal{A}$.  If $\mathcal{A}$ has exact coproducts and enough projective objects,  then $^{\perp}\mathfrak{g}_{\mathcal{S}} = \langle  \add (\mcE\mbox{-}\Proj \star \Sum(\mcS)) \rangle$ is an object ideal.
\end{proposition}

Applying Proposition~\ref{finite1} and Theorem~\ref{key}, we obtain the following result.

\begin{corollary}\label{finite}
let $\mcS \subseteq \mcA$ be a set of objects in $\mcA$. If $\mcA$ has exact coproducts, then the $n$-th power $\mathfrak{g}^n_{\mathcal{S}}$
is an object-special preenveloping ideal with $\mathfrak{g}^n_{\mathcal{S}}$-cosyzygy subcategory $n$-$\Filt(\Sum(\mathcal{S}))$.  Moreover, if $\mathcal{A}$ has enough projectives, then $(^{\perp}\mathfrak{g}^n_{\mathcal{S}},\mathfrak{g}^n_{\mathcal{S}})$ is a complete ideal cotorison pair.
\end{corollary}

\section{Adjustment of splittings}
As we have seen in Proposition \ref{finite1} a finite power of an object special preenveloping ideal is still an object-special preenveloping ideal. In the next few sections, we will push forward, that is, we will define {\rm transfinite inductive powers} of an ideal, and show that under very mild conditions on the exact category $\mcA,$ a transfinite inductive power of an object special preenveloping ideal is still an object-special preenveloping ideal. As in the corollary above, this will immediately apply to the $\mathcal S$-ghost ideal. In order to carry this out, we need to generalize Eklof's Lemma, an important pillar in classical approximation theory, to the present setting. We call this generalization the~\emph{Ideal Eklof Lemma}. We will begin with a very useful technique \emph{the adjustment of splittings,} which will play an important role in the proof of the Ideal Eklof Lemma, and may be of independent interest.

If $(\mcA; \mcE)$ is an exact category, then the morphisms of conflations in $(\mcA; \mcE)$ themselves form an exact category $(\Arr (\mcA); \Arr (\mcE)).$ The notation $\xymatrix@1{x \ar[r]^i & y \ar[r]^p & z}$ is used to represent the conflation of arrows given by
$$\xymatrix{X_0 \ar[r]^{i_0} \ar[d]^x & Y_0 \ar[r]^{p_0} \ar[d]^y & Z_0 \ar[d]^z \\
	X_1 \ar[r]^{i_1} & Y_1 \ar[r]^{p_1} & Z_1.
}$$
In this section, we consider conditions under which a morphism of trivial conflations in $(\mcA; \mcE)$ is itself trivial when regarded as a conflation in $(\Arr (\mcA); \Arr (\mcE)).$ This involves the introduction of a splitting for each of the conflations of objects and a study of the criteria for when those splittings may be adjusted to yield a splitting of arrows. These ideas are best understood when the conflations are considered as complexes in $\bC (\mcA),$ so that we will follow horizontal indexing conventions in this section, rather than the usual vertical ones.

\subsection{The connecting map} Let us consider conflations in $(\mcA; \mcE)$ as complexes and work in the $\Hom$-complex of maps. Concretely, a conflation $\xymatrix@1{X_1 \ar[r] & X_0 \ar[r] & X_{-1}}$ may be regarded as the object $X.$
$$\xymatrix@1{\cdots \ar[r] & 0 \ar[r] & X_1 \ar[r]^{d^X_1} & X_0 \ar[r]^{d^X_0} & X_{-1} \ar[r] & 0 \ar[r] & \cdots
}$$
in $\bC (\mcA),$ concentrated in degrees $1,$ $0,$ and $-1.$ If $Y.$ is another conflation in $(\mcA; \mcE),$ the complex $\Hom. (X., Y.)$ is defined so that the $n$-chains are maps of degree $n,$ $$\Hom_n (X., Y.) = \prod_i \; \Hom (X_i, Y_{i+n}).$$
Since $X.$ and $Y.$ are conflations, $\Hom_n (X., Y.) = 0$ for $|n| > 2$,  and therefore,  the chain complex   $\Hom. (X., Y.)$ looks like
$$\xymatrix@C=14pt{0 \ar[r] & \Hom_2 (X.,Y.) \ar[r]^{D_2} & \Hom_1 (X.,Y.) \ar[r]^{D_1} & \Hom_0 (X.,Y.) \ar[r]^{D_0} & \Hom_{-1} (X.,Y.) \ar[r]^{D_{-1}} & \Hom_{-2} (X.,Y.) \ar[r] &0.}$$
The differential $D_n$ is defined so that its $i$-th component acts on $\Hom (X_i, Y_{i+n})$ by $$[D_n(f)]_i = d_{i+n}^Y \circ f_i - (-1)^{n} f_{i-1} \circ d_i^X.$$

We are interested in degree $2$ maps $\Delta \in \Hom_2 (X., Y.)$
$$\xymatrix{
		X_1 \ar[r]^{i_X} & X_0 \ar[r]^{p_X} & X_{-1} \ar[dll]_{\grD} \\
		Y_1 \ar[r]_{i_Y} & Y_0 \ar[r]^{p_Y} & Y_{-1},
	}$$
which we call {\em connecting.} The boundary of a connecting map is given by $D_2 (\Delta) = (-\Delta p_X, i_Y \Delta).$

\begin{lemma} \label{L:homotopy of 0} For given conflations $X.$ and $Y.$ in $\mcA$,   the chain  complex $\Hom. (X., Y.)$ is exact at $\Hom_2 (X., Y.)$ and $\Hom_1 (X., Y.).$ In other words, a degree $1$ map $h = (h_0, h_{-1})$ of conflations
	$$\xymatrix@R=40pt@C=35pt{
		X_1 \ar[r]^{i_X} & X_0 \ar[r]^{p_X} \ar[dl]_{h_0} & X_{-1} \ar[dl]_-{h_{-1}} \\
		Y_1 \ar[r]^{i_Y} & Y_0 \ar[r]^{p_Y} & Y_{-1}
	}$$
	is a {\em homotopy of $0,$} i.e., $D_1(h)=0$, if and only if there exists a unique connecting map $\Delta \colon X_{-1} \to Y_1$ such that $h_0 = - \grD p_X$ and $h_{-1} = i_Y \grD,$ i.e., $D_2 (\Delta) = h.$
\end{lemma}

\begin{proof}
	If $\grD \colon X_{-1} \to Y_1$ is a connecting map, then the degree $1$ map $D_2 (\Delta) = (-\grD p_X,  i_Y \grD)$ is a $1$-boundary and therefore a $1$-cycle, i.e., a homotopy of $0.$
Conversely, suppose that $h = (h_0, h_{-1})$ is a homotopy of $0,$ $D_1 (h) = 0.$ Then $h_0 i_X = 0,$ $i_Y h_0 + h_{-1} p_X = 0,$ and $p_Y h_{-1} = 0.$
The first equation implies there is a {\em unique} $\grD \colon X_{-1} \to Y_1$ such that $h_0 = -\grD p_X$ and the third equation implies there is a unique $\grD' \colon X_{-1} \to Y_1$ such that $h_{-1} = i_Y \grD'.$
Substituting these equations into the second yields $$0 = i_Y h_0 + h_{-1} p_X = - i_Y \grD p_X + i_Y \grD' p_X = i_Y (-\grD + \grD') p_X.$$ As $i_Y$ is a kernel and $p_X$ a cokernel, they are a monomorphism
and an epimorphism respectively, whence $\grD' = \grD$ and $$h = (h_0, h_{-1}) =  (-\grD p_X,  i_Y \grD) = D_2 (\Delta).$$
The uniqueness of $\Delta$ ensures exactness of $\Hom. (X., Y.)$ at $\Hom_2 (X., Y.).$
\end{proof}

\subsection{Splittings}
A {\em splitting} $(r,s)$ of a trivial conflation $\xymatrix@1@C=25pt{X_1 \ar[r]^i & X_0 \ar[r]^-p & X_{-1}}$ in $\mcA$ is a degree $1$ map for which $D_1 (r,s) = 1_{X.}.$ In other words, it is a self-homotopy of the identity morphism $1_{X.},$
$$\xymatrix@R=30pt@C=30pt{
	X_1 \ar[r]^i \ar@{=}[d] & X_0 \ar[r]^p \ar@{=}[d] \ar[ld]_r & X_{-1} \ar@{=}[d] \ar[ld]_s\\
	X_1 \ar[r]^i & X_0 \ar[r]^p & X_{-1},
}$$
which consists of a retraction $r: X_0 \to X_1$ of $i,$ $ri=1_{X_1},$ and a section $s: X_{-1} \to X_0$ of $p,$ $ps=1_{X_1},$ satisfying $1_{X_0} = ir + sp.$ The splitting is uniquely determined by the section $r$ or the retration $s$, that is, once a retraction $r$ is given, then there is a unique section $s$ induced by $r$ such that $(r,s)$ is a splitting, and vice versa.

\begin{proposition} \label{P:adj}
	If $(r,s)$ is a splitting of $\xymatrix@1@C=25pt{X_1 \ar[r]^i & X_0 \ar[r]^p & X_{-1},}$ then every splitting of $X.$ is of the form $(r,s) - D_2 (\grD) = (r + \grD p, s - i \grD)$ for some connecting map $\Delta \colon X_{-1} \to X_1.$
\end{proposition}

\begin{proof}
	Suppose that $(r',s') \in \Hom_1 (X., X.)$ is another splitting, $D_1 (r', s') = 1_{X.}.$ Then \linebreak $D_1 (r-r', s-s') = 0$, and therefore, by Lemma~\ref{L:homotopy of 0}, there exists a connecting map $\Delta$ such that $(r-r', s-s') = D_2 (\Delta) = (-p\Delta, \Delta i),$ whence $(r', s') = (r + p\Delta, s - \Delta i).$
\end{proof}

A given splitting $(r,s)$  yields an isomorphism of kernel-cokernel pairs
$$\xymatrix@C=40pt{
	X_{-1} \ar[r]^-{{\mbox{\tiny $\left[ \begin{array}{c} 1 \\ 0 \end{array} \right]$}}} \ar@{=}[d] & X_{-1} \oplus X_1 \ar[r]^-{[0, 1]} \ar[d]^{[s, i]} & X_1 \ar@{=}[d]\\
	X_{-1} \ar[r]^s & X_0 \ar[r]^r & X_1,}$$
because the morphism $[s, i] \colon X_{-1} \oplus X_1 \to X_0$ is an isomorphism with inverse ${\dis \left[ \begin{array}{c} p \\ r \end{array} \right] \colon X_0 \to X_{-1} \oplus X_1.}$
By~\cite[Lemma 2.7]{B}, the kernel-cokernel pair in the bottom row is a trivial conflation with a splitting $(p,i).$

Let $f \colon X. \to Y.$ be a morphism between two trivial conflations in $\mcA$, displayed as a commutative diagram
$$\xymatrix@R=30pt@C=30pt{
	X_1 \ar[r]^{i_X} \ar[d]^{f_1} & X_0 \ar[r]^{p_X}\ar[d]^{f_0} & X_{-1} \ar@<-1ex>[d]^{f_{-1}} \\
      Y_1 \ar[r]^{i_Y} & Y_0 \ar[r]^{p_Y} & Y_{-1}.}$$
Equivalently, $f \in \Hom_0 (X., Y.)$ is a $0$-cycle. Let $(r_X,s_X)$ and $(r_Y,s_Y)$ be splittings of $X.$ and $Y.$ and consider $f$ as a degree $0$ map of reversed conflations
$$\xymatrix@R=30pt@C=30pt{
	X_{-1} \ar[r]^{s_X} \ar@<-1ex>[d]^{f_{-1}} & X_0 \ar[r]^{r_X} \ar[d]^{f_0} & X_1 \ar[d]^{f_1} \\
      Y_{-1} \ar[r]^{s_Y} & Y_0 \ar[r]^{r_Y} & Y_1.}$$
If the diagram is commutative, we say that the splittings $(r_X, s_X)$ and $(r_Y, s_Y)$ are a splitting for the conflation $f$ of arrows. In any case, we may consider its $-1$-boundary $(s_Yf_{-1} - f_0s_X, r_Yf_0 - f_1r_X)$ as a degree $1$ map
$$\xymatrix@R=40pt@C=50pt{
				X_1 \ar[r]^{i_X} & X_0 \ar[r]^{p_X} \ar[dl]_{r_Yf_0 - f_1r_X} & X_{-1} \ar[dl]^{s_Yf_{-1} - f_0s_X} \\
				Y_1 \ar[r]^{i_Y} & Y_0 \ar[r]^{p_Y} & Y_{-1}.
			}$$
Let us show that it is a homotopy of $0.$

\begin{proposition} \label{P:conn map}
Given a morphism $f \colon X. \to Y.$ of trivial conflations, equipped with splittings $(r_X, s_X)$ and $(r_Y, s_Y)$ as above, there is a unique connecting map $\Delta =: \Delta (f)$ such that $$D_2 (\Delta) =  (r_Yf_0 - f_1r_X, s_Yf_{-1} - f_0s_X).$$
\end{proposition}

\begin{proof}
By Lemma~\ref{L:homotopy of 0}, it suffices to verify that $(r_Yf_0 - f_1r_X, s_Yf_{-1} - f_0s_X)$ is a $1$-cycle, that is, 
$D_1 (r_Yf_0 - f_1r_X, s_Yf_{-1} - f_0s_X)=0$. First,
$$(r_Yf_0 - f_1r_X)i_X  = r_Y f_0 i_X - f_1 r_X i_X = r_Y f_0 i_X - f_1 = 0,$$
because $i_Y (r_Y f_0 i_X - f_1) = i_Y r_Y f_0 i_X - i_Y f_1 = (i_Y r_Y - 1_{Y_0}) (i_Y f_1) =  s_Y p_Y (i_Y f_1) = 0$ and $i_Y$ is a monomorphism. Then,
\begin{eqnarray*} i_Y(r_Yf_0 - f_1r_X) + (s_Yf_{-1} - f_0s_X)p_X & = & i_Y r_Y f_0 - i_Y f_1 r_X + s_Y f_{-1} p_X - f_0 s_X p_X \\
 & = & i_Yr_Yf_0 -  f_0 i_X r_X + s_Y p_Y f_0 - f_0 s_X p_X \\
 & = & (i_Yr_Y + s_Y p_Y) f_0 - f_0 (i_X r_X + s_X p_X) = 0.
\end{eqnarray*}
The verification of the third equation $p_Y (s_Yf_{-1} - f_0s_X) = 0$ is similar to the first.
\end{proof}

We call the connecting map obtained in Proposition~\ref{P:conn map} the {\em connecting map of} $f \colon X. \to Y.$ By the proof of Lemma~\ref{L:homotopy of 0}, it is uniquely determined by either of the equations
\begin{equation} \label{Eq:Delta eqns}
-\Delta (f) p_X = r_Yf_0 - f_1r_X \;\; \mbox{  or  }\;\;  i_Y \Delta (f) = s_Yf_{-1} - f_0s_X.
\end{equation}
This can be used to verify the following.

\begin{lemma} \label{L:derivation}
Suppose that $\xymatrix@1{X.\; \ar[r]^f & Y. \ar[r]^g & Z.}$ is a composition of morphisms of trivial conflations,
$$\xymatrix@C=0pt{X.: \ar[d]_f & X_1 \ar[rrrrr]^{i_X} \ar[d]^{f_1} &&&&& X_0 \ar[rrrrr]^{p_X} \ar[d]^{f_0} &&&&& X_{-1} \ar@<-1ex>[d]^{f_{-1}} \\
	Y.:\; \ar[d]_g & Y_1 \ar[rrrrr]^{i_Y} \ar[d]^{g_1} &&&&& Y_0 \ar[rrrrr]^{p_Y} \ar[d]^{g_0} &&&&& Y_{-1} \ar@<-1ex>[d]^{g_{-1}} \\
	Z.:\; & Z_1 \ar[rrrrr]^{i_Z} &&&&& Z_0 \ar[rrrrr]^{p_Z} &&&&& Z_{-1},
}$$
equipped with respective splittings $(r_X, s_X),$ $(r_Y, s_Y),$ and $(r_Z, s_Z).$ Then $\Delta(gf) = \Delta(g) f_{-1} + g_1 \Delta(f).$
\end{lemma}

\subsection{Adjustment of splittings} If $f \colon X. \to Y.$ is a morphism of trivial conflations, equipped with splittings $(r_X, s_X)$ and $(r_Y, s_Y),$ respectively, we know from Proposition~\ref{P:adj} that every splitting of $X.$ is of the form
$(r_X, s_X) - D_2 (\Delta_X)$ for some $\Delta_X \in \Hom_2 (X., X.),$ and likewise for $(r_Y, s_Y).$ The next theorem uses the connecting map $\Delta (f)$ to give a necessary and sufficient condition for $f$ to be a trivial arrow conflation.

\begin{theorem} \label{T:adj}
Let $f \colon X. \to Y.$ be a morphism of trivial conflations, equipped with splittings $(r_X, s_X)$ and $(r_Y, s_Y),$ respectively. The conflation of arrows $\xymatrix{f_1 \ar[r]^i & f_0 \ar[r]^p & f_{-1}}$ is trival in $\operatorname{Arr}(\mathcal{A})$ if and only if there exist $\Delta_X \in \Hom_2 (X., X.)$ and $\Delta_Y \in \Hom_2 (Y., Y.)$ such that $$\Delta (f) = \Delta_Y f_{-1} - f_1 \Delta_X.$$
In that case, the splittings $(r_X, s_X) - D_2 (\grD_X)$ and $(r_Y, s_Y) - D_2 (\grD_Y)$ give a splitting for $f$ in $\Arr(\mcA)$.
\end{theorem}

\begin{proof}
	Suppose the condition holds and let us verify that $f$ is a morphism of the reverse complexes associated to the adjusted splittings, that is, that the diagram
	$$\xymatrix@R=40pt@C=50pt{
		X_{-1} \ar[r]^{s_X - i_X \grD_X} \ar@<-1ex>[d]^{f_{-1}} & X_0 \ar[r]^{r_X + \grD_X p_X} \ar[d]^{f_0} & X_1 \ar[d]^{f_1} \\
		Y_{-1} \ar[r]^{s_Y - i_Y \grD_Y} & Y_0 \ar[r]^{r_Y + \grD_Y p_Y} & Y_1
	}$$
	commutes. We only verify commutativity of the right diagram; verification of the left one is similar. Recall that $r_Y f_0  -  f_1 r_X = -\grD (f) p_X,$ whence
	\begin{eqnarray*}
f_1(r_X + \grD_X p_X) = f_1 r_X + f_1 \grD_X p_X & = & r_Y f_0 + \grD (f) p_X + f_1 \grD_X p_X  \\
& = & r_Y f_0 + \grD_Y f_{-1} p_X  \\
& = & r_Y f_0 + \grD_Y p_Y f_0 = (r_Y + \grD_Y p_Y) f_0.
\end{eqnarray*}
	
	To prove the converse, assume the conflation is trivial. By Proposition~\ref{P:adj}, there exist connection maps $\grD_X \colon X_{-1} \to X_1$ and $\grD_Y \colon Y_{-1} \to Y_1$ such that the splittings
$(r_X + \grD_X p_X, s_X - i_X \grD_X)$ and $(r_Y + \grD_Y p_Y, s_Y - i_Y \grD_Y)$ of $X.$ and $Y.,$ respectively, make the diagram commute. To verify that $\grD (f) = \grD_X f_{-1} - f_1 \grD_X,$ use the commutativity of the right square
$f_1 (r_X + \grD_X p_X) = (r_Y + \grD_Y p_Y) f_0$ to get that
$$-(\grD_Y f_{-1} - f_1 \grD_X)p_X = - \grD_Y (f_{-1}p_X) + f_1 \grD_X p_X = - \grD_Y p_Y f_0 +  f_1 \grD_X p_X = r_Y f_0 - f_1 r_X.$$
	It means that $\grD_Y f_{-1} - f_1 \grD_X$ is a solution to the first of Equations~(\ref{Eq:Delta eqns}) and therefore equals $\grD (f).$
\end{proof}

In practice, Theorem~\ref{T:adj} will be used for \emph{downward adjustment of splittings} of arrows (see Definition~\ref{D:down adjust}). This means that the splitting $(r_X, s_X)$ stays fixed and $(r_Y, s_Y)$ is adjusted. In other words, we are faced with the constraint that $\grD_X = 0$ and the task of finding a morphism
$\grD_Y \colon Y_{-1} \to Y_1$ such that $\grD (f) = \grD_Y f_{-1},$
$$\xymatrix@R=35pt@C=35pt{X_1 \ar[r]^{i_X} \ar[d]^{f_1} & X_0 \ar[r]^{p_0} \ar[d]^{f_0} & X_{-1} \ar@<-1ex>[d]^{f_{-1}} \\
			Y_1 \ar[r]^{i_Y} & Y_0 \ar[r]^{p_Y} & Y_{-1}. \ar@/^1pc/[ll]^{\grD_Y}
		}$$

\begin{corollary}\label{L:FH}
		Let $f \colon X. \to Y.$ be a morphism of trivial conflations in $\mcA$, equipped with respective splittings $(r_X,s_X)$ and $(r_Y,s_Y).$ If the connecting map $\Delta(f): X_{-1} \to Y_1$ extends along $f_{-1}$ to $\grD_Y \colon Y_{-1} \to Y_1,$ $\grD (f) = \grD_Y f_{-1},$ then the splittings $(r_X, s_X)$ and $(r_Y, s_Y) - D_2 (\grD_Y)$  provide a splitting for $f.$
	\end{corollary}

Corollary~\ref{L:FH} allows us to handle a situation that will arise in the sequel in the following way.

\begin{lemma} \label{L:FH1}
	Let $f \colon X. \to Z.$ be a morphism of trivial conflations that admits a factorization $\xymatrix@1{f \colon X. \ar[r]^-{h} & Y. \ar[r]^g & Z.}$ of the form
	$$\xymatrix@C=0pt{X. \ar[d]_h & X_1 \ar[rrrrr]^{i_X} \ar@{=}[d] &&&&& X_0 \ar[rrrrr]^{p_X} \ar[d] &&&&& X_{-1} \ar@<-1ex>[d]^{h_{-1}} \\
		Y. \ar[d]_g & X_1 \ar[rrrrr] \ar[d]^{g_1} &&&&& Y_0 \ar[rrrrr] \ar[d] &&&&& Z_{-1} \ar@{=}@<-1ex>[d] \\
		Z. & Z_1 \ar[rrrrr]^{i_Z} &&&&& Z_0 \ar[rrrrr]^{p_Z} &&&&& Z_{-1},
	}$$
	with $Y.$ a trivial conflation. If $h_{-1}$ is an inflation satisfying $\Ext (\Coker (h_{-1}), g_1) = 0,$ then for every splitting $(r_X, s_X)$ of $X.$, there exists a splitting $(r_Z, s_Z)$ of $Z.$
	that yields a splitting for $f.$
\end{lemma}

\begin{proof}
	Let $(r_X, s_X)$ be a given splitting of $X.$ and fix any splitting $(r_Y,s_Y)$ of $Y.$. Because we have that
	$\Ext (\Coker(h_{-1}), g_1 \Delta(h)) = \Ext (\Coker(h_{-1}), g_1) \Ext (\Coker(h_{-1}), \grD (h)) = 0,$ the morphism $g_1 \grD (h) \colon X_{-1} \to Z_1$ extends along the inflation $h_{-1}$ as in
	$$\xymatrix{X_{-1} \ar[r]^{h_{-1}} \ar[d]_{\grD (h)} & Z_{-1} \ar[r] \ar@{.>}[ldd]^{\grD_Z} & \Coker (h_{-1}) \\
		X_1 \ar[d]_{g_1} \\
		Z_1,}$$
	where the top row is a conflation. Now let $(r_Z, s_Z)$ be an arbitrary splitting of $Z.$ and apply Lemma~\ref{L:derivation} to get that
	$$\Delta(f) = \Delta (gh) = \grD(g)h_{-1} + g_1 \Delta(h) = \grD(g)h_{-1} + \grD_Z  h_{-1} = (\grD(g) + \grD_Z)h_{-1}.$$
	By Corollary~\ref{L:FH}, the splittings $(r_X, s_X)$ and $(r_Z, s_Z) - D_2 (\grD(g) + \grD_Z)$ give a splitting for $f.$
\end{proof}

\section{The Ideal Eklof Lemma}
In this section, we aim at proving The Ideal Eklof Lemma (Theorem \ref{TFH4}), which will provide us with a crucial tool  for proving our main result in Section 6. We start by recalling some necessary terminology and notation.

\subsection{The category of directed $\lambda$-systems}  Let $\lambda$ be an ordinal. We  consider it as a category from its canonical partial order. We let $(\grl, \mcA)$ denote  the category of functors from $\lambda$ to $\mcA$. A functor $\sfA \colon \grl \to \mcA$  is in fact a directed $\grl$-system $\mathsf{A}:= (A_{\alpha},a^{\alpha}_{\beta} \ |\ \alpha<\beta<\lambda)$ which can be displayed as
$$
\xymatrix{
A_0 \ar[r]^{a^0_1} \ar@/_/[rr]_{a^0_2} \ar@/_{9mm}/[rrrr]^{a^0_{\alpha}}  \ar@/_{11mm}/[rrrrr]_{a^0_{\alpha+1}}& A_1 \ar[r]^{a^1_2} \ar @/^{8mm}/[rrr]^{a^1_{\alpha}} & A_2 \ar[r] & \cdots  \ar[r] & A_{\alpha} \ar[r]^-{a^\alpha_{\alpha+1}} & A_{\alpha+1} \ar[r] & \cdots.
}
$$
The \textit{composition} of the $\lambda$-system $\mathsf{A}$ in $\mcA$, if it exists, is the canonical colimit morphism in $\mcA$
$$A_0 \longrightarrow \lim\limits_{\longrightarrow\atop{\alpha<\lambda}}A_{\alpha}.$$

 \begin{definition} \label{D:cont} Let $\grl$ be an ordinal. A $\grl$-system $\sfA \colon \grl \to \mcA$ is said to be {\em continuous at} the limit ordinal $\grb < \grl$ if $\lim\limits_{\longrightarrow\atop{\alpha<\beta}}A_{\alpha}$ exists, and
the canonical morphism $$a_{\grb} \colon \lim\limits_{\longrightarrow\atop{\alpha<\beta}}A_{\alpha} \longrightarrow A_{\beta}$$ is an isomorphism. It is said to be {\em continuous} if it is continuous at every limit ordinal $\grg < \grl.$ The subcategory of continuous $\grl$-systems in $\mcA$ is denoted by $\Cont (\grl, \mcA) \subseteq (\grl, \mcA).$
\end{definition}

Note that a continuous $\grl $-system $\sfA \colon \grl  \to \mcA$ with composition morphism ${\dis A_0 \longrightarrow \lim\limits_{\longrightarrow\atop{\alpha<\lambda}}A_{\alpha}}$ gives rise to a  continuous $(\grl+1) $-system $(A_{\alpha},a^{\alpha}_{\beta} \ | \ \alpha<\beta \leq\lambda)$, where $A_\lambda:=\lim\limits_{\longrightarrow\atop{\alpha<\lambda}}A_{\alpha}$, and $a_\lambda^\alpha:A_\alpha \longrightarrow A_\lambda$ is the canonical colimit morphism  for every $\alpha <\lambda$.
Conversely, any continuous $(\lambda+1)$-system $\mathsf{A}$ in $\mcA$ is of the form  $(A_{\alpha},a^{\alpha}_{\beta} \ | \ \alpha<\beta \leq\lambda)$ with composition morphism $a_\lambda^0:A_0 \longrightarrow A_\lambda$.

For a given  continuous $(\grl +1)$-system $\sfA \colon \grl + 1 \to \mcA$, if the morphisms $a^{\gra}_{\gra + 1}$, $\alpha < \lambda$,  belong to some class $\mcJ$ of morphisms in $\mcA$, the composition morphism $a^0_{\grl}$ is said to be a $\lambda$-\textit{transfinite composition of morphisms in} $\mcJ$, or simply a $\lambda$-\textit{composition of morphisms in} $\mcJ.$ We say that \emph{$\mcJ$ has  $\lambda$-compositions} if every continuous $\lambda$-system in $\mcJ$ has a composition morphism  in $\mcA$.

The following proposition is immediate as (co)limits in functor categories can be calculated object-wise.

\begin{proposition} An exact structure $\mathcal{E}$ on $\mcA$ induces the exact structure $(\grl, \mathcal{E})$ on $(\lambda, \mcA)$ whose conflations are the directed $\grl$-systems of conflations (or just $\grl$-conflations) in $\mcE.$
\end{proposition}

A conflation in $((\grl, \mathcal{A});(\grl, \mathcal{E}))$, denoted by $\mathsf{\Xi} \colon \xymatrix@1{  \mathsf{A} \ar[r]^{\mathsf{i}} & \mathsf{ B} \ar[r]^{\mathsf{p}} & \mathsf{C},}$
may be called a $\lambda$-conflation and is displayed as a $\lambda$-system ${(\ \Xi_{\alpha},  \xi_{\beta}^\alpha \mid \ \alpha <\beta <  \lambda \ )}$ of conflations in $\mcA,$ where $\xi_{\beta}^\alpha:\Xi_{\alpha}\to \Xi_{\beta}$ is a morphism of conflations
$$\xymatrix@C=0pt{\Xi_{\alpha} \ar@<-3ex>[d]_{\xi^{\alpha}_{\beta}}: \;A_{\alpha}\ar[rrrrrr]^-{i_\alpha}\ar@<2ex>[d]^{a^{\alpha}_{ \beta}}&&&&&& B_{\alpha}\ar[rrrrrr]^{p_\alpha}\ar[d]^{b^{\alpha}_{\beta}}&&&&&&C_{\alpha}\ar[d]^{c^{\alpha}_{\beta}}\\
\Xi_{\beta}: \; A_{\beta}\ar[rrrrrr]^-{i_\beta}&&&&&& B_{\beta }\ar[rrrrrr]^{p_\beta}&&&&&&C_{\beta }.}$$
A $\lambda$-conflation $\mathsf{\Xi}$ is \emph{continuous} if its objects belong to $\Cont (\grl, \mcE).$

As in the exact category $\mcA,$ a $\grl$-conflation $\mathsf{\Xi} : \xymatrix@1{  \mathsf{A} \ar[r]^{\mathsf{i}} & \mathsf{ B} \ar[r]^{\mathsf{p}} & \mathsf{C}}$ of  $\grl$-systems is \emph{trivial} provided there exists a splitting
$\mathsf{\Xi}^- : \xymatrix@1{  \mathsf{C} \ar[r]^{\mathsf{s}} & \mathsf{ B} \ar[r]^{\mathsf{r}} & \mathsf{A}}$ such that $\mathsf{r} \mathsf{i} = 1_{\mathsf{A}},$ $\mathsf{p} \mathsf{s} = 1_{\mathsf{C}}$ and
$\mathsf{i} \mathsf{r} + \mathsf{s} \mathsf{p} = 1_{\mathsf{B}}.$ Equivalently, the $\grl$-conflation $\sfXi^-$ consists of a family $\{(r_\alpha, s_\alpha)\}_{\alpha < \lambda}$ of splittings of the $\Xi_\alpha$ in $\mcA $ such that for every $\alpha < \beta <\lambda,$ $((r_\alpha, s_\gra), (r_\grb, s_\grb))$ is a splitting of the conflation $\xi_\beta^\alpha$ in $(\Arr(\mcA); \Arr (\mcE)).$

The following lemma is the limit step in the transfinite induction of the proof of Theorem~\ref{T:DAP}.

\begin{lemma} \label{L:workhorse}
Let $\grl$ be a limit ordinal and $\mathsf{\Xi}  \colon \xymatrix@1{  \mathsf{A} \ar[r]^{\mathsf{i}} & \mathsf{ B} \ar[r]^{\mathsf{p}} & \mathsf{C}}$ be  a $(\grl+1)$-conflation such that $\sfXi|_{\grl} \in (\grl, \mcE)$ is trivial. If $\sfA$ and $\sfC$ are continuous at $\grl,$ then $\sfXi$ is continuous at $\grl$ and trivial. Furthermore, every splitting $(\sfXi|_{\grl})^{-}$ extends to one of $\sfXi,$ that is, $(\sfXi^-)|_{\grl} = (\sfXi|_{\grl})^-.$
\end{lemma}

\begin{proof}
The   $\grl$-systems $\sfA$ and $\sfC$ are continuous at $\grl,$ so that the direct sum $\sfA \oplus \sfC$ is as well. Let $(\sfXi|_{\grl})^- \colon \xymatrix@1{  \mathsf{C}|_{\grl} \ar[r]^{\mathsf{s}} & \mathsf{ B}|_{\grl} \ar[r]^{\mathsf{r}} & \mathsf{A}|_{\grl}}$ be a splitting for $\sfXi|_{\grl}.$ Because $\mathsf{(i|_{\grl})r + s(p|_{\grl}) = 1_{\sfB|_{\grl}}},$ the morphism of $\grl$-conflations
$$\xymatrix@R=40pt@C=0pt{\sfXi|_{\grl} \colon & \sfA|_{\grl} \ar[rrrrrr]^-{\mathsf{i|_{\grl}}} \ar@{=}[d] &&&&&& \sfB|_{\grl} \ar[rrrrrr]^-{\mathsf{p|_{\grl}}} \ar[d]^{\left[ \begin{array}{c} \mathsf{r} \\ \mathsf{p} \end{array} \right]} &&&&&& \sfC|_{\grl} \ar@{=}[d] \\
& \sfA|_{\grl} \ar[rrrrrr]^-{\mathsf{(i_A)|_{\grl}}} &&&&&& \sfA|_{\grl} \oplus \sfC|_{\grl} \ar[rrrrrr]^-{\mathsf{(p_C)|_{\grl}}} &&&&&& \sfC|_{\grl}
}$$
is an isomorphism. Since the $\grl$-conflation in the bottom row has a limit at $\grl,$ so does the one in the top row. By the universal property of the limit, we have a morphism of $\grl$-conflations
$$\xymatrix@C=0pt@R=35pt{\lim\limits_{\longrightarrow\atop{\alpha<\grl}} \Xi_{\gra} \colon \ar[d]^{\xi_{\grl}} & \lim\limits_{\longrightarrow\atop{\alpha<\grl}} A_{\gra} \ar[rrrrrr]^-{\lim (i_{\gra})} \ar[d]^-{a_{\grl}} &&&&&& \lim\limits_{\longrightarrow\atop{\alpha<\grl}} B_{\gra} \ar[rrrrrr]^-{\lim (p_{\gra})} \ar[d]^-{b_{\grl}} &&&&&& \lim\limits_{\longrightarrow\atop{\alpha<\grl}} C_{\gra} \ar[d]^-{c_{\grl}} \\
\;\;\; \Xi_{\grl}\colon & A_{\grl} \ar[rrrrrr]^{i_{\grl}} &&&&&& B_{\grl} \ar[rrrrrr]^{p_{\grl}} &&&&&& C_{\grl},
}$$
where, because $\sfA$ and $\sfC$ are continuous at $\grl,$ the morphisms $a_{\grl}$ and $c_{\grl}$ are isomorphisms. This implies that this morphism of conflations is an isomorphism and that $\sfXi$ is also continuous at $\grl.$

Finally, to extend the splitting $(\sfXi|_{\grl})^-$ to include $\grl$ just take the limit at $\grl$ of the isomorphism of $\grl$-conflations
$$\xymatrix@R=40pt@C=0pt{(\sfXi|_{\grl})^- \colon & \sfC|_{\grl} \ar[rrrrrr]^-{\mathsf{s}} \ar@{=}[d] &&&&&& \sfB|_{\grl} \ar[rrrrrr]^-{\mathsf{r}} \ar[d]^{\left[ \begin{array}{c} \mathsf{p}|_{\grl} \\ \mathsf{r}|_{\grl} \end{array} \right]} &&&&&& \sfA|_{\grl} \ar@{=}[d] \\
& \sfC|_{\grl} \ar[rrrrrr]^-{\mathsf{(i_C)|_{\grl}}} &&&&&& \sfC|_{\grl} \oplus \sfA|_{\grl} \ar[rrrrrr]^-{\mathsf{(p_A)|_{\grl}}} &&&&&& \sfA|_{\grl}.
}$$
\end{proof}

\subsection{Transfinite powers of an ideal} In order to state the Ideal Eklof Lemma, we need to introduce (transfinite) inductive powers of an ideal $\mcJ$ in $\mcA$. Recall that if $n$ is a positive integer, the $n$-th power $\mcJ^n$ of $\mcJ$ is the ideal generated by all compositions $a_na_{n-1}\cdots a_1$ with $a_i \in \mcJ$ for every $i=1,\ldots, n$. In other words, it is generated by  all $n$-compositions of morphisms in $\mcJ$. In a similar way, for any ordinal $\lambda>0$, we define the \textit{$\lambda$-th inductive power of} $\mcJ$, denoted by  $\mcJ^{(\lambda)}$, as the ideal generated by all $\lambda$-compositions of  morphisms in $\mcJ$.

In the following result, we provide  a characterization of morphisms in $\mcJ^{(\lambda)}$. We will use the fact that if $\sfA$ and $\mathsf{A'}$ are continuous $(\grl + 1)$-systems, then so is $\mathsf{A} \oplus \mathsf{A'}.$

\begin{proposition}\label{prop:morph_induct}
A morphism $g$ in $\mcA$ belongs to  $\mcJ^{(\lambda)}$ if and only if $g$ is a composition of the form $ha_\lambda^0$, where $a_\lambda^0:A_0 \longrightarrow A_\lambda$ is a $\lambda$-composition of morphisms in $\mcJ$.
\end{proposition}

\begin{proof}
Let $\mathcal{T}_{\grl}$ denote the class of all morphisms in $\mcA$ of the form $ha_\lambda^0$, where $a_\lambda^0$ is a $\lambda$-composition of morphisms in $\mcJ$. By the definition of an ideal, $\mathcal{T}_{\grl} \subseteq \mcJ^{(\lambda)}$, so it is enough to prove that $\mathcal{T}_{\grl}$ is an ideal in $\mcA.$ It is obviously closed under left multiplication, but it is also closed under right multiplication. Indeed, if $ha^0_{\grl}  \in \mcT_{\grl}$ with $a^0_{\grl}$ a $\grl$-composition of morphisms in $\mcJ,$ then for any morphism $t,$ $ha^0_{\grl}t = ha^1_{\grl}(a^0_1t).$ As $a^0_1t \in \mcJ,$ this morphism also belongs to $\mcT_{\grl}.$

Now, it is left to prove that $\mathcal{T}_{\grl}$ is additive, that is, if  $ha_\lambda^0, h' {a'}_\lambda^0: A_0 \longrightarrow B$ belong to    $\mathcal{T}_{\grl}$, then so does  $h a_\lambda^0+h'{a'}_\lambda^0:A_0 \longrightarrow B$. Let $\mathsf{A}:=(A_\alpha, a_{\beta}^\alpha \mid\ \alpha <\beta \leq \lambda)$ and $\mathsf{A}':=(A'_\alpha, {a'}_{\beta}^\alpha \mid\ \alpha <\beta \leq \lambda)$ be two $(\lambda+1)$-sequences in $\mcA$ with the composition morphisms $a_\lambda^0$ and ${a'}_\lambda^0$, respectively,  satisfying    $a_{\alpha+1}^\alpha, {a'}_{\alpha+1}^\alpha \in \mcJ$ for every $\alpha < \lambda$. Since both the colimit functor and $\mcJ$ are  additive, the morphism $[a_\lambda^0, {a'}_\lambda^0]:A_0 \oplus A_0 \longrightarrow A_\lambda \oplus A'_\lambda$ is a $(\lambda+1)$-composition of morphisms in $\mcJ$
of the $(\lambda+1)$-sequence $\mathsf{A}\oplus \mathsf{A}'$ defined by
$$(A_\alpha \oplus A'_\alpha, [a_\beta^\alpha, {a'}_\beta^\alpha] \mid\ \alpha< \beta \leq \lambda).$$
  Consider the (co)diagonal morphisms $\varsigma:A_0 \longrightarrow A_0\oplus A_0$ and $\rho: B\oplus B \longrightarrow B$. By the previous argument, the composition $[a_{\lambda}^0, {a'}_{\lambda}^0] \varsigma$
$$\left[ \begin{array}{c} a_\lambda^0 \\ {a'}_\lambda^0 \end{array} \right]=[a_{\lambda}^0, {a'}_{\lambda}^0] \varsigma :A_0 \longrightarrow A_\lambda \oplus A'_\lambda  $$
is a $(\lambda+1)$-composition of morphisms in $\mcJ$. Finally, the morphism
$$   ha_\lambda^0+h{a'}_\lambda^0= \rho [h,h'] \left[ \begin{array}{c} a_\lambda^0 \\ {a'}_\lambda^0 \end{array} \right]$$
belongs to  $\mathcal{T}_{\grl}$.
\end{proof}

\begin{proposition} \label{P:prod of powers}
Given an ideal $\mcJ$ and ordinals $\gra$ and $\grb,$ $\mcJ^{(\gra)}\mcJ^{(\grb)} = \mcJ^{(\grb + \gra)}.$
\end{proposition}

\begin{proof}
To prove $ \mcJ^{(\gra)}\mcJ^{(\grb)} \subseteq \mcJ^{(\grb + \gra)}$, note that a composition $ha^0_{\gra} g b^0_{\grb} = h a^1_{\gra} (a^0_1 g) b^0_{\grb}$ is a $(\grb + \gra)$-composition of morphisms in $\mcJ$ post-composed with $h.$ For the converse inclusion, just factor \lb $ha^0_{\grb + \gra} = h a^{\grb}_{\grb + \gra} a^0_{\grb}.$
\end{proof}

The morphisms in the $\lambda$-th inductive power of a successor ordinal are even simpler.

\begin{proposition}\label{prop:suc_ordinal}
If $\lambda = \grg + 1$ is a successor ordinal, then $\mcJ^{(\lambda)} = \mcJ \mcJ^{(\grg)}$ consists of $\lambda$-compositions of morphisms in $\mcJ$.
 \end{proposition}

 \begin{proof}
A composition $ha^0_{\grl} = (ha^{\grg}_{\grg + 1}) a^0_{\grg}$ is a $(\grg + 1)$-composition of morphisms in $\mcJ.$
\end{proof}

In particular, we have the following.

\begin{proposition}
If $n < \gro,$ then $\mcJ^n=\mcJ^{(n)}$ consists of $n$-compositions of morphisms in $\mcJ.$
\end{proposition}

\subsection{From local triviality to global triviality} Every trivial conflation of $\grl$-systems is evidently a $\grl$-system of trivial conflations. In this subsection, we present a condition that ensures a $\grl$-system of trivial conflations is a trivial conflation of $\grl$-systems. Here is one of the main ideas of the paper.

\begin{definition} \label{D:down adjust}
Let $\grl$ be a limit ordinal. A $\grl$-system $\sfXi : \xymatrix@1{  \mathsf{A} \ar[r]^{\mathsf{i}} & \mathsf{ B} \ar[r]^{\mathsf{p}} & \mathsf{C}}$ of trivial conflations is said to have the {\em downward adjustment property} if for every $\gra < \grl,$
and every splitting $(r_{\gra}, s_{\gra})$ of $\Xi_{\gra},$ there exists a splitting $(r_{\gra +1}, s_{\gra + 1})$ of $\Xi_{\gra +1}$ that provides a splitting for the arrow conflation
$$\xymatrix@C=30pt{\xi^{\gra}_{\gra+1} \colon a^{\gra}_{\gra+1} \ar[r] & b^{\gra}_{\gra+1} \ar[r] & c^{\gra}_{\gra +1}.}$$
\end{definition}

\begin{theorem} \label{T:DAP} {\rm (From local triviality to global triviality)}
Let $\grl$ be a limit ordinal and suppose that $\sfXi : \xymatrix@1{  \mathsf{A} \ar[r]^{\mathsf{i}} & \mathsf{ B} \ar[r]^{\mathsf{p}} & \mathsf{C}}$ is a $\grl$-system of trivial conflations with the downward adjustment property.
If $\mathsf{A}$ and $\mathsf{C}$ are continuous, then $\sfXi$ is trivial and continuous.
\end{theorem}

\begin{proof}
Let us define a splitting $\sfXi^-$ of $\sfXi$ by recursion on the ordinals $\gra < \grl$ so that $(\sfXi^-)|_{\gra}$ is a splitting for $\sfXi|_{\gra}.$ Because $\Xi_0$ is a trivial conflation in $\mcA,$ define $(r_0, s_0)$ to be any splitting of $\Xi_0.$ For a successor ordinal $\gra = \grb + 1,$ assume that the system $\sfXi^-$ has been defined for $\grg \leq \grb,$ with $\Xi^-_{\grb} \colon \xymatrix{C_{\grb} \ar[r]^{s_{\grb}} & B_{\grb} \ar[r]^{r_{\grb}} & A_{\grb}.}$ By the downward adjustment property of $\sfXi,$ there is a splitting $(r_{\gra}, s_{\gra})$ such that the two splittings provide one for the arrow conflation $\xi^{\grb}_{\gra} \colon \xymatrix{a^{\grb}_{\gra} \ar[r] & b^{\grb}_{\gra} \ar[r] & c^{\grb}_{\gra}}.$ Define $\Xi^-_{\gra}$ to be the conflation $\xymatrix{C_{\gra} \ar[r]^{s_{\gra}} & B_{\gra} \ar[r]^{r_{\gra}} & A_{\gra}.}$

For the case when $\gra$ is a limit ordinal, we use Lemma~\ref{L:workhorse}. We have up until now defined as splitting $(\sfXi|_{\gra})^-$ for $\sfXi|_{\gra}.$  Because $\sfA$ and $\sfC$ are continuous at $\gra,$ the lemma applies to the
$(\gra+1)$-conflation $\sfXi|_{\gra +1}$ and we can extend the splitting $(\sfXi|_{\gra})^-$ to one of $\sfXi|_{\gra + 1}.$
\end{proof}

\subsection{Eklof's Lemma} Let us use the results of the previous subsection to give a proof of a general version of Eklof's Lemma (Corollary~\ref{Eklof_object}) in an exact category.

\begin{definition} \label{D:filt}
A continuous $\grl$-system $\sfA$ is said to be  a $\grl$-{\em filtration} if $A_0 = 0$ and all of the successive structural morphisms $a^{\gra}_{\gra+1},$ $\gra + 1 < \grl,$ are inflations.
If all of the cokernels $\Coker (a^{\gra}_{\gra+1})$ belong to some subcategory $\mcC \subseteq \mcA,$ then $\sfA$ is called a $\grl$-$\mcC$-\emph{filtration}.
 An object $C \in \mcA$ is said to \emph{have a} $\grl$-\emph{filtration in} $\mcC$ if there exists a $(\grl+1)$-$\mcC$-filtration $\sfC$ with $C_{\grl} = C.$ The subcategory of objects having a $\grl$-$\mcC$-filtration  is denoted by $\grl \mbox{-}\Filt ( \mcC).$ We let $\Filt ( \mcC)$ denote the subcategory of objects having a $\grl$-$\mcC$-filtration  for some ordinal $\lambda$.
\end{definition}

\begin{theorem}\label{TFH1}
	Let $\lambda$ be a limit ordinal and suppose that $\sfXi : \xymatrix@1{  \mathsf{A} \ar[r]^{\mathsf{i}} & \mathsf{ B} \ar[r]^{\mathsf{p}} & \mathsf{C}}$ is a $\grl$-system of trivial conflations with $\mathsf{A}$ continuous and $\mathsf{C}$ a $\lambda$-filtration in $\mcA.$ If for every ordinal $\alpha$ with $\alpha+1 <\lambda$, $\Hom(c^{\alpha}_{\alpha+1},A_{\alpha+1})$ is an epimorphism, then $\sfXi$ has the downward adjustment property.
\end{theorem}

\begin{proof}
Consider the arrow conflation obtained from the successive structural morphisms of the three $\grl$-systems,
$$\xymatrix@C=0pt{\Xi_{\gra} \ar@<-3ex>[d]_{\xi^{\gra}_{\gra +1}}: \;A_{\gra} \ar[rrrrrr]^-{i_\gra} \ar@<2ex>[d]^{a^{\gra}_{ \gra +1}} &&&&&& B_{\gra} \ar[rrrrrr]^{p_\gra} \ar[d]^{b^{\gra}_{\gra+1}} &&&&&& C_{\gra}  \ar[d]^{c^{\gra}_{\gra+1}}\\
\Xi_{\gra+1}: \; A_{\gra+1} \ar[rrrrrr]^-{i_{\gra+1}}  &&&&&& B_{\gra+1}  \ar[rrrrrr]^{p_{\gra+1}}  &&&&&&  C_{\gra+1},}$$
and choose splittings $(r_{\gra}, s_{\gra})$ and $(r_{\gra +1}, s_{\gra + 1})$ of the respective trivial conflations $\Xi_{\gra}$ and $\Xi_{\gra +1}.$ By hypothesis, the induced connecting map $\Delta (\xi^{\gra}_{\gra + 1}) \colon C_{\gra} \to A_{\gra+1}$ may be extended along $c^{\gra}_{\gra+1}$ as shown,
$$\xymatrix{C_{\gra} \ar[r]^{c^{\gra}_{\gra+1}} \ar[d]_{\Delta(\xi^{\gra}_{\gra+1})} &  C_{\gra+1} \ar[r] \ar[ld] & C_{\gra+1}/C_{\gra}  \\
	A_{\gra+1}.  &&}$$
Corollary~\ref{L:FH} may therefore by employed to adjust the given splitting $(r_{\gra+1},s_{\gra+1})$ to a splitting $(r'_{\gra+1},s'_{\gra+1})$ such that $(r_{\gra},s_{\gra})$ and $(r'_{\gra+1},s'_{\gra+1})$ gives a splitting of the arrow conflation $\xi^{\gra}_{\gra+1}.$
\end{proof}

\begin{corollary}\label{Eklof_object} {\rm (Eklof's Lemma)}
	Let $\mcM$  be a class of objects in $\mcA$.  If an object $C$ in $\mcA$ has a $ {^{\perp}}\mcM$-filtration, then $C \in  {^{\perp}}\mcM.$
\end{corollary}
\begin{proof}
 Let $\sfC $ be a $(\lambda+1)$-$ {^{\perp}}\mcM$-filtration  with $C_{\lambda}=C$.  We shall show, by induction on $\lambda,$ that any conflation $\Xi: A \longrightarrow B \longrightarrow C$ in $\Ext(C,A)$ with  $A \in \mcM$ is trivial. If $\lambda = 0,$ then $C = C_0 = 0.$ If $\lambda = \grg + 1$ is a successor, then there is a conflation $\xymatrix@1{C_{\grg} \ar[r] & C \ar[r] & C/C_{\grg}}.$ By assumption, $C_{\grg}$ and $C/C_{\grg}$ belong to ${^{\perp}}\mcM,$ and hence $C$ belongs to ${^{\perp}}\mcM.$

If $\lambda$ is a limit ordinal, take the pullback of $\Xi$ along $\sfC$ to obtain the $(\grl+1)$-conflation $\sfXi,$
$$\xymatrix@C=0pt@R=30pt{
		\Xi_{0} \colon \ar@<-1ex>[d]_-{\xi_1^0} & A \ar[rrrrrr] \ar@{=}[d] &&&&&& B_{0}\ar[rrrrrr] \ar[d]^-{ b_\beta^\alpha} &&&&&& C_{0}\ar[d]^-{c_1 ^0}\\
		\Xi_{1} \colon \ar@<-1ex>[d]_-{\xi_2^1} & A \ar[rrrrrr] \ar@{=}[d] &&&&&& B_{1}\ar[rrrrrr] \ar[d]^-{ b_2^1} &&&&&& C_{\alpha}\ar[d]^-{c_2^1}\\
		\vdots \;\;\; \ar@<-1ex>[d] & \vdots\ar@{=}[d] &&&&&& \vdots \ar[d] &&&&&&  \vdots\ar[d] \\
		\Xi \colon & A \ar[rrrrrr] &&&&&& B \ar[rrrrrr] &&&&&& C,}$$
where $\Xi_{\grl} = \Xi.$ By the induction hypothesis, $\sfXi|_{\grl}$ is a $\grl$-system of trivial conflations. We will apply Theorem~\ref{TFH1} to show that $\sfXi|_{\grl}$ has the downward adjustment property. Then Theorem~\ref{T:DAP} will apply to show that $\sfXi|_{\grl}$ is trivial. Because both of the outside $(\grl+1)$-systems are continuous at $\grl,$ Lemma~\ref{L:workhorse} will apply to show that $\sfXi$ and therefore $\Xi = \Xi_{\grl}$ is trivial.

To apply Theorem~\ref{TFH1}, it suffices to verify that $\Hom (c^{\gra}_{\gra+1}, A) \colon \Hom (C_{\gra+1},A) \to \Hom (C_{\gra}, A)$ is epic for very $\alpha < \lambda$. But the hypothesis $\Ext (C_{\gra+1}/C_{\gra}, A) = 0$ ensures that every morphism $f \colon C_{\gra} \to A$ extends as in$$\xymatrix@R=30pt@C=30pt{C_{\gra} \ar[r]^-{c^{\gra}_{\gra+1}} \ar[d]_f &  C_{\gra+1} \ar[r] \ar@{.>}[ld] & C_{\gra+1}/C_{\gra}  \\
	A.  &&}$$
\end{proof}

\subsection{The Ideal Eklof Lemma}  Now we are in a position to state and prove the Ideal Eklof Lemma.

 \begin{theorem}\label{TFH4} {\rm (The Ideal Eklof Lemma)}
 	Let $\mcJ$ be an ideal of $\mcA.$ If an object $C$ in $\mcA$ has a $\grl$-filtration in $\operatorname{Ob}({^{\perp}}\mcJ),$ then $C \in {^{\perp}(\mcJ^{(\lambda)})}.$
 \end{theorem}

\begin{proof} Let $\sfC$ be  an  $(\grl+1)$-$\mbox{Ob}({^{\perp}\mcJ})$-filtration  with $C=C_{\lambda},$ and $\Ext(C_{\gra + 1}/C_{\gra},j)=0$ for all morphisms $j$ in $\mcJ$ and ordinals $\gra < \grl.$ We prove that $C \in {^{\perp}}(\mcJ^{(\lambda)})$ by induction on $\lambda$. If $\lambda=\gamma+1$ is a successor ordinal, then there is a conflation
$$\xymatrix@1{ C_{\gamma} \ar[r]^{c^{\grg}_{\grl}} &  C \ar[r] & C/C_{\grg} ,}$$
where $C_{\gamma}$ has a $\grg$-filtration in $\mbox{Ob}({^{\perp}}\mcJ)$ and $C/C_{\grg} \in {^{\perp}\mcJ}$. By the induction hypothesis and Lemma~\ref{finiteghost}, $C \in ({^{\perp}}\mcJ)  ({^{\perp}}\mcJ)^{(\grg)} =   ({^{\perp}}\mcJ)^{(\grg + 1)}.$
	
Now we suppose that $\lambda$ is a limit ordinal and that the result holds for all ordinals less than $\lambda$. Let $j:J_0\to J_{\lambda}$ be a morphism in $\mcJ^{(\lambda)}$. We need to prove that $\Ext(C,j):\Ext(C,J_0)\to\Ext(C,J_{\lambda})$ is the zero map. Because  $j$ is a composition of the form $hj_\lambda^0$, where $j_\lambda^0:J_0 \longrightarrow J_\lambda$ is a $\lambda$-composition of morphisms in $\mcJ$ by Proposition \ref{prop:morph_induct}, we may just assume that $j=j_\lambda^0$. Note that to prove $\Ext(C,j)=0$ is equivalent to showing that for every conflation $J_0\to X\to C$ in $\Ext(C,J_0),$ the conflation $\Xi_{\lambda}:J_{\lambda}\to X_{\lambda}\to C$ that appears in the following pushout diagram is trivial
	$$\xymatrix{J_0\ar[r]\ar[d]_j&X\ar[r]\ar[d]&C\ar@{=}[d]\\
	J_{\lambda}\ar[r]&X_{\lambda}\ar[r]&C.}$$

By taking the pushout of $J_0\to X\to C$ along the $(\lambda+1)$-system $(J_{\alpha},j^{\alpha}_{\beta},\alpha<\beta\leq\lambda)$, and taking the pullback along the filtration $0=C_0\to C_1\to\cdots\to C_{\lambda}$, we obtain the following cubical diagram with all rows conflations.
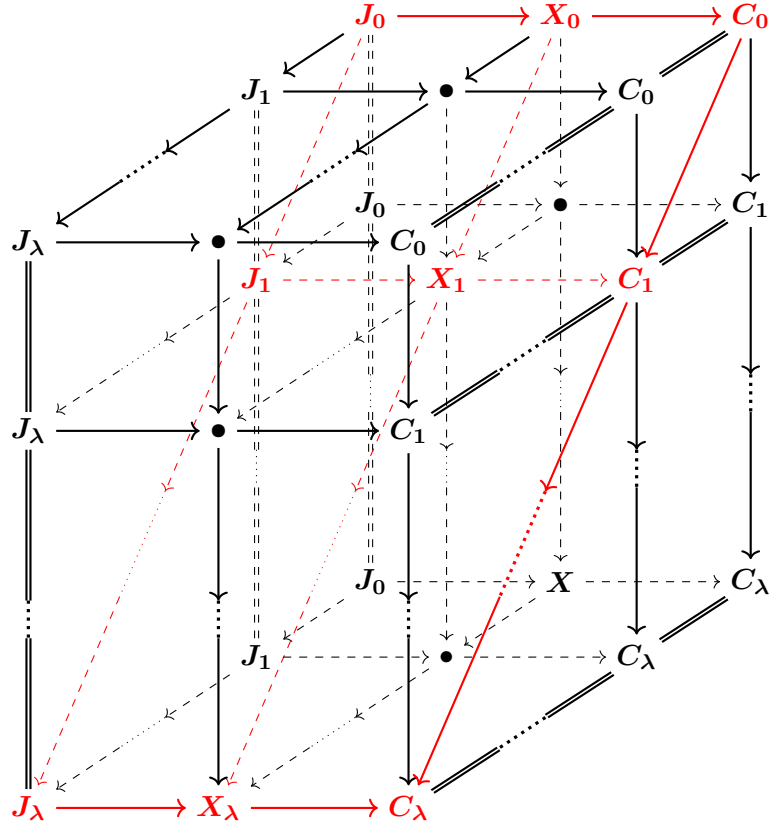
\begin{center}
	\begin{tikzpicture}
		\node (firstJ_0)      at (4.5, 10.5) {\textcolor{red}{$\bm{J_0}$}};  \node (secondJ_0)      at (4.5, 8) {$\bm{J_0}$};                        \node (lastJ_0)        at (4.5, 3) {$\bm{J_0}$};
		\node (firstJ_1)      at (3, 9.5)    {$\bm{J_1}$};                   \node (secondJ_1)      at (3, 7)   {\textcolor{red}{$\bm{J_1}$}};       \node (lastJ_1)        at (3, 2)   {$\bm{J_1}$};
		\node (firstJ_lambda) at (0, 7.5)    {$\bm{J_{\lambda}}$};           \node (secondJ_lambda) at (0, 5)   {$\bm{J_{\lambda}}$};           \node (lastJ_lambda)   at (0, 0)   {\textcolor{red}{$\bm{J_{\lambda}}$}};
		\node (X_0)           at (7, 10.5)   {\textcolor{red}{$\bm{X_0}$}};  \node (secondX)        at (7, 8)   {$\bm{\bullet}$};                          \node (lastX)          at (7, 3)   {$\bm{X}$};
		\node (Y)             at (5.5, 9.5)  {$\bm{\bullet}$};                     \node (secondX_1)      at (5.5, 7) {\textcolor{red}{$\bm{X_1}$}};       \node (lastBullet)     at (5.5, 2) {$\bullet$};
		\node (firstBullet)   at (2.5, 7.5)  {$\bm{\bullet}$};               \node (secondBullet)   at (2.5, 5) {$\bm{\bullet}$};               \node (X_lambda)       at (2.5, 0) {\textcolor{red}{$\bm{X_{\lambda}}$}};
		\node (C)             at (9.5, 10.5) {\textcolor{red}{$\bm{C_0}$}};    \node (firstC_1)       at (9.5, 8) {$\bm{C_1}$};                   \node (firstC_lambda)  at (9.5, 3) {$\bm{C_{\lambda}}$};
		\node (C_0)           at (8, 9.5)    {$\bm{C_0}$};                   \node (secondC_1)      at (8, 7)   {\textcolor{red}{$\bm{C_1}$}};  \node (secondC_lambda) at (8, 2)   {$\bm{C_{\lambda}}$};
		\node (lastC_0)       at (5, 7.5)    {$\bm{C_0}$};                   \node (lastC_1)        at (5, 5)   {$\bm{C_1}$};                   \node (lastC_lambda)   at (5, 0)   {\textcolor{red}{$\bm{C_{\lambda}}$}};
		
		\draw[->, thick, red] (firstJ_0) -- (X_0);     \draw[->, thick, red] (X_0) -- (C);             
		\draw[->, dashed] (secondJ_0)    -- (secondX); \draw[->, dashed] (secondX) -- (firstC_1);      
		\draw[->, dashed] (lastJ_0)      -- (lastX);   \draw[->, dashed] (lastX)   -- (firstC_lambda); 
		\draw[dashed]     (4.5-0.025, 10.25) -- (4.5-0.025, 8.25) [xshift=1.5pt] (4.5-0.025, 10.25) -- (4.5-0.025, 8.25); \draw[dashed]     (4.5-0.025, 7.75) -- (4.5-0.025, 5.75) [xshift=1.5pt] (4.5-0.025, 7.75) -- (4.5-0.025, 5.75); \draw[dotted] (4.5, 5.75)             -- (4.5, 5.25);  \draw[dashed]     (4.5-0.025, 5.25) -- (4.5-0.025, 3.25) [xshift=1.5pt] (4.5-0.025, 5.25) -- (4.5-0.025, 3.25); 
		\draw[->, dashed] (X_0)              -- (secondX);                                                                \draw[->, dashed] (secondX)         -- (7, 5.75);                                                               \draw[dotted] (7, 5.75)               -- (7, 5.25);    \draw[->, dashed] (7, 5.25)         -- (lastX);                                                                 
		\draw[->, thick]  (C)                -- (firstC_1);                                                               \draw[->, thick] (firstC_1)         -- (9.5, 5.75);                                                             \draw[very thick, dotted] (9.5, 5.75) -- (9.5, 5.25);  \draw[->, thick]  (9.5, 5.25)       -- (firstC_lambda);                                                         
		\draw[->, thick]  (firstJ_0)  -- (firstJ_1);   \draw[->, thick]  (firstJ_1)  -- (1.5+0.3, 8.5+0.2); \draw[very thick, dotted] (1.5+0.3, 8.5+0.2) -- (1.5-0.3, 8.5-0.2); \draw[->, thick] (1.5-0.3, 8.5-0.2) -- (firstJ_lambda); 
		\draw[->, dashed] (secondJ_0) -- (secondJ_1);  \draw[->, dashed] (secondJ_1) -- (1.5+0.3, 6+0.2);   \draw[dotted] (1.5+0.3, 6+0.2)               -- (1.5-0.3, 6-0.2);   \draw[->, dashed] (1.5-0.3, 6-0.2)  -- (secondJ_lambda);
		\draw[->, dashed] (lastJ_0)   -- (lastJ_1);    \draw[->, dashed] (lastJ_1)   -- (1.5+0.3, 1+0.2);   \draw[dotted] (1.5+0.3, 1+0.2)               -- (1.5-0.3, 1-0.2);   \draw[->, dashed] (1.5-0.3, 1-0.2)  -- (lastJ_lambda);  
		\draw[dashed] (3-0.025, 9.25) -- (3-0.025, 7.25) [xshift=1.5pt] (3-0.025, 9.25) -- (3-0.025, 7.25);   \draw[dashed] (3-0.025, 6.75) -- (3-0.025, 4.75) [xshift=1.5pt] (3-0.025, 6.75) -- (3-0.025, 4.75); \draw[dotted]             (3, 4.75) -- (3, 4.25); \draw[dashed] (3-0.025, 4.25) -- (3-0.025, 2.25) [xshift=1.5pt] (3-0.025, 4.25) -- (3-0.025, 2.25); 
		\draw[thick]  (-0.025, 7.25)  -- (-0.025, 5.25)  [xshift=1.5pt] (-0.025, 7.25)  -- (-0.025, 5.25);    \draw[thick]  (-0.025, 4.75)  -- (-0.025, 2.75)  [xshift=1.5pt] (-0.025, 4.75)  -- (-0.025, 2.75);  \draw[very thick, dotted] (0, 2.75) -- (0, 2.25); \draw[thick]  (-0.025, 2.25)  -- (-0.025, 0.25)  [xshift=1.5pt] (-0.025, 2.25)  -- (-0.025, 0.25);  
		\draw[->, dashed, red] (firstJ_0) -- (secondJ_1); \draw[->, dashed, red] (secondJ_1) -- (1.5+0.3, 4+0.2); \draw[dotted, red] (1.5+0.3, 4+0.2) -- (1.5-0.3, 2.8); \draw[->, dashed, red] (1.5-0.3, 2.8) -- (lastJ_lambda); 
		
		\draw[->, dashed]      (Y)       -- (secondX_1);  \draw[->, dashed] (secondX_1)  -- (5.5, 4.75);    \draw[dotted] (5.5, 4.75)    -- (5.5, 4.25);    \draw[->, dashed] (5.5, 4.25)    -- (lastBullet);
		\draw[->, dashed]      (secondX) --(secondX_1);   \draw[->, dashed] (secondX_1)  -- (4+0.3, 6+0.2); \draw[dotted] (4+0.3, 6+0.2) -- (4-0.3, 6-0.2); \draw[->, dashed] (4-0.3, 6-0.2) -- (secondBullet);
		\draw[->, dashed]      (lastX)   -- (lastBullet); \draw[->, dashed] (lastBullet) -- (4+0.3, 1+0.2); \draw[dotted] (4+0.3, 1+0.2) -- (4-0.3, 1-0.2); \draw[->, dashed] (4-0.3, 1-0.2) -- (X_lambda);
		\draw[->, dashed, red] (X_0)     -- (secondX_1); \draw[->, dashed, red] (secondX_1) -- (4+0.3, 4+0.2); \draw[dotted, red] (4+0.3, 4+0.2) -- (4-0.3, 2.8); \draw[->, dashed, red] (4-0.3, 2.8) -- (X_lambda);
		
		\draw[->, dashed, red] (secondJ_1) -- (secondX_1); \draw[->, dashed, red] (secondX_1) -- (secondC_1);
		\draw[->, dashed] (lastJ_1) -- (lastBullet); \draw[->, dashed] (lastBullet) -- (secondC_lambda);
		
		\draw[->, thick]      (firstJ_1)       -- (Y);            \draw[->, thick]      (Y)            -- (C_0);
		\draw[->, thick]      (firstJ_lambda)  -- (firstBullet);  \draw[->, thick]      (firstBullet)  -- (lastC_0);
		\draw[->, thick]      (secondJ_lambda) -- (secondBullet); \draw[->, thick]      (secondBullet) -- (lastC_1);
		\draw[->, thick, red] (lastJ_lambda)   -- (X_lambda);     \draw[->, thick, red] (X_lambda)     -- (lastC_lambda);
		\draw[->, thick] (X_0)      -- (Y);        \draw[->, thick] (Y)        -- (4+0.3, 8.5+0.2);   \draw[very thick, dotted] (4+0.3, 8.5+0.2)   -- (4-0.3, 8.5-0.2);   \draw[->, thick] (4-0.3, 8.5-0.2)   -- (firstBullet);
		\draw[->, thick] (firstBullet)  -- (secondBullet);                                                    \draw[->, thick] (secondBullet) -- (2.5, 2.75);                                                    \draw[very thick, dotted] (2.5, 2.75) -- (2.5, 2.25); \draw[->, thick] (2.5, 2.25)    -- (X_lambda);
		\draw[->, thick] (lastC_0)      -- (lastC_1);                                                         \draw[->, thick] (lastC_1)      -- (5, 2.75);                                                      \draw[very thick, dotted] (5, 2.75)   -- (5, 2.25);   \draw[->, thick] (5, 2.25)      -- (lastC_lambda);
		\draw[->, thick] (C_0)          -- (secondC_1);                                                       \draw[->, thick] (secondC_1)    -- (8, 4.75);                                                      \draw[very thick, dotted] (8, 4.75)   -- (8, 4.25);   \draw[->, thick] (8, 4.25)      -- (secondC_lambda);
		
		\draw[->, thick, red] (C) -- (secondC_1); \draw[->, thick, red] (secondC_1) -- (6.5+0.3, 4+0.2); \draw[very thick, dotted, red] (6.5+0.3, 4+0.2) -- (6.5-0.3, 2.8); \draw[->, thick, red] (6.5-0.3, 2.8) -- (lastC_lambda); 
		\draw[thick] (9.5-0.335, 10.5-0.2) -- (8+0.3, 9.5+0.235); \draw[thick] (9.5-0.3, 10.5-0.235) -- (8+0.335, 9.5+0.2);       \draw[thick] (9.5-0.335, 10.5-0.2-2.5) -- (8+0.3, 9.5+0.235-2.5); \draw[thick] (9.5-0.3, 10.5-0.235-2.5) -- (8+0.335, 9.5+0.2-2.5);          \draw[thick] (9.5-0.335, 10.5-0.2-7.5) -- (8+0.3, 9.5+0.235-7.5); \draw[thick] (9.5-0.3, 10.5-0.235-7.5) -- (8+0.335, 9.5+0.2-7.5);
		\draw[thick] (8-0.335, 9.5-0.2) -- (6.5+0.3, 8.5+0.235); \draw[thick] (8-0.3, 9.5-0.235) -- (6.5+0.335, 8.5+0.2);         \draw[thick] (8-0.335, 9.5-0.2-2.5) -- (6.5+0.3, 8.5+0.235-2.5); \draw[thick] (8-0.3, 9.5-0.235-2.5) -- (6.5+0.335, 8.5+0.2-2.5);            \draw[thick] (8-0.335, 9.5-0.2-7.5) -- (6.5+0.3, 8.5+0.235-7.5); \draw[thick] (8-0.3, 9.5-0.235-7.5) -- (6.5+0.335, 8.5+0.2-7.5);
		\draw[thick] (6.5-0.335, 8.5-0.2) -- (5+0.3, 7.5+0.235); \draw[thick] (6.5-0.3, 8.5-0.235) -- (5+0.335, 7.5+0.2);         \draw[thick] (6.5-0.335, 8.5-0.2-2.5) -- (5+0.3, 7.5+0.235-2.5); \draw[thick] (6.5-0.3, 8.5-0.235-2.5) -- (5+0.335, 7.5+0.2-2.5);            \draw[thick] (6.5-0.335, 8.5-0.2-7.5) -- (5+0.3, 7.5+0.235-7.5); \draw[thick] (6.5-0.3, 8.5-0.235-7.5) -- (5+0.335, 7.5+0.2-7.5);
		\draw[very thick, dotted] (6.5+0.3, 8.5+0.2) -- (6.5-0.3, 8.5-0.2);
		\draw[very thick, dotted] (6.5+0.3, 8.5+0.2-2.5) -- (6.5-0.3, 8.5-0.2-2.5);
		\draw[very thick, dotted] (6.5+0.3, 8.5+0.2-7.5) -- (6.5-0.3, 8.5-0.2-7.5);
		
	\end{tikzpicture}
\end{center}

Consider the red $(\grl+1)$-conflation $\sfXi \colon \xymatrix{\sfJ \ar[r] & \sfX \ar[r] & \sfC}$ appearing in the main diagonal and let us note how the induction hypothesis implies that each of the conflations
$\Xi_{\gra} \colon \xymatrix{J_{\gra} \ar[r] & X_{\gra} \ar[r] & C_{\gra}},$ $\gra < \grl,$
is trivial. Indeed, this conflation is obtained from the horizontal pushout at $\gra,$ given by
\begin{equation*}
	\begin{tikzcd}[row sep=scriptsize]
		J_0\arrow[r]\arrow[d, "{j^0_{\alpha}}"] & \bullet\arrow[r]\arrow[d] & C_{\alpha}\arrow[d, equal] \\
		J_{\alpha}\arrow[r]                     & X_{\alpha}\arrow[r]       & C_{\alpha}
	\end{tikzcd}
\end{equation*}
Because $j^0_{\gra} \in \mcJ^{(\gra)}$ and $C_{\gra}$ has an $\gra$-filtration in ${^{\perp}}\mcJ,$ $\Ext (C_{\gra}, j^0_{\gra}) = 0$ and the conflation is trivial.  

If we factor the successive structural morphisms of $\sfXi$ into their pushout-pullback factorizations, we obtain the system of conflations on the left. Every conflation in that system, except perhaps $\Xi_{\grl}$ in the bottom row, is trivial, because it is either of the form $\Xi_{\gra},$ $\gra < \grl,$ or arises as the pushout/pullback of a trivial conflation.
\begin{equation*}
	\begin{tikzcd}[row sep=scriptsize]
		\textcolor{red}{J_0}         \arrow[r, red]\arrow[d, red, "{j_1^0}"]                    & \textcolor{red}{X_0}         \arrow[r, red]\arrow[d] 	& \textcolor{red}{0}           \arrow[d, equal]                            & & & & 	\textcolor{red}{J_0}        \arrow[r, red]\arrow[d, red, "{j_1^0}"]                     & 	\textcolor{red}{X_0}        \arrow[r, red]\arrow[d]  & 	\textcolor{red}{0}          \arrow[d, equal]					\\
		J_1          \arrow[r]     \arrow[d, equal]                             &				  Y_1          \arrow[r]     \arrow[d] 	& 				  C_0          \arrow[d, red, "{c_1^0}"]                   & & & & 					J_1         \arrow[r]     \arrow[dd, red, "{j^1_2}"]                    & 					Y_1         \arrow[r]     \arrow[dd] & 					C_0         \arrow[dd, red, "{c^0_1}"]			\\
		\textcolor{red}{J_1}         \arrow[r, red]\arrow[d, red, "{j^1_2}"]                    & \textcolor{red}{X_1}         \arrow[r, red]\arrow[d] 	& \textcolor{red}{C_1}         \arrow[d, equal]                            & & & &  				                                                                        & 					                                     &	                           										\\
		J_2          \arrow[r]     \arrow[d, equal]                             & 				  Y_2          \arrow[r]     \arrow[d] 	& 				  C_1          \arrow[d, red, "{c^1_2}"]                   & & & & 					J_2         \arrow[r]     \arrow[dd, red, "{j^2_3}"]                    & 					Y_2         \arrow[r]     \arrow[dd] & 					C_1         \arrow[dd, red, "{c^1_2}"]   \\
		\textcolor{red}{J_2}         \arrow[r, red]\arrow[d, red, "{j^2_3}"]                    & \textcolor{red}{X_2}         \arrow[r, red]\arrow[d] 	& \textcolor{red}{C_2}         \arrow[d, equal]                            & & & & 	                                                                                        &					                                     &                             								\\
		J_3          \arrow[r]     \arrow[d, equal]                             & 				  Y_3          \arrow[r]     \arrow[d] 	& 				  C_2          \arrow[d, red, "{c^2_3}"]                   & & & & 					J_3         \arrow[r]     \arrow[d, red, "{j^3_4}"]                     & 					Y_3         \arrow[r]     \arrow[d]  & 					C_2         \arrow[d, red, "{c_3^2}"]    \\
		\vdots                     \arrow[d]                                    & 				  \vdots                     \arrow[d] 	& 				  \vdots       \arrow[d              ]                     & & & & 					\vdots                    \arrow[d                ]                     & 					\vdots                    \arrow[d]  & 					\vdots      \arrow[d]            				\\
		\textcolor{red}{J_{\omega}}  \arrow[r, red]\arrow[d, red, "{j_{\omega+1}^{\omega}}"]	& \textcolor{red}{X_{\omega}}  \arrow[r, red]\arrow[d] 	& \textcolor{red}{C_{\omega}}  \arrow[d, equal]                            & & & & 	\textcolor{red}{J_{\omega}} \arrow[r, red]\arrow[d, red, "{j_{\omega+1}^{\omega}}"]     & 	\textcolor{red}{X_{\omega}} \arrow[r, red]\arrow[d]  & 	\textcolor{red}{C_{\omega}} \arrow[d, equal] 						\\
		J_{\omega+1} \arrow[r]     \arrow[d, equal]                             & 				  Y_{\omega+1} \arrow[r]     \arrow[d] 	& 				  C_{\omega}   \arrow[d, red, "{c_{\omega+1}^{\omega}}"]   & & & & 					J_{\omega+1}\arrow[r]     \arrow[dd, red, "{j_{\omega+2}^{\omega+1}}"]  & 					Y_{\omega+1}\arrow[r]     \arrow[dd] & 					C_{\omega}  \arrow[dd, red, "{c_{\omega+1}^{\omega}}"] \\
		\textcolor{red}{J_{\omega+1}}\arrow[r, red]\arrow[d, red, "{j^{\omega+1}_{\omega+2}}"]  & \textcolor{red}{X_{\omega+1}}\arrow[r, red]\arrow[d] 	& \textcolor{red}{C_{\omega+1}}\arrow[d, equal]                       	   & & & & 					                                                                        & 					                                     & 	                            										\\
		J_{\omega+2} \arrow[r]     \arrow[d, equal]                             & 				  Y_{\omega+2} \arrow[r]     \arrow[d] 	& 				  C_{\omega+1} \arrow[d, red, "{c^{\omega+1}_{\omega+2}}"] & & & & 					J_{\omega+2}\arrow[r]     \arrow[d, red, "{j^{\omega+2}_{\omega+3}}"]   & 					Y_{\omega+2}\arrow[r]     \arrow[d]  & 					C_{\omega+1}\arrow[d, red, "{c^{\omega+1}_{\omega+2}}"] \\
		\vdots                     \arrow[d]                                    & 				  \vdots                     \arrow[d]  & 				  \vdots       \arrow[d              ]                     & & & & 					\vdots                    \arrow[d                ]                     & 					\vdots                    \arrow[d]  & 					\vdots      \arrow[d] 								\\
		\textcolor{red}{J_{\lambda}} \arrow[r, red]                                             & \textcolor{red}{X_{\lambda}} \arrow[r, red]           & \textcolor{red}{C_{\lambda}}                                             & & & & 	\textcolor{red}{J_{\lambda}}\arrow[r, red]                                              & 	\textcolor{red}{X_{\lambda}}\arrow[r, red]           & \textcolor{red}{C_{\lambda}}
	\end{tikzcd}
\end{equation*}
The $(\grl + 1)$-conflation on the right, call it $\sfgrS,$ is obtained from the diagram on the left by composing the pairs that factor through a conflation of the form $\Xi_{\gra+1}.$ Thus $\sfgrS|_{\grl}$ is a $\grl$-system of trivial conflations. We will show that it satisfies the downward adjustment property. Then Theorem~\ref{T:DAP} applies to show that $\sfgrS|_{\grl}$ is a trivial $\grl$-conflation. Because the two outside $\grl$-systems are continuous at $\grl,$ the workhorse Lemma~\ref{L:workhorse} will apply to show that $\sfgrS$ and therefore $\grS_{\grl} = \Xi_{\grl}$ is trivial, as required.

To verify that splittings adjust downwards for $\sfgrS,$ there are two cases to consider. When $\grg = 0$ or a limit, the conflation of arrows is of the form
$$\xymatrix@C=0pt{\grS_{\grg} \colon & J_{\grg} \ar[rrrrrr] \ar[d] &&&&&& X_{\grg} \ar[rrrrrr] \ar[d] &&&&&& C_{\grg} \ar@{=}[d] \\
\grS_{\grg+1} \colon & J_{\grg+1} \ar[rrrrrr] &&&&&& Y_{\grg+1} \ar[rrrrrr] &&&&&& C_{\grg}}$$
and will therefore adjust downwards by Corollary~\ref{L:FH}. For the case when $\grg = \gra+1$ is a successor, the factorization given in the left diagram satisfies the hypotheses of Lemma~\ref{L:FH1}.
\end{proof}

\section{Transfinite inductive powers of ghost ideals}
In this short section, we generalize Proposition~\ref{finite1} to the transfinite situation. To do so, we need the following hypothesis.

\begin{hypothesis}\label{hyp1}{\rm Let $\lambda$ be an ordinal.
Every continuous $\lambda$-system $\sfA$ in $\mcA$ whose structural morphisms $a_{\alpha+1}^\alpha$, $\alpha+1 < \lambda$, are inflations in $\mcA$,
 has a composition morphism $A_0 \longrightarrow \lim\limits_{\longrightarrow\atop{\alpha<\lambda}} A_{\alpha}$ that is also an inflation in $\mcA.$}
\end{hypothesis}

\begin{theorem} \label{T:OSPE preservation}
	Suppose that $\mcA$ satisfies Hypothesis~\ref{hyp1} for some  infinite ordinal $\lambda$. If  $\mcJ$ is  an object-special preenveloping ideal with $\Omega^{-1}(\mcJ)$ a cosyzygy subcategory, then the  $\lambda$-th inductive power $\mcJ^{(\lambda)}$ is still an object-special preenveloping ideal with cosyzygy category $\grl \mbox{-}\Filt (\Omega^{-1}(\mcJ))$.
\end{theorem}

\begin{proof}
	Let $A$ be an object in the category $\mcA$.  We will construct, by transfinite induction on $\mu \leq  \lambda,$ a $(\mu+1)$-system $(\Xi_\alpha, {\xi}_\beta^\alpha \mid\ \alpha < \beta \leq \mu)$ of conflations
	$\Xi_\alpha:$ $\xymatrix{A\ar[r]^{j_{\alpha}}& J_\alpha \ar[r]& W_\alpha}$
	with $j_{\alpha}\in \mcJ^{(\alpha)}$, and $W_{\alpha}\in \gra$-$\Filt ( \Omega^{-1}(\mcJ))$ such that  for every $\alpha<  \mu,$ the morphisms $j_{\alpha+1}^{\alpha} $ and $w_{\alpha+1}^\alpha$ in the commutative diagram
	$$\xymatrix@C=0pt{\Xi_{\alpha} \ar@<-1ex>[d]_{\xi^{\alpha}_{\beta}}: &A\ar[rrrrrr]^-{j_\alpha}\ar@{=}[d]&&&&&& J_{\alpha}\ar[rrrrrr]\ar[d]^{j^{\alpha}_{\alpha+1}}&&&&&&W_{\alpha}\ar[d]^{w^{\alpha}_{\alpha+1}}\\
		\Xi_{\alpha+1}: & A\ar[rrrrrr]^-{j_{\alpha+1}}&&&&&&J_{\alpha+1 }\ar[rrrrrr]&&&&&&W_{\alpha+1 }}$$
 of conflations are inflations.

	The inductive step has already been treated in Propostion \ref{finite1}.
	For the  limit step, assume that for some limit ordinal $\beta \leq \mu$, we have such a $\beta$-system  $(\Xi_\alpha, {\xi}_{\alpha'}^\alpha \mid\ \alpha < \alpha' < \beta)$ of conflations in $\mcA$
	$$\xymatrix@C=0pt{\Xi_1\ar@<-1ex>[d]_{{\xi}_2^1}: &&A\ar[rrrrrr]^-{j_1}\ar@{=}[d]&&&&&&J_1\ar[rrrrrr]\ar[d]^{j_2^1}&&&&&&W_1\ar[d]^{w_2^1}\\
		\Xi_2\ar@<-1ex>[d]_{{\xi}_3^2}: &&A\ar[rrrrrr]^-{j_2}\ar@{=}[d]&&&&&&J_2\ar[rrrrrr]\ar[d]^{j_3^2}&&&&&&W_2\ar[d]^{w_3^2}\\
		\vdots\;\;\; \ar@<-1ex>[d] &&\vdots\ar@{=}[d]&&&&&&\vdots\ar[d]&&&&&&\vdots\ar[d]\\
		\Xi_{\alpha}\ar@<-1ex>[d]_{{\xi}_{\alpha+1}^{\alpha}}: &&A\ar[rrrrrr]^-{j_\alpha}\ar@{=}[d]&&&&&&J_{\alpha}\ar[rrrrrr]\ar[d]^{j_{\alpha+1}^{\alpha}}&&&&&&W_{\alpha}\ar[d]^{w_{\alpha+1}^{\alpha}}\\
		\vdots\;\;\;&&\vdots&&&&&&\vdots&&&&&&\vdots},$$
	As ${j}_{\alpha+1}^\alpha$ and ${w}_{\alpha+1}^\alpha$ are inflations for every $\alpha < \beta$,  $j_\beta:=\lim\limits_{\longrightarrow\atop{\alpha<\beta}} j_{\alpha}$ and ${w_\beta:=\lim\limits_{\longrightarrow\atop{\alpha<\beta}} w_{\alpha}}$ exist, and
	$$\xymatrix@C=0pt{\Xi_{\beta}:=\lim\limits_{\longrightarrow\atop{\alpha<\beta}} \Xi_\alpha :&A\ar[rrrrrr]&&&&&&J_{\beta}\ar[rrrrrr]&&&&&&W_{\beta}}$$
	is a conflation in $\mcA$ by Hypothesis \ref{hyp1}. By construction and the Ideal Eklof Lemma (Theorem~\ref{TFH4}),  $W_{\beta} \in \grb \mbox{-}\Filt( \Omega^{-1}(\mcJ)) \subseteq {}^\perp (\mcJ^{(\beta)})$. Since $j_{\alpha+1}^\alpha \in \mcJ$ for every $\alpha < \beta$, $j_{\beta}^0=j_\beta:A \longrightarrow J_{\beta}$ is a $\beta$-transfinite composition of morphisms in $\mcJ$,  and therefore, $j_\beta \in\mcJ^{(\beta)}$.
	
	Taking $\mu=\lambda$ will give us a conflation $\Xi_\lambda \colon \xymatrix@1{A \ar[r]^{j_{\grl}} &  J_\lambda \ar[r] & W_{\lambda}}$
	such that $j_\lambda \in \mcJ^{(\lambda)}$ and $W_\lambda \in \grl$-$\Filt (\Omega^{-1}(\mcJ))$. Note that $W_\lambda$ belongs to $^{\perp}(\mcJ^{(\lambda)})$ by the Ideal Eklof Lemma, so $\mcJ^{(\lambda)}$ is an object-special preenveloping ideal.
\end{proof}

As an immediate consequence, we have:

\begin{corollary}\label{TFH5} Let $\mathcal{S}$ be a set of objects in $\mcA$. Suppose that  $\mcA$ has  exact coproducts, and satisfies Hypothesis \ref{hyp1} for an infinite ordinal  $\lambda$. Then  the $\lambda$-th inductive power $\mathfrak{g}^{(\lambda)}_{\mathcal{S}}$ is an object-special preenveloping ideal with $\grl \mbox{-}\Filt(\Sum(\mathcal{S}))$ a cosyzygy subcategory. Moreover, if $\mathcal{A}$ has enough projectives, then $(^{\perp}(\mathfrak{g}^{(\lambda)}_{\mathcal{S}}),\mathfrak{g}^{(\lambda)}_{\mathcal{S}})$ is a complete ideal cotorsion pair such that $^{\perp}(\mathfrak{g}^{(\lambda)}_{\mathcal{S}})$ is an object ideal with
$$\operatorname{Ob}(^{\perp}(\mathfrak{g}^{(\lambda)}_{\mathcal{S}}))  =  \operatorname{add}(\mcE \operatorname{-Proj} \star \grl \mbox{-}\Filt(\Sum(\mathcal{S}))).$$
\end{corollary}

\begin{remark}\label{dual.results}

{\rm The notions/techniques and results   presented so far may be dualized. For a given set $\mcS$ of objects in $ \mcA$, the ideal left orthogonal to $\langle \mcS \rangle$ is called \textit{the ideal of the $\mcS$-coghost morphisms},  and is denoted by $\mathfrak{Cog}_{\mcS} = {^{\perp}} \langle \mcS \rangle.$ On the other hand,
the {\em projective $\grb$-th power} $\mcI^{\grb}$ of an ideal $\mcI$ is defined as the ideal   generated by all $\grb$-cocompositions  of morphisms in $\mcI$ (dual to Definition~\ref{D:cont}). Under dual conditions given in Corollary~\ref{TFH5}, we have
a complete ideal cotorsion pair
$$(\mathfrak{Cog}_{\mcS}^{\grb}, \langle\operatorname{add}( \grb \mbox{-}\Cofilt(\Prod(\mathcal{S}))\star \mcE \operatorname{-Inj})\rangle),$$ where for a given subcategory $\mathcal C\subseteq \mcA$, the subcategory $\grb \mbox{-}\Cofilt(\mathcal C)$ denotes the objects having a $\grb\mbox{-}\mathcal C$-cofiltration (dual to Definition \ref{D:filt}). }

\end{remark}

\section{The Generalized Generating Hypothesis in Exact Categories}
Let $\mcJ$ be an ideal. Given ordinals $\mu \geq  \lambda,$ Proposition~\ref{prop:morph_induct} implies that $\mcJ ^{(\mu)} \subseteq \mcJ^{(\lambda)},$ so there is a decreasing chain of inductive powers of  $\mcJ$
\begin{equation} \label{chain}
	\cdots \subseteq \mcJ^{(\lambda)} \subseteq \cdots \subseteq  \mcJ^2 \subseteq \mcJ.
\end{equation}
In this section, we focus on  the question of whether the chain~\eqref{chain} stabilizes at some ordinal, that is, if there exists an ordinal $\lambda$ such that  $\mcJ ^{(\mu)} =\mcJ^{(\lambda)}$ for every ordinal $\mu \geq \lambda.$ It is hard to  answer   the question  in such a general form, so we consider two special cases of $\mcJ$, each of which   comes up naturally in certain contexts of interest.

\subsection{Idempotent ideals} Consider an \emph{idempotent} ideal $\mcJ = \mcJ^2.$ In that case, $\mcJ^n=\mcJ$ for every finite ordinal $n \geq 1,$ but it is not immediate that the chain~\eqref{chain} stabilizes, so it remains to investigate its infinite inductive powers. In order to do so, let us consider some sufficient conditions. Let $\lambda$ be a given infinite ordinal.
We say that \emph{$\mcJ$ is closed under $\lambda$-directed limits} if $\mcJ$ has $\lambda$-compositions, and for any continuous $\mu$-system $(J_\alpha, j_\beta^\alpha \mid  \alpha< \beta< \mu)$
in  $\mcJ$ with $\mu \leq \lambda$ and a family of maps $\{g_{\alpha}:J_\alpha\to X \mid \alpha<\mu\}$ with $g_{\alpha}\in\mcJ$ such that $g_{\alpha}=g_{\beta}j^{\alpha}_{\beta}$ for $\alpha<\beta<\mu$, the canonical morphism $\lim\limits_{\longrightarrow\atop{\alpha <\mu} }g_\alpha:\lim\limits_{\longrightarrow\atop{\alpha <\mu} }J_\alpha\to X$ belongs to $\mcJ$.

\begin{proposition}\label{PFH41}
	Let $\mcJ$ be an idempotent ideal in $\mcA.$ If $\mcJ$ has $\omega$-compositions, then $\mcJ=\mcJ^{(\omega)}$
	If $\lambda\geq\omega$ is an infinite ordinal, and if $\mcJ$ is closed under $\lambda$-directed limits, then $\mcJ^{(\lambda)}=\mcJ$.
\end{proposition}

\begin{proof} The inclusion $\mcJ^{(\omega)}\subseteq\mcJ$ is trivial. For the other direction,
	let $g: A \longrightarrow J$ be a morphism in $\mcJ$.  We will construct a continuous $\omega$-system $(J_n, j_m^n \mid n < m \in \mathbb{N} )$ in $\mcJ$   with $J_0=A$ , and a family  $\{g_{n}:J_n\to J\mid  n \in \mathbb{N}\}$ of morphisms with $g_{n} = g_{ n+1} j_{ n+1}^n$ such that $g$ factors through ${J_{\omega}:=\lim\limits_{\longrightarrow\atop{n \in \mathbb{N}}} J_n}$. Because $J_0 \longrightarrow J_{\omega}$ belongs to $\mcJ^{(\omega)}$, this will show that $g\in\mcJ^{(\omega)}.$
	
	Since $\mcJ=\mcJ^2$, there exist morphisms $j_1^0:A \longrightarrow J_1$ and $g_1:J_1 \longrightarrow J$ in $\mcJ$ such that $g=g_0= g_1 j_1^0$. But again, since $g_1$ belongs to $\mcJ$, there exist morphisms $j_2^1:J_1 \longrightarrow J_2$ and $g_2:J_2 \longrightarrow J$ in $\mcJ$ such that $g_1= g_2 j_2^1$. Continuing this process, we finally obtain a continuous $\omega$-system of $\mcJ$ $(J_n, j_m^n \mid n < m \in \mathbb{N} )$ and a family of morphisms $\{g_{n}:J_n\to J\mid  m \in \mathbb{N}\}$ satisfying $g_{n} = g_{ n+1} j_{ n+1}^n$. That $g$ factors through ${J_{\omega}:=\lim\limits_{\longrightarrow\atop{n \in \mathbb{N}}} J_n}$ follows from the universal property of direct limits.
	
	For the second assertion, if $\mcJ$ is closed under $\lambda$-directed limits, then the above construction can keep going on for any ordinal $\alpha<\lambda$, because if $\alpha=\mu$ is a limit ordinal, the induced map  $g_{\mu}=\lim\limits_{\longrightarrow\atop{\alpha <\lambda} }g_\alpha:\lim\limits_{\longrightarrow\atop{\alpha <\mu} }J_\alpha\to J$ belongs to $\mcJ$. Thus we may construct a $\lambda$-system of $\mcJ$ $(J_\alpha, j_\beta^\alpha \mid  \alpha< \beta< \lambda)$ and a family of maps $\{g_{\alpha}:J_\alpha\to J \mid \alpha<\lambda\}$ with $g_{\alpha}\in\mcJ$ such that $g_{\alpha}=g_{\beta}j^{\alpha}_{\beta}$ for $\alpha<\beta<\lambda$, and $g$ factors through $\lim\limits_{\longrightarrow\atop{\alpha<\lambda} }J_\alpha$ follows from the universal property of direct limits, this will establish the result.
\end{proof}

\subsection{The Generalized Generating Hypothesis for Ghost ideals}Next we consider the case that $\mcJ$ is an $\mcS$-ghost ideal $\mathfrak{g}_{\mathcal{S}}$, then we have the following  decreasing chain of inductive powers of $\mathfrak{g}_{\mathcal{S}}$
$$\cdots \subseteq \mathfrak{g}^{(\lambda)}_{\mathcal{S}} \subseteq \cdots \subseteq \mathfrak{g}^{2}_{\mathcal{S}} \subseteq\mathfrak{g}_{\mathcal{S}}.$$
It is clear  from the definition of inductive powers of an ideal that  any inductive power of $\mathfrak{g}_{\mathcal{S}}$ contains the object ideal
$\langle\mathcal{S}^\perp\rangle$. Hence, the aforementioned  decreasing chain is bounded below by $\langle\mathcal{S}^\perp\rangle$

\begin{equation}\label{chain:ghost}
	\langle\mathcal{S}^\perp\rangle \subseteq \cdots \subseteq \mathfrak{g}^{(\lambda)}_{\mathcal{S}} \subseteq \cdots \subseteq \mathfrak{g}^{2}_{\mathcal{S}} \subseteq\mathfrak{g}_{\mathcal{S}}.
\end{equation}
We say that  \textit{the Generalized $\lambda$-Generating Hypothesis $\grl$-{\rm GGH(}$\mathfrak{g}_\mcS)$ holds}  if $\mathfrak{g}^{(\lambda)}_{\mathcal{S}}=\langle\mcS^\perp\rangle.$

\begin{remark}\label{remark:termin_ghost} \rm
	The Generating Hypothesis was first introduced by Freyd in \cite{Fre66}.  Lockridge~\cite{Loc07} proposed and studied The Generating Hypothesis for an algebraic triangulated category. Almost at the same time,  the study of the Generating Hypothesis for the ghost ideal  in the stable module category $kG \operatorname{-\underline{mod}}$ (see Example \ref{ex:ghost_tate})  was initiated by Benson, Chebolu, Christensen and Miná\u{c} in \cite{BCCM07}.
	 Note that the ghost ideal in $kG \operatorname{-\underline{Mod}}$  is trivial  if and only if the corresponding ghost ideal in $kG  \operatorname{-Mod}$ is the object ideal $\langle \operatorname{Proj} \rangle$. In  Propositions \ref {proposition:generating_homotopy} and \ref{prop:generating_Derived}, we show that being trivial for ghost ideals in $\mathbf{K}(R)$ and $\mathbf{D}(R)$ is equivalent to being an object ideal for the corresponding ghost ideals in the exact categories $\mathbf{C}(R)$ and $\operatorname{dg-Proj}$, respectively. From this point of view, the object ideals in an exact category can be considered as `trivial'.
\end{remark}

Recall that an exact category $\mcA$ is said to be \textit{efficient}  \cite{SS11} if it satisfies the following conditions: (i) $\mcA$ is weakly idempotent; (ii) arbitrary transfinite composition of inflations in $\mcA$ exists, and is an inflation, as well; (iii) every object in $\mcA$ is small relative to inflations; (iv) $\mcA$ admits a generator.

\begin{lemma} \cite[Lemma 1.4 and Theorem 2.13(4)]{SS11} and \cite[ Proposition 5.8]{Sto13}
	Let $\mcS$ be a set of objects in $\mcA$. Assume that $\mcA$ is an efficient exact category. Then $\mcA$ has exact coproducts, and  for any object $A$ in $\mcA$, there exists a conflation in $\mcA$ of the form
	$$\Xi: \xymatrix{A \ar[r] & B \ar[r] & C}$$
	with $B \in \mcS^\perp$ and $C \in \Filt(\mcS)$, that is, the cotorsion pair $( ^\perp (\mcS^\perp), \mcS^\perp )$ has enough injectives.
\end{lemma}

The above lemma can be used to prove the following result.
\begin{proposition}\label{prop:object}Let $\mcS$ be a set of objects in $\mcA$. Assume that $\mcA$ is  an efficient exact category. Then $\grl$-{\rm GGH(}$\mathfrak{g}_\mcS)$ holds if and only if
	for every ordinal $\mu \geq\lambda $,
	$$\mathfrak{g}^{(\mu)}_{\mathcal{S}}=\mathfrak{g}^{(\lambda)}_{\mathcal{S}}.$$
\end{proposition}

\begin{proof}
	If $\grl$-{\rm GGH(}$\mathfrak{g}_\mcS)$ holds, then it is clear that for every ordinal $\mu \geq\lambda $,
	$\mathfrak{g}^{(\mu)}_{\mathcal{S}}=\mathfrak{g}^{(\lambda)}_{\mathcal{S}}.$	
	Conversely, let $f:A\to X$ be a morphism in $\mathfrak{g}^{(\lambda)}_{\mathcal{S}}$ . By the above lemma, there is a conflation $$\Xi: \xymatrix{A \ar[r] & B \ar[r] & C}$$
	with $B\in {\mcS}^\perp$ and $C \in \mu \mbox{-}\Filt(\mcS)$ for some $\mu$, and hence $C \in \mu \mbox{-}\Filt(\mbox{Sum}(\mcS))$. If $\mu\leq\lambda$, then $C\in \lambda  \mbox{-}\Filt(\mbox{Sum}(\mcS))$, and hence $\Ext(C,f)=0$ by The Ideal Eklof Lemma (Theorem \ref{TFH4}). If $\mu\geq\lambda$, then $\Ext(C,f)=0$ follows from $f\in \mathfrak{g}^{(\mu)}_{\mathcal{S}}=\mathfrak{g}^{(\lambda)}_{\mathcal{S}}$. Therefore $f$ always factors through $B$, and so $f\in\langle  \mcS^\perp\rangle$.
\end{proof}

Now we handle the question of whether there exists an ordinal $\lambda$ such that $\grl$-{\rm GGH(}$\mathfrak{g}_\mcS)$ holds. Recall that a cardinal number $\lambda$ is just the least ordinal of cardinality $\lambda$.
So far, $\lambda$ has been considered as an arbitrary ordinal. Now we switch $\lambda$ to a cardinal number, which can be still considered as an ordinal when needed to work with inductive powers of an ideal and filtrations.

Recall that an object $A$ in  $\mcA$ is said to be \emph{$\lambda$-presentable} if the functor $\Hom(A,-)$
preserves $\lambda$-directed colimits. The category $\mcA$ is called
\emph{ locally $\lambda$-presentable} if it is cocomplete and there is a
set $\mathcal C$ of $\lambda$-presentable objects in $\mcA$ such that any other object in $\mcA$ is a $\lambda$-directed colimit of objects in
$\mathcal C$. In case $\lambda=\omega$, a $\lambda$-presentable object is called \emph{finitely presented} and a locally $\lambda$-presentable category is known as \emph{locally finitely presentable}. A category $\mcA$ is called \emph{locally presentable} if it is locally $\lambda$-presentable for some cardinal $\lambda$. Every object in a locally presentable category is $\lambda'$-presentable for some regular cardinal $\lambda'$; see \cite[Proposition~1.16]{AR94}. For a detailed treatment on locally presentable categories, see  \cite{AR94}.

It is well known that if $\mcA$ is a Grothendieck category with the absolute exact structure, then $\mcA$ is locally $\lambda$-presentable for some infinite regular cardinal $\lambda$; see \cite[Proposition~3.10]{Beke}. For a given class $\mcS$ of $\lambda$-presentable objects in $\mcA$,  \v{S}\'{t}ov\'{i}\v{c}ek in \cite[Theorem 3.1]{Sto} showed that an object $X$ which  has an $\mcS$-filtration, has a $\mbox{Sum}(\mcS)$-filtration of length $\leq\lambda$. This gives a global bound on the length of the $\mbox{Sum}(\mcS)$-filtration for objects in $\Filt(\mcS)$. Enochs obtained a similar bound in the settting of the category of modules (see~\cite{Enochs2}), so we call it the \emph{Enochs-\v{S}\'{t}ov\'{i}\v{c}ek bound} on filtrations for objects in $\Filt$-$\mcS$.

\begin{theorem} \label{EFHS}
	Suppose that  $\lambda$ is an infinite regular cardinal, and $\mcA$ is a locally $\lambda$-presentable Grothendieck category. If $\mathcal{S}$ is a set of $\lambda$-presentable objects in $\mcA$ such that $^\perp (\mathcal{S}^\perp)$ contains a generating set for $\mcA$, then the Generalized $\lambda$-Generating Hypothesis holds for $\mathfrak{g}_{\mathcal{S}}$.
\end{theorem}

\begin{proof}
	From  \cite[Lemma 3.6]{Gill07},  we have
	$^\perp(\mcS^\perp)=\add(\mbox{Filt}(\mcS)),$
	which also implies that  $( ^\perp(\mcS^\perp), \mcS^\perp)$ is a complete cotorsion pair in $\mcA$.
	On the other hand, the Enochs-\v{S}\'{t}ov\'{i}\v{c}ek bound tells us that
	$\mbox{Filt} (\mcS)= \lambda \mbox{-} \Filt( \mbox{Sum}(\mcS)),$ (here $\lambda$ is thought as an ordinal in the usual way). Then
	$$\langle^\perp(\mcS^\perp)\rangle=\langle \mbox{Filt}(\mcS)\rangle =\langle \lambda \mbox{-} \Filt( \mbox{Sum}(\mcS)) \rangle\ .$$
	By Theorem \ref{TFH5}, $\lambda \mbox{-} \Filt( \mbox{Sum}(\mcS))$ is a cosyzygy subcategory of  $\mathfrak{g}_{\mathcal{S}}^{(\lambda)}$, so
	$$ \mathfrak{g}_{\mathcal{S}}^{(\lambda)} \subseteq   \langle\lambda \mbox{-} \Filt( \mbox{Sum}(\mcS)) \rangle^\perp=\langle^\perp(\mcS^\perp)\rangle^\perp$$
	Note that $( \langle ^\perp(\mcS^\perp) \rangle, \langle \mcS^\perp \rangle   )$ is an ideal cotorsion pair in $\mcA$ by \cite[Theorem 28]{FGHT}. Therefore
	$\mathfrak{g}_{\mathcal{S}}^{(\lambda)}\subseteq \langle^\perp(\mcS^\perp)\rangle^\perp = \langle\mcS^\perp \rangle.$
\end{proof}

By the considerations above, every object of a Grothendieck category $\mcA$ is presentable, so if $\mcS$ is a set of objects in $\mcA$, there is an infinite cardinal $\lambda$ such that $\mcA$ is locally $\lambda$-presentable and $\mcS$ is a set of $\lambda$-presentable objects in $\mcA.$

\begin{remark} \rm
	Let us look more closely at the case when $\lambda$ (regarded as an ordinal) is $\omega,$ as it will be the main source of applications in the next sections. Theorem \ref{EFHS} ensures that if $\mcA$ is a locally finitely presentable Grothendieck category,  and $\mathcal{S}$ is a set of finitely presented objects in $\mcA$ such that $^\perp (\mathcal{S}^\perp)$ contains a generating set for $\mcA$ (for instance, if $\mcA$ has enough projective objects), then  $\omega$-{\rm GGH(}$\mathfrak{g}_\mcS)$ holds, or equivalently, that the $\omega$-inductive power $\mathfrak{g}_{\mathcal{S}}^{{ } (\omega)}$ of $\mathfrak{g}_{\mathcal{S}}$ is the ideal $\langle\mathcal{S}^\perp\rangle$.
\end{remark}

\begin{remark}\label{dual.results2}
{\rm By contrast to Remark \ref{dual.results}, we are not able to dualize the results of this section. This is due to the absence of a theory of cofiltrations and their bounds.}
\end{remark}

\section{The ghost ideal in $\mathbf{K}(R)$, $\mathbf{D}(R)$ and $\mathbf{C}(R)$}
From now on, $R$ denotes an associative ring with identity. Let $R\mbox{-Mod}$ be the category of left $R$-modules,  $\mathbf{C}(R)$ the category of chain complexes of left $R$-modules,  $\mathbf{K}(R)$ the homotopy category of chain complexes of left $R$-modules, and  $\mathbf{D}(R)$ the derived category of chain complexes of left $R$-modules.

In this section, we investigate ghost morphisms and ghost ideals in the category $\mathbf{C}(R)$, and its relation with the existent ones in $\mathbf{D}(R)$ and $\mathbf{K}(R)$.

Given $n \in \mathbb{Z}$ and $A \in R$-$\Mod$, \textit{the $n$th sphere chain complex} $\operatorname{S}^n(A)$ is the chain complex of the form
 $$\xymatrix{ \cdots \ar[r] &  0 \ar[r] &  A \ar[r] & 0 \ar[r] &  \cdots }$$
 with $A$ in the $n$th place; \textit{the $n$th disc chain complex} $\operatorname{D}^n(A)$ is the chain complex of the form
 $$\xymatrix{\cdots\ar[r]&0\ar[r]&A\ar[r]^{1_A}&A\ar[r]&0\ar[r]&\cdots}$$ with two $A$'s in the $n$-th and $(n-1)$-th places.

The \textit{translation functor on $\mathbf{C}(\mcA)$} is the autofunctor $\Sigma:\ \mathbf{C}(\mcA)\to \mathbf{C}(\mcA)$ defined by
$$\Sigma (X)_n:=X_{n-1}$$
and with differential $d^{ \Sigma (X)}=-d^{X}$. Its inverse is denoted by $\Sigma^{-1}$. For a given morphism $f: X \rightarrow Y$ in $\mathbf{C}(\mcA)$, the \textit{cone} of $f$, denoted by $ \operatorname{cone}(f)$, is a chain complex defined by $\operatorname{cone}(f)_n:= Y_n\oplus X_{n-1}$ with differentials
$$\left[ \begin{array}{cc} d_n^Y & f_{n-1}\\ 0 & -d_{n-1}^X  \end{array} \right]: Y_n \oplus X_{n-1} \longrightarrow Y_{n-1}\oplus X_{n-2}. $$
 The cone  of $f$ induces the following canonical exact sequence in $\mathbf{C}(\mcA)$ defined by degree-wise inclusion and projection
$$\xymatrix{ Y \ar[r] & \operatorname{cone}(f) \ar[r] & \Sigma (X)}.$$

\subsection{The Ghost ideal in $\mathbf{C}(R)$ and the Cartan-Eilenberg exact structure} Note that for any integer $n$, there are four basic functors from the category $\mathbf{C}(R)$ to  $R\mbox{-Mod}$: the $n$th cycle functor $\operatorname{Z}_n$, the $n$th cokernel functor $\mbox{C}_n$, the $n$th boundary functor  $\operatorname{B}_n$, and the $n$th homology functor $\operatorname{H}_n$.
\begin{proposition}\label{CE1}\cite[Lemma 5.2]{Enochs} Let $\Xi :\ 0 \rightarrow X \rightarrow Y \rightarrow Z \rightarrow 0$ be a short exact sequence of   chain complexes of  left $R$-modules. The following are equivalent:
	\begin{enumerate}[(i)]
		\item For every $n \in \mathbb{Z}$, the sequence $\operatorname{Z}_n(\Xi)$  is exact.
		\item For every $n \in \mathbb{Z}$, the sequence $\operatorname{B}_n(\Xi)$ is exact.
		\item For every $n \in \mathbb{Z}$, the sequence $\operatorname{H}_n(\Xi )$ is exact.
		\item For every $n \in \mathbb{Z}$, the sequence $\operatorname{C}_n(\Xi)$ is exact.
	\end{enumerate}
\end{proposition}

\begin{definition}
	A short exact sequence $\Xi: 0 \rightarrow X \rightarrow Y \rightarrow Z \rightarrow 0$ of chain complexes of left $R$-modules is said to be \emph{Cartan-Eilenberg exact} {\rm (}$\operatorname{CE}$\emph{-exact} for short{\rm )} if one of the equivalent conditions given in Proposition \ref{CE1} is satisfied.
\end{definition}

It is easy to verify that the class $\mcE_{\operatorname{CE}}$ consisting of  Cartan-Eilenberg short exact sequences in $\mathbf{C}(R)$ is an exact structure on $\mathbf{C}(R)$, we call it the {\em Cartan-Eilenberg exact structure} on $\mathbf{C}(R)$.

Next, we always let $E$ be an injective cogenerator in the category $R\mbox{-Mod}$ of left $R$-modules. We have the following result.

\begin{proposition}\label{prop:CE_proj_gener}
	The Cartan-Eilenberg exact structure  $\mcE_{\operatorname{CE}}$ is projectively generated by the set $\{\operatorname{S}^n (R)\}_{n \in \mathbb{Z}}$, and is injectively generated by the set $\{\operatorname{S}^n(E)\}_{n \in \mathbb{Z}}$.
\end{proposition}
\begin{proof}
	The first statement follows from Proposition \ref{CE1}(i) and the fact
	$$\Hom(\operatorname{S}^n(R), \Xi)\cong \Hom(R, \operatorname{Z}_n(\Xi)) \cong \operatorname{Z}_n(\Xi).$$
	For the second one, we have
	$$\Hom(\Xi, \operatorname{S}^n(E)) \cong \Hom(\operatorname{C}_n(\Xi), E).$$
	Since  $E$ an injective cogenerator in $\RMod$, $\Hom(\operatorname{C}_n(\Xi), E)$ is a short exact sequence  of abelian groups  if and only if $\operatorname{C}_n(\Xi)$ is a short exact sequence of left $R$-modules. Applying  Proposition~\ref{CE1}(iv), we have the second statement.
\end{proof}

Note that the set  $\{\operatorname{D}^n(R)\}_{n \in \mathbb{Z}}$ is a generating set for $\mathbf{C}(R)$ with projective chain complexes, and the exact structure  $\mcE_{\operatorname{CE}}$ is clearly also projectively generated by $\{\operatorname{S}^n(R), \operatorname{D}^n(R)\}_{n \in \mathbb{Z}}$.  Dually, the exact structure  $\mcE_{\operatorname{CE}}$ is injectively generated by the set $\{\operatorname{S}^n(E), \operatorname{D}^n(E)\}_{n \in \mathbb{Z}}$, which  contains a cogenerating set $\{\operatorname{D}^n(E)\}_{n \in \mathbb{Z}}$. Therefore the exact category $(\mathbf{C}(R);\mcE_{\operatorname{CE}} )$ has enough projectives and injectives by Proposition \ref{projectivelygenerated}.

\begin{proposition}\label{prop:CE-proj}\cite[Propositions 3.3 and 3.4.]{Enochs}
	Let $X \in \mathbf{C}(R)$. Then $X$ is a projective object in $(\mathbf{C}(R); \mcE_{\operatorname{CE}})$, called  $ \operatorname{CE}$-projective, if and only if for all $n \in \mathbb{Z}$, all $\operatorname{Z}_n(X)$, $\operatorname{B}_n(X)$, $\operatorname{H}_n(X)$ and $\operatorname{C}_n(X)$ are projective left $R$-modules if and only if  it is isomorphic to a chain complex  of the form
	$$  \bigoplus_{n \in \mathbb{Z}} \operatorname{D}^n(P_n) \oplus  \bigoplus_{n \in \mathbb{Z}} \operatorname{S}^n(Q_n)$$ where  $P_n$ and $Q_n$ are projective left $R$-modules, for all  $n \in \mathbb{Z}$.
	
	Similarly,  $X$ is an injective object in $(\mathbf{C}(R); \mcE_{\operatorname{CE}})$,   called $\operatorname{CE}$-injective,   if and only if for all $n \in \mathbb{Z}$, all $\operatorname{Z}_n(X)$, $\operatorname{B}_n(X)$, $\operatorname{H}_n(X)$ and $\operatorname{C}_n(X)$ are injective left $R$-modules if and only if it is isomorphic to a chain complex of the form
	$$ \prod_{n \in \mathbb{Z}}\operatorname{D}^n(I_n) \oplus  \prod_{ n \in \mathbb{Z}}\operatorname{S}^n(E_n) $$ where  all $I_n$ and $E_n$ are injective left $R$-modules.
	
\end{proposition}

We let $\operatorname{CE-Proj}$ and $\operatorname{CE-Inj}$ be the classes of CE-projective and CE-injective chain complexes of left $R$-modules, respectively. Then $\operatorname{CE-Proj}= \mcE_{\operatorname{CE}} \operatorname{-Proj}$ and $\operatorname{CE-Inj}= \mcE_{\operatorname{CE}} \operatorname{-Inj}$.

\begin{definition}\label{def:ghost}
	A morphism $f $ in  $\mathbf{C}(R)$ is called a \emph{ghos}t map if $\operatorname{H}_n(f)=0$ for every $n \in \mathbb{Z}$.
\end{definition}
We denote by $\mathfrak{g}(\mathbf{C}(R))$ the class of all ghost morphisms in $\mathbf{C}(R)$ which is clearly an ideal. Recall from Remark \ref{dual.results} that if $\mcS\subseteq \mcA$ is a subset, the ideal left orthogonal to $\langle \mcS \rangle$ consists of the $\mcS$-coghost morphisms and is denoted by $\mathfrak{Cog}_{\mcS} = {^{\perp}} \langle \mcS \rangle$.
\begin{proposition}\label{prop:relative_ghost}
	Let  $\mathcal{S}:=\{ \operatorname{S}^n(R)\}_{n \in \mathbb{Z}}$ and $\mathcal{S}':=\{\operatorname{S}^n(E)\}_{n \in \mathbb{Z}}$ in $\mathbf{C}(R)$. We have
	$$ \mathfrak{g}_{\mathcal{S}}= \mathfrak{g} (\mathbf{C}(R)) =\mathfrak{Cog}_{\mathcal S'}.$$
\end{proposition}
\begin{proof}
	Let $f$ be a morphism in $\mathbf{C}(R)$. Since $\operatorname{S}^n(R)$ is a degree-wise projective chain complex, we have
	$$\Ext(\operatorname{S}^n(R),f)=\Hom_{\mathbf{K}(R)}(\operatorname{S}^{n-1}(R),f).$$
	On the other hand, one can easily show that
	$$\Hom_{\mathbf{K}(R)}(\operatorname{S}^{n-1}(R),f) \cong \operatorname{H}_{n-1}(f).$$
	Therefore, $f$ is $\mcS$-ghost if and only if $f$ is ghost as given in Definition \ref{def:ghost}.
	
	As for the second equality, since $\operatorname{S}^n(E)$ is a degree-wise injective chain complex,  $$\Ext(f,\operatorname{S}^n(R))=\Hom_{\mathbf{K}(R)}(f, \operatorname{S}^{n+1}(E)).$$
	Again, we have
	$$ \Hom_{\mathbf{K}(R)}(f, \operatorname{S}^{n+1}(E)) \cong \Hom(\operatorname{H}_{n+1}(f), E).$$
	Since $E$ is a cogenerator, we have that $f$ is $\mcS'$-coghost if and only if $f$ is ghost in $\mathbf{C}(R)$.
\end{proof}
The following is an immediate consequence of Proposition \ref{prop:relative_ghost}

\begin{corollary}\label{coghost}Let $f:X\to Y$ be a morphism in $\mathbf{C}(R)$. Then the following are equivalent:
	\begin{enumerate}[(i)]
		\item $f$ is a ghost map in $\mathbf{C}(R)$;
		\item $f$ is $\mathcal{E}_{\operatorname{CE}}$-phantom;
		\item $f$ is $\mathcal{E}_{\operatorname{CE}}$-cophantom;
		\item The pullback of the short exact sequence $0\to \Sigma^{-1}(Y)\to \operatorname{cone}(1_{\Sigma^{-1}(Y)})\to Y\to 0$ along $f$ is $\operatorname{CE}$-exact;
		\item  The pushout of the short exact sequence $0\to X\to \operatorname{cone}(1_{X})\to \Sigma X\to 0$ along $f$ is $\operatorname{CE}$-exact.
	\end{enumerate}
\end{corollary}
\begin{proof}
	The equivalence of the statements (i), (ii) and (iii) and the implications (ii)  $\Rightarrow  $ (iv) and (iii) $ \Rightarrow$ (v)  follow from Propositions \ref{prop:CE_proj_gener} and \ref{prop:relative_ghost}.
	For (iv) $\Rightarrow $(i), consider the pullback diagram
	$$\xymatrix{0\ar[r]&\Sigma^{-1}(Y)\ar[r]\ar@{=}[d]&H\ar[r]\ar[d]&X\ar[r]\ar[d]^f&0\\
		0\ar[r]&\Sigma^{-1}(Y)\ar[r]&\operatorname{cone}(1_{\Sigma^{-1}(Y)})\ar[r]&Y\ar[r]&0}$$ Since $0\to \Sigma^{-1}(Y)\to H\to X\to 0$ is CE-exact, for any integer $n$, every map $g:\operatorname{S}^n(R)\to X$ factors through $H$, and hence the composition $fg$ factors through $\operatorname{cone}(1_{\Sigma^{-1}(Y)})$. This shows that $fg$ is null homotopic, and so $\operatorname{H}_n(f)=0$.
	
	A similar argument can be applied to prove the implication (v) $ \Rightarrow$ (i).
\end{proof}

In the following, we show that $\mathfrak{g} (\mathbf{C}(R))$ is in fact in the middle of a complete ideal cotorsion triple whose left and right perpendicular ideals of $\mathfrak{g} (\mathbf{C}(R))$ have simpler forms.
\begin{proposition}\label{prop:cotorsion_triple_ghost} There exists a complete ideal cotorsion triple
	$$(\langle\operatorname{CE-Proj}\rangle, \mathfrak{g} (\mathbf{C}(R)), \langle\operatorname{CE-Inj}\rangle).$$
\end{proposition}
\begin{proof}
	The inclusion  $\langle\mbox{CE-Proj}\rangle \subseteq {}^\perp (\mathfrak{g} (\mathbf{C}(R))) $ follows from Proposition \ref{prop:relative_ghost} immediately.
		Conversely, by Proposition \ref{key1}, it is enough to show the inclusion  $ \mbox{Proj} (\mathbf{C}(R)) \star  \mbox{Sum}(\mcS) \subseteq  \mbox{CE-Proj}$. If $X\in  \mbox{Proj}(\mathbf{C}(R)) \star  \mbox{Sum}(\mcS)  $, then there exists a short exact sequence
	\begin{equation}\label{extension}
		\xymatrix{0\ar[r] & P \ar[r] & X \ar[r] & T \ar[r] & 0}
	\end{equation}
	where $P \in \mbox{Proj}(\mathbf{C}(R))$ and $T \in\mbox{Sum}(\mcS)  $. Since $P$ is an acyclic chain complex, it is easy to verify that the short exact sequence given in \eqref{extension} is CE-exact. Note that, all cycles, boundaries, homologies and cokernels of both $P$ and $T$ are projectives.  Therefore, $X $ is CE-projective, as well.
	
	The other part of the ideal cotorsion triple comes from Corollary \ref{coghost}, and the dual argument of above.	
\end{proof}
For the $n$-th power of the ghost ideal, $\mathfrak{g} (\mathbf{C}(R))^n$, we have the following result by Corollary \ref{finite}.
\begin{proposition} For any $n \geq 1$, there exists a complete ideal cotorsion triple
	$$(\langle \operatorname{add}( n\operatorname{-Filt(CE-Proj)} )\rangle, \mathfrak{g} (\mathbf{C}(R))^n, \langle \operatorname{add}(n\operatorname{-Cofilt(CE-Inj)} )\rangle).$$
\end{proposition}
Notice that in the previous proposition, we have used the fact that finite projective and inductive powers of an ideal are the same. However,  as the infinite inductive and projective powers of $\mathfrak{g} (\mathbf{C}(R))$ may not coincide,
apparently there is no reason for $\mathfrak{g} (\mathbf{C}(R))^{(\lambda)}$ to be a part of a cotorsion triple of ideals of the form just as given in the previous result.

\subsection{The Relation between ghost ideals in $\mathbf{C}(R)$, $\mathbf{D}(R)$ and $\mathbf{K}(R)$ }\label{ghost_K(R)}
Note that there exist the  canonical functors
\begin{equation}\label{chain-homotopy-derived}
	\xymatrix{\mathbf{C}(R) \ar@{->>}[r] & \mathbf{K}(R) \ar[r] & \mathbf{D}(R),}
\end{equation}
which are the identity on objects. The first one is the canonical quotient functor, and the second one is the localization functor. For a given morphism $f \in  \mathbf{C}(R)$, we denote its image in $\mathbf{K}(R)$ by $\overline{f}$.

For any integer  $n $, the $n$th homology functor $\operatorname{H}_n:\ \mathbf{C}(R) \longrightarrow \RMod$ is factorized over $\mathbf{K}(R)$ and $\mathbf{D}(R)$
\begin{equation}\label{homology_factorization}
	\operatorname{H}_n:\ \mathbf{C}(R) \longrightarrow \mathbf{K}(R) \longrightarrow \mathbf{D}(R) \longrightarrow \RMod.
\end{equation}

A morphism $\overline{f}$ in $\mathbf{K}(R)$ is called \textit{ghost} if it  induces  zero on homologies, that is, $\operatorname{H}_n(\overline{f})=0$ for every $n \in \mathbb{Z}$. Ghost morphisms in $\mathbf{K}(R)$ form an ideal, which  will be denoted  by $\mathfrak{g}(\mathbf{K}(R))$.
On the other hand,  Christensen~\cite[Section 8]{C} calls  \textit{ghost} a morphism $g$ in $\mathbf{D}(R)$ if  $\operatorname{H}_n(g)=0$, for every $n \in \mathbb{Z}$. They also form an ideal in $\mathbf{D}(R)$, which will be denoted by $\mathfrak{g}(\mathbf{D}(R))$.

Using the factorization of the homology functor given in \eqref{homology_factorization}, it is immediate that the composition  \eqref{chain-homotopy-derived} of functors induces the following natural transformations between finite powers of ghost  ideals
\begin{equation*}
	\xymatrix{ \mathfrak{g}(\mathbf{C}(R))^n \ar[r] & \mathfrak{g}(\mathbf{K}(R))^n  \ar[r] &\mathfrak{g} (\mathbf{D}(R))^n. }
\end{equation*}
The quotient functor $ \mathbf{C}(R) \longrightarrow \mathbf{K}(R)$ is full,  and therefore, the induced natural transformation
\begin{equation}\label{n-ghost}
	\xymatrix{ \mathfrak{g}(\mathbf{C}(R))^n \ar@{->>}[r] & \mathfrak{g}(\mathbf{K}(R))^n }
\end{equation}
is in fact an epimorphism.

We denote the class of all acyclic chain complexes of left $R$-modules by $\mbox{Acyc}$.
Recall that a chain complex $X$ of projective left $R$-modules is said to be $\mbox{dg}$-\textit{projective} if for every acyclic chain complex $Y$ of left $R$-modules, $\Hom_{\mathbf{K}(R)}(X, Y)=0$. Let $\operatorname{dg-Proj}$ denote the class of all $\mbox{dg}$-projective chain complexes in $\mathbf{C}(R)$. The crucial fact which will be used in sequel is that the  pair $( \operatorname{dg-Proj}, \mbox{Acyc})$ is a complete cotorsion pair in $\mathbf{C}(R)$.  Dually, the class $\mbox{dg-Inj}$ of   dg-injective chain complex of left $R$-modules is defined, and we have a  complete cotorsion pair $( \operatorname{Acyc}, \mbox{dg-Inj})$ in $\mathbf{C}(R)$;  see  \cite[Theorem 4.1.4, Example 4.1.8, Corollary 4.1.15]{EJ2011}.

It is well known that there exists a projective model structure on $\mathbf{C}(R)$ whose homotopy category is equivalent to $\mathbf{D}(R)$. In fact, this projective model structure on $\mathbf{C}(R)$ is induced from the triple $( \operatorname{dg-Proj}, \operatorname{Acyc}, \mathbf{C}(R) )$, see  \cite[Example 3.3]{Hov02}. This projective model structure   implies that the  composition of canonical functors $$ \theta: \xymatrix{\mathbf{K}(\mbox{dg-Proj})   \ar@{^{(}->}[r] & \mathbf{K}(R) \ar[r] & \mathbf{D}(R) }$$
is an equivalence of categories, where $\mathbf{K}(\mbox{dg-Proj})$ denotes the quotient category of $\mbox{dg-Proj}$ modulo homotopic to zero maps.
Note that the inverse of the equivalence $\theta: \mathbf{K}(\mbox{dg-Proj})  \longrightarrow \mathbf{D}(R)$ assigns to every chain complex $X$ of left $R$-modules a dg-projective cofibrant replacement $P_X$ of $X$: there exists a quasi-isomorphism $\sigma_X:P_X \longrightarrow X$ in $\mathbf{C}(R)$, where $P_X$ is a dg-projective chain complex. To every morphism $g: X \longrightarrow Y$ in $\mathbf{D}(R)$, $\theta^{-1}$ assigns  a morphism $\overline{g'}:P_X \longrightarrow P_Y$ in $\mathbf{K}(\mbox{dg-Proj})$ such that $\theta^{-1}(g)=\overline{g'}$ in such a way that the following diagram commutes  in $\mathbf{D}(R)$
$$
\xymatrix{P_X \ar[r]^{\sigma_X} \ar[d]_{\overline{g'}} & X \ar[d]^g \\
	P_Y \ar[r]_{\sigma_Y} \ar[r] & Y }
$$
As an aside, the morphisms $\sigma_X$ and $\sigma_Y$ are isomorphisms in $\mathbf{D}(R)$.

Note that  $\mbox{dg-Proj} $ is an extension closed subcategory of $\mathbf{C}(R)$, and therefore, is an exact subcategory  of $\mathbf{C}(R)$.  So we have
the following   commutative diagram
\begin{equation}\label{diagram:dgproj}
	\xymatrix{ \mbox{dg-Proj} \ar@{->>}[r] \ar@{^{(}->}[d] & \mathbf{K}(\mbox{dg-Proj}) \ar@{^{(}->}[d] \ar[rd]^{\cong}_\theta &\\
		\mathbf{C}(R) \ar@{->>}[r] &  \mathbf{K}(R) \ar[r]& \mathbf{D}(R),
	}
\end{equation}

We let $\mathfrak{g}(\mathbf{K}(\mbox{dg-Proj}))$ and  $\mathfrak{g}(\mbox{dg-Proj} )$ denote the classes  of ghost morphisms in $\mathbf{K}(R)$ and $\mathbf{C}(R)$, respectively,  between dg-projective chain complexes. It is immediate that both classes $\mathfrak{g}(\mathbf{K}(\mbox{dg-Proj}))$ and  $\mathfrak{g}(\mbox{dg-Proj} )$ are ideals in $\mathbf{K}(\mbox{dg-Proj})$ and $\mbox{dg-Proj}$, respectively.

Clearly, if $\overline{f}$ is a ghost morphism in $\mathfrak{g}(\mathbf{K}(\mbox{dg-Proj})) $, then its image $\theta ( \overline{f})$ in $\mathbf{D}(R)$ is also a ghost morphism.  Conversely, if $g:X \longrightarrow Y$ is a ghost morphism in $\mathbf{D}(R)$, then there exists a unique morphism $\overline{g'}: P_X  \longrightarrow P_Y$ in $\mathbf{K}(\mbox{dg-Proj})$ such that $\sigma_Y \overline{g'}=g \sigma_X$ in $\mathbf{D}(R)$. Applying the homology functor $\operatorname{H}_*$ to the composition,  we see that $\overline{g'}$ is a ghost morphism in $ \mathbf{K}(\mbox{dg-Proj})$. Therefore, the equivalence $\theta$ induces the isomorphism $\mathfrak{g}(\mathbf{K}(\mbox{dg-Proj})) \cong \mathfrak{g}(\mathbf{D}(R)) $ of ideals.

Using the commutative diagram given in  \eqref{diagram:dgproj} and the composition \eqref{homology_factorization}, we have that if $f$ is a morphism in $\mathfrak{g}(\mbox{dg-Proj} )$, then its image in $ \mathbf{K}(\mbox{dg-Proj}) \cong \mathbf{D}(R)$   is ghost.
Therefore,  the diagram \eqref{diagram:dgproj} induces the following commutative diagram of ghost ideals

\begin{equation}\label{diagram:ghosts}
	\xymatrix{\mathfrak{g}(\mbox{dg-Proj}) \ar@{->>}[r] \ar@{^{(}->}[d] & \mathfrak{g}(\mathbf{K}(\mbox{dg-Proj})) \ar@{^{(}->}[d] \ar[rd]^{\cong} &\\
		\mathfrak{g}(\mathbf{C}(R)) \ar@{->>}[r] & \mathfrak{g}(\mathbf{K}(R)) \ar[r]& \mathfrak{g}(\mathbf{D}(R))
	}
\end{equation}

Note that the $n$th sphere chain complex $\operatorname{S}^n(R)$ is a dg-projective chain complex, so $\mcS:=\{ \operatorname{S}^n(R)\}_{n \in \mathbb{Z}} $ is also a set of objects in the exact category  $\mbox{dg-Proj}$. Therefore  in the exact category $\operatorname{dg-Proj}$, the $\mcS$-ghost ideal
	$\mathfrak{g}_{\mcS}$ is just the ideal $\mathfrak{g} (\operatorname{dg-Proj}).$

\subsection{The Generalized Generating Hypothesis} As already mentioned, the following triple is a  complete cotorsion triple in $\mathbf{C}(R)$
$$(\mbox{dg-Proj}, \mbox{Acyc}, \mbox{dg-Inj} ).$$
By \cite[Theorem 28]{FGHT}, it  yields  a complete ideal cotorsion triple in $\mathbf{C}(R)$
$$(\langle\mbox{dg-Proj}\rangle,\langle\mbox{Acyc}\rangle, \langle\mbox{dg-Inj} \rangle).$$
On the other hand, it is  well known that
$$\mcS^\perp=\mbox{Acyc}= {}^\perp \mcS',$$
where
$\mcS:=\{ \operatorname{S}^n(R)\}_{n \in \mathbb{Z}}$ and $\mcS':=\{\operatorname{S}^n(E)\}_{n \in \mathbb{Z}}$.
Therefore, the chain for the inductive powers of $\mathfrak{g}(\mathbf{C}(R))$ given in  \eqref{chain:ghost}  is bounded below by  the object ideal $ \langle \mbox{Acyc} \rangle$, that is,
$$
\langle\mbox{Acyc}\rangle \subseteq \cdots \subseteq  \mathfrak{g}(\mathbf{C}(R))^{(\lambda)} \subseteq \cdots \mathfrak{g}(\mathbf{C}(R))^2 \subseteq
\mathfrak{g}(\mathbf{C}(R)).  $$
Since $\mathbf{C}(R)$ is a locally finitely presentable Grothendieck category, Theorem \ref{EFHS} gives the following result.
\begin{proposition}
	$\omega$-{\rm GGH}{\rm(}$\mathfrak{g}(\mathbf{C}(R))\rm{)}$ always holds.
\end{proposition}
The next proposition relates the Generalized $1$-Generating Hypothesis for $\mathfrak{g}(\mathbf{C}(R))$ with the triviality of the ideal $\mathfrak{g}(\mathbf{K}(R))$.
\begin{proposition}\label{proposition:generating_homotopy}
		$1$-{\rm GGH}{\rm(}$\mathfrak{g}(\mathbf{C}(R))\rm{)}$ holds if and only if  $\mathfrak{g}(\mathbf{K}(R))=0$.
\end{proposition}
\begin{proof}
	Suppose that $\mathfrak{g}(\mathbf{K}(R))=0$. Let $f:X \rightarrow Y$ be a morphism in $\mathfrak{g}(\mathbf{C}(R))$.  Using  the diagram \eqref{diagram:ghosts},  its image $\overline{f}$ in $\mathbf{K}(R)$ belongs to $\mathfrak{g}(\mathbf{K}(R))$. By assumption, $\overline{f}$ is homotopic to zero. From  \cite[ Remark 9.11]{B}, we know that $f$ factors through the cone of $1_X$, which is contractible. So $\mathfrak{g}(\mathbf{C}(R))  \subseteq \langle\mbox{Acyc}\rangle$.

	Conversely, if $\mathfrak{g}(\mathbf{C}(R))=\langle\mbox{Acyc}\rangle$, then
	$\langle \mbox{dg-Proj} \rangle= {}^\perp \langle\mbox{Acyc}\rangle =    {}^\perp \mathfrak{g}(\mathbf{C}(R))= \langle  \mbox{CE-Proj} \rangle.$   It implies that a  projective resolution of any left  $R$-module is CE-projective, and hence, every left $R$-module is  projective. So  every acyclic chain complex of left $R$-modules  is contractible. As indicated before,  if $\overline{f}$ is a ghost morphism in $\mathbf{K}(R)$, then $f$ is ghost in $\mathbf{C}(R)$. But  it is factorized over a contractible chain complex, so $\overline{f}=0$.
\end{proof}
Note that $\mbox{dg-Proj}$  is a Frobenius exact category with a  complete cotorsion triple
$$(\mbox{dg-Proj}, \mbox{Proj}(\mathbf{C}(R)), \mbox{dg-Proj} ),$$ where $\operatorname{Proj}(\mathbf{C}(R) )$ is the class of projective objects in $\mathbf{C}(R)$ which consists of  contractible chain complexes of projective left $R$-modules.
Again, by using \cite[Theorem 28]{FGHT},
$$( \langle \mbox{dg-Proj} \rangle, \langle \mbox{Proj}(\mathbf{C}(R)) \rangle,  \langle \mbox{dg-Proj}  \rangle )$$
is a complete ideal cotorsion triple in $\mbox{dg-Proj}$. Similarly,
$\mcS^\perp=\mbox{Proj}(\mathbf{C}(R))$ in $\mbox{dg-Proj}$, where $\mcS:=\{ \operatorname{S}^n(R)\}_{n \in \mathbb{Z}}$.  Therefore, the chain \eqref{chain:ghost} for  $\mathfrak{g}( \mbox{dg-Proj})$ is bounded below by  $\langle \mbox{Proj}(\mathbf{C}(R)) \rangle$
$$
\langle\mbox{Proj}(\mathbf{C}(R))\rangle \subseteq \cdots \subseteq
\mathfrak{g}(\mbox{dg-Proj})^{(\lambda)}\subseteq \cdots
\subseteq \mathfrak{g}(\mbox{dg-Proj})^2\subseteq \mathfrak{g}( \mbox{dg-Proj}). $$

\begin{proposition}\label{prop:generating_Derived}
	For any integer $n \geq 1$, 	$n$-{\rm GGH}{\rm(}$\mathfrak{g}(\operatorname{dg-Proj})\rm{)}$ holds  if and only if $\mathfrak{g}(\mathbf{D}(R))^n=0$
\end{proposition}
\begin{proof}
	We only prove the statement for $n=1$.  Suppose that $\mathfrak{g}(\mathbf{D}(R))=0$. Let $f$ be a map in  ${\mathfrak{g}(\mbox{dg-Proj})}$.  From the diagram given in \eqref{diagram:ghosts}, the image $\overline{f}$ of $f$ in  $\mathbf{K}(\mbox{dg-Proj})$ belongs to $\mathfrak{g}(\mathbf{K}(\mbox{dg-Proj}))\cong \mathfrak{g}(\mathbf{D}(R))=0$, and therefore, $\overline{f}$ is homotopic to zero. From \cite[Remark 9.11]{B}, $f$ factors through the cone of the domain of $f$ which is acyclic and dg-projective, and therefore, is a projective chain complex.
	The other side follows from the diagram \eqref{diagram:ghosts}  and the fact that the intersection $\mbox{dg-Proj}\cap \mbox{Acyc}$ is exactly the class $\mbox{Proj}(\mathbf{C}(R))$.
\end{proof}
%

We should point out that  for a positive integer $n$, the  Generalized $n$-Generating Hypothesis for both ideals $\mathfrak{g}(\mathbf{C}(R)) $ and $\mathfrak{g}(\mbox{dg-Proj})$ forces to impose certain conditions on the ring $R$, namely, finiteness of the left global dimension of $R$. In fact, Hovey and Lockridge showed in \cite[Lemma 1.4]{HL} that if $\mathfrak{g}(\mathbf{D}(R))^{n+1}=0 $, then the left global dimension of $R$, denoted by $\operatorname{l.gl.dim}(R) $,  is $\leq n$. On the other hand, by \cite[Theorem 8.3]{C}, we know that if $\operatorname{l.gl.dim}(R)  \leq n$, then $\mathfrak{g}(\mathbf{D}(R))^{n+1}=0$. Now, we are able to obtain these results in our setting by using tools from ideal approximation theory. For this end, we need a filtration  version of the statement \cite[Theorem 8.3]{C}.

\begin{proposition}\label{filt-CE}
	Let $n \geq 0$ and  $X$ be a dg-projective chain complex  of left $R$-modules. If  for every $i \in \mathbb{Z}$, the  projective dimensions of $\operatorname{H}_i(X)$ and $\operatorname{Z}_i(X)$ are $\leq n$,
	then $X$ is a direct summand of an object in
	$(n+1)\operatorname{-Filt} (\operatorname{CE-Proj})$.
\end{proposition}
\begin{proof}
	We will apply induction on $n$.  Let $n=0$. Suppose that $X$ is a dg-projective chain complex of left  $R$-modules with projective homologies and cycles. For every $i \in \mathbb{Z}$, consider the induced   short exact sequences
	$$\xymatrix{0 \ar[r] & \operatorname{B}_i(X) \ar[r] & \operatorname{Z}_i(X) \ar[r] & \operatorname{H}_i(X) \ar[r] & 0},$$
	$$\xymatrix{0 \ar[r] & \operatorname{H}_{i+1}(X) \ar[r] & \operatorname{C}_{i+1}(X) \ar[r] & \operatorname{B}_i(X) \ar[r] & 0}.$$
	\sloppy By assumption, the first short exact sequence splits, and therefore,   $\operatorname{B}_i(X)$  is a projective left $R$-module. Similarly, we have that $\operatorname{H}_i(X)$ is a projective left $R$-module.  In other words,  ${X \in\mbox{CE-Proj}}= \operatorname{add}(1\mbox{-Filt}(\mbox{CE-Proj}))$.
	
	Suppose that the statement is true for some $n \geq 0$. Let  $X \in \mbox{dg-Proj}$ with projective dimensions $\leq n+1$ for  $\operatorname{H}_i(X) $ and $\operatorname{Z}_i(X)$ for every $i\in \mathbb{Z}$. Consider a CE-short exact sequence of the form
	$$\xymatrix{0 \ar[r] & T\ar[r] & P \ar[r] & X \ar[r] & 0}  $$
	where $P$ is  a CE-projective chain complex. It is easy to verify  that $T$ is  a dg-projective chain complex with projective dimensions   $\leq n$ for $ \operatorname{H}_i(T)$ and $\mbox{Z}_i(T) $ for every $i\in \mathbb{Z}$.  By hypothesis, $T \in \operatorname{add}((n+1)\mbox{-Filt(CE-Proj)})$. Consider the  pushout diagram
	$$\xymatrix{& 0 \ar[d]&0 \ar[d] & &\\
		0 \ar[r] & T\ar[r] \ar[d] & P  \ar[d]\ar[r] & X \ar@{=}[d] \ar[r] & 0\\
		0 \ar[r]& \operatorname{cone}(1_T) \ar[r] \ar[d] & Z \ar[r] \ar[d]& X \ar[r] & 0\\
		& \Sigma T \ar@{=}[r] \ar[d] & \Sigma T \ar[d] & &\\
		& 0 & 0 & & &}  $$
	Since $\operatorname{cone}(1_T)$ is an  acyclic chain complex, and $X$ is dg-projective, the second row splits, and therefore,  $X$ is a direct summand of $Z$. On the other hand, $$Z \in \mbox{CE-Proj} \star \operatorname{add}((n+1)\mbox{-Filt(CE-Proj)})= \operatorname{add}((n+2)\mbox{-Filt(CE-Proj)}).$$
\end{proof}

\begin{theorem}\label{theo:global_dim_ghost_dim}
	Let $n\geq 0$. The following are equivalent:
	\begin{enumerate}[(i)]
		\item $\mathfrak{g}(\mathbf{D}(R))^{n+1}=0$;
		\item 	$(n+1)$-{\rm GGH}{\rm(}$\mathfrak{g}(\mathbf{C}(R))\rm{)}$ holds; 
		\item $\mathfrak{g}(\operatorname{dg-Proj})^{n+1}=\langle\operatorname{Proj}(\mathbf{C}(R))\rangle$;
		\item $\operatorname{l.gl.dim}(R) \leq n$.
	\end{enumerate}
\end{theorem}

\begin{proof}
	(i $\Leftrightarrow$ iii) Follows from Proposition \ref{prop:generating_Derived}.
	
	(iii$ \Rightarrow $iv) Suppose $\mathfrak{g}(\mbox{dg-Proj})^{n+1}=\langle\mbox{Proj}(\mathbf{C}(R))\rangle$.
	For any left $R$-module $M$, consider its   projective resolution  $P_* \rightarrow M$. It is known that $P_*$ is a dg-projective chain complex.
	Besides, the canonical   morphism
	$$
	\xymatrix{ P_*: & \cdots \ar[r] & P_{n+2} \ar@{=}[d]  \ar[r] &  P_{n+1} \ar[r] \ar@{=}[d] & P_{n} \ar[r] \ar[d] &  \cdots  \ar[r]  & P_0 \ar[d] \ar[r] & 0 \ar[d]\\
		&	\cdots \ar[r] & P_{n+2} \ar[r] & P_{n+1} \ar[r] & 0 \ar[r] & \cdots \ar[r]   &0  \ar[r] &0 }	
	$$
	belongs to $\mathfrak{g}(\mbox{dg-Proj})^{n+1}$, so it is homotopic to zero. In particular, there exist morphisms $s_{n+1}: P_{n+1} \rightarrow P_{n+2}$ and $s_n:P_n \rightarrow P_{n+1}$ such that $ s_{n}\circ d +  d \circ s_{n+1}= 1_{P_{n+1}}$. If we restrict it to  $\mbox{Z}_{n+1}(P_*)$, then   $\mbox{Z}_{n+1}(P_*) \hookrightarrow P_{n+1}$ is a section of $P_{n+1}$, and therefore, the projective dimension of $M$ is $ \leq n+1$. So we can assume that $P_i=0$ for every $i >n+1$, and $s_{n+1}=0$. Then the homotopy  leads to $1_{P_{n+1}}=s_{n} \circ d$. This means that $P_{n+1}$ is a direct summand of $P_{n}$ which implies that the projective dimension of $M$ is in fact $ \leq n$.
	
	(iv $ \Rightarrow $ii) By Proposition \ref{filt-CE}, we have  $\mbox{dg-Proj} = \mbox{add} ((n+1) \mbox{-Filt}(\mbox{CE-Proj}))$, and therefore, $\langle \mbox{Acyc} \rangle=\langle \mbox{dg-Proj} \rangle^\perp = \langle \mbox{add} ((n+1) \mbox{-Filt}(\mbox{CE-Proj})) \rangle^\perp =\mathfrak{g}(\mathbf{C}(R))^{n+1}$.
	
	(ii $ \Rightarrow $ iii) If $f: X \rightarrow Y \in \mathfrak{g}(\mbox{dg-Proj})^{n+1}$, then by assumption, it factorizes over an acyclic chain complex. However, any morphism from a dg-projective to an acyclic chain complex is null-homotopic. Hence, $f $ is homotopic to zero. Again, $f$ factorizes over  $\operatorname{cone}(1_X)$, which is a projective chain complex.
\end{proof}

\subsection{The Generating Hypothesis for $\mathbf{D}(R)$}  As already mentioned in Remark \ref{remark:termin_ghost}, the original algebraic versions of the Generating Hypothesis in both  $\mathbf{D}(R)$ and $kG\mbox{-} \underline{\mbox{Mod}}$ were stated for    ghost morphisms between compact objects. Namely, they ask when the ideals $\mathfrak{g}( \mathbf{D}^c(R))$ and $\mathfrak{g}( kG\mbox{-} \underline{\mbox{mod}})$ are trivial, where $ \mathbf{D}^c(R)$ and $kG\mbox{-} \underline{\mbox{mod}}$ are the subcategories of compact objects in $\mathbf{D}(R)$ and $kG\mbox{-} \underline{\mbox{Mod}}$, respectively. In case of the stable module category of a finite $p$-group, the authors in \cite{BCCM07} show that $\mathfrak{g}( kG\mbox{-} \underline{\mbox{mod}})$ is trivial if and only if $\mathfrak{g}( kG\mbox{-} \underline{\mbox{Mod}})$ is trivial. As for $\mathbf{D}(R)$, if $R$ is left coherent, then $\mathfrak{g}( \mathbf{D}^c(R))=0$ if and only if $R$ is von Neumann regular; see \cite[Theorem 3.1]{HLP07}.

The  analogue of compactness  in an abelian category is the notion of a finitely presented object.  We end this section by briefly showing that the result~\cite[Theorem 3.1]{HLP07} stated for the derived category of a ring can be obtained through ghost morphisms for chain complexes and by using techniques from ideal approximation theory.  Mostly we only write the statement for the finitely presented case without any proof, since the techniques used in the unbounded case work well under the coherence condition.

We let $\mathbf{C}^{\operatorname{fp}}(R)$ denote the subcategory of finitely presented chain complexes of left $R$-modules. Note that a chain complex $X$ is finitely presented  in $\mathbf{C}(R)$ if and only if it is a bounded chain complex of finitely presented left $R$-modules. Besides, if $R$ is left coherent, then $\mathbf{C}^{\operatorname{fp}}(R)$ is an abelian subcategory of $\mathbf{C}(R)$.

We let $\mathbf{C}^b(\operatorname{proj}) $ denote the class of bounded chain complexes of finitely generated projective left $R$-modules, and
$$ \mbox{acyc}^b:= \mathbf{C}^{\operatorname{fp}}(R)  \cap \mbox{Acyc}.$$
The following result can be proved easily.
\begin{proposition}\label{prop:fp_cot}
	If $R$ is left coherent, then there exists a complete cotorsion pair in $\mathbf{C}^{\operatorname{fp}}(R)$
	$$( \mathbf{C}^b(\operatorname{proj}), \operatorname{acyc}^b)$$
	generated by $\mcS:=\{\operatorname{S}^n(R)\}_{ n \in \mathbb{Z}}$.
\end{proposition}

We let $\mathfrak{g}( \mathbf{C}^{\operatorname{fp}}(R))$ denote  the ideal of ghost morphisms between finitely presented chain complexes of left $R$-modules.
Similar to  the unbounded case, if $R$ is left-coherent, $\mcS:=\{\operatorname{S}^n(R)\}_{ n \in \mathbb{Z}}$ and $\langle\mcS\rangle^\perp=\mathfrak{g}_{\mathcal S}$ denotes the $\mathcal S$-ghost ideal in $\mathbf{C}^{\operatorname{fp}}(R) $,  then
$\mathfrak{g}_{\mathcal S}=\mathfrak{g}( \mathbf{C}^{\operatorname{fp}}(R)),$
and if $1$-Generating hypothesis  for $\mathfrak{g}( \mathbf{C}^{\operatorname{fp}}(R))$ is satisfied, then $\mathfrak{g}( \mathbf{C}^{\operatorname{fp}}(R))= \langle \operatorname{acyc}^b \rangle$.

\begin{proposition}\label{prop:fp_ideal}
	If $R$ is left coherent, then there exists a complete ideal cotorsion pair
	$$ (\langle\operatorname{CE-proj}^b\rangle,  \mathfrak{g} ( \mathbf{C}^{\operatorname{fp}}(R)))$$
	in $\mathbf{C}^{\operatorname{fp}}(R)$, where $\operatorname{CE-proj}^b = \operatorname{CE-Proj} \cap \mathbf{C}^{\operatorname{fp}}(R)$.
\end{proposition}
\begin{proof}
	We only show that $\mathfrak{g} ( \mathbf{C}^{\operatorname{fp}}(R))$ is an object-special preenveloping ideal with object-cosyzgy $\operatorname{CE-proj}^b $, since the rest is similar to the proof of Proposition  \ref{prop:cotorsion_triple_ghost}.  Let $X \in \mathbf{C}^{\operatorname{fp}}(R) $. Since $R$ is left coherent,   $\operatorname{Z}_n(X)$ is finitely presented for every $n \in \mathbb{Z}$. So for every $n \in \mathbb{Z}$, there exist finitely presented projective left $R$-modules
	$P_n \twoheadrightarrow \mbox{Z}_n(X)$ and $Q_n \twoheadrightarrow X_n$. They induce  an epimorphism
	$$f: \bigoplus_{n \in \mathbb{Z}} \operatorname{S}^n(P_n) \oplus  \operatorname{D}^n(Q_n) \twoheadrightarrow X.$$
	\sloppy It is in fact a CE-epimorphism. By assumption, $X$ is a bounded chain complex, and therefore, the sum ${P:=\bigoplus_{n \in \mathbb{Z}} \operatorname{S}^n(P_n) \oplus  \operatorname{D}^n(Q_n)}$ is a finite sum. So $ P \in \mbox{CE-proj}^b$.
	\end{proof}
\begin{proposition}\label{prop:von neuman}
	Suppose that $R$ is left coherent. The following are equivalent.
	\begin{enumerate}[(i)]
		\item  $\mathfrak{g}( \mathbf{C}^{\operatorname{fp}}(R))= \langle \operatorname{acyc}^b \rangle$;
		\item $R$ is von Neumann regular.
	\end{enumerate}
\end{proposition}
\begin{proof}
	(i$ \Rightarrow $ii)
	Since $\mathfrak{g}( \mathbf{C}^{\operatorname{fp}}(R))= \langle \operatorname{acyc}^b \rangle$, by Propositions \ref{prop:fp_cot}  and \ref{prop:fp_ideal}, we have
	$$\langle\operatorname{CE-proj}^b\rangle ={}^\perp \mathfrak{g}( \mathbf{C}^{\operatorname{fp}}(R))= {}^\perp \langle \operatorname{acyc}^b \rangle= \langle \mathbf{C}^b(\operatorname{proj}) \rangle.$$
	Let $M$ be a finitely presented left $R$-module with a partial projective resolution $P_1 \rightarrow P_0 \rightarrow M \rightarrow 0$, where   $P_1$ and $P_0$ are finitely generated projective left $R$-modules. The chain complex
	$$\xymatrix{\cdots \ar[r] & 0 \ar[r] & P_1 \ar[r] & P_0 \ar[r] & 0 \ar[r] & \cdots}$$
	belongs to $\mathbf{C}^b(\operatorname{proj}) $ and therefore, it is CE-projective. So its  $0$th homology $M$  is projective. It means that every finitely presented left $R$-module is projective, so $R$ is von Neumann regular.
	
	(ii$ \Rightarrow $i) By assumption, every finitely presented left $R$-module is projective. So $\mathbf{C}^b(\operatorname{proj}) =\mbox{CE-proj}^b$. The satement follows from Proposition~\ref{prop:fp_cot} and Proposition~\ref{prop:fp_ideal} immediately.
\end{proof}

Using similar methods applied in Propositions \ref{prop:generating_Derived} and  \ref{filt-CE}, we can state the finitely presented version of Theorem \ref{theo:global_dim_ghost_dim} for the case $n=1$.
\begin{theorem}\label{GH}
	Suppose that $R$ is left coherent. Then the following are equivalent:
	\begin{enumerate}[(i)]
		
		\item $\mathfrak{g}(\mathbf{D}^c(R))=0$;
		\item  $\mathfrak{g}( \mathbf{C}^{\operatorname{fp}}(R))= \langle \operatorname{acyc}^b \rangle$;
        \item $\mathfrak{g}(\mathbf{C}^b(\operatorname{proj}))= \langle \operatorname{proj}(\mathbf{C}(R)) \rangle$;
		\item $R$ is von Neumann regular,
		
	\end{enumerate}
	where $\operatorname{proj}(\mathbf{C}(R))$ denotes the class of finitely generated projective chain complexes.
\end{theorem}

\begin{remark}
	{\rm The equivalence (i$\Leftrightarrow$iv) was already proved in \cite[Theorem 3.1]{HLP07}. However, we, as well as reproving with a different approach, complete the whole picture including ghost morphisms in the category of chain complexes.}
\end{remark}

\section{The FP-ghost ideal}
\subsection{The FP-ghost ideal} In this section, we consider the $\mcS\operatorname{-ghost}$ ideal  $\mathfrak{g}_{\Rmod}$ in $R\operatorname{-Mod}$ with $\mcS$ the set of isomorphism classes of finitely presented left $R$-modules. We call this ideal the \emph{FP-ghost ideal}. It is already known as the ideal of \emph{pure cophantom morphism,} and is usually denoted by $\Psi$ (see \cite{FGHT}), so we keep the notation $\Psi$ for the FP-ghost ideal  $\mathfrak{g}_{\Rmod}$ in $R \operatorname{-Mod}$. Note that  $\mbox{Add}(\mcS) =  R \operatorname{-PProj}$ is the class of pure projective left $R$-modules,  and therefore, $\langle R \operatorname{-PProj}\rangle \subseteq {}^\perp \Psi$. By Proposition \ref{key1}, the ideal $\Psi$ is an object-special preenveloping ideal in $R \operatorname{-Mod}$ with a coszygy subcategory $R \operatorname{-PProj},$ so for a given left $R$-module $M$,  there is a short exact sequence in $R \operatorname{-Mod}$ of the form
$$\xymatrix{0\ar[r]&M\ar[r]^f&X\ar[r]&N\ar[r]&0,}$$
where $f$ is an FP-ghost morphism, and $N$ is a pure-projective left $R$-module. Moreover, by Corollary~\ref{TFH5}, the pair
$(  \langle \operatorname{add}(R\textmd{-Proj}\star \lambda \mbox{-Filt} (R\textmd{-PProj}) )\rangle,  \Psi^{(\lambda)})$
is a complete ideal cotorsion pair, where $R\textmd{-Proj}$ denotes the class of projective left $R$-modules.

Recall that a left $R$-module $M$ is called FP-\textit{injective} \cite{St1} (or {\rm absolutely pure} \cite{M1}) if $\Ext^1(S,M)=0$ for all finitely presented left $R$-module $S$, that is,
$M \in \mcS^\perp.$ A left $R$-module $N$ is called FP-\emph{projective} \cite{MD} if $\Ext^1(N,M)=0$ for every $FP$-injective left $R$-module $M$, that is, $N \in {}^\perp(\mcS^\perp)$. We let  $\mathcal{FP}$ and $\mathcal{FI}$ denote the classes of FP-projective and FP-injective left  $R$-modules, respectively. Then it is well known that $(\mathcal{FP},\mathcal{FI})$ is a complete cotorsion pair in $R  \operatorname{-Mod}$. Replacing $\Psi$ in the chain (\ref{chain:ghost}), we have the following decreasing chain of ideals
$$   \langle\mathcal{FI}\rangle \subseteq \cdots \subseteq \Psi^{(\lambda)} \subseteq \cdots  \subseteq \Psi^2 \subseteq \Psi.$$

As an immediate application of Theorem \ref{EFHS}, we have the following result.
\begin{proposition}\label{TFH6}
	$\Psi^{(\omega)}=\langle\mathcal{FI}\rangle$, i.e., $\omega$-{\rm GGH}{\rm ($\Psi$)} holds.
\end{proposition}

\subsection{ Generalized $n$-Generating Hypothesis $n$-{\rm GGH}{\rm ($\Psi$)}}
\begin{lemma}\label{Lem61}
	If every pure projective module has an FP-injective cosyzygy which is still pure projective, then every module in  $\operatorname{add}(n\mbox{-}\Filt (R\operatorname{-PProj}))$ has an $\operatorname{FP}$-injective cosyzygy which is in $\operatorname{add}(n\mbox{-}\Filt (R\operatorname{-PProj}))$. In other words, $\operatorname{add}(n\mbox{-}\Filt (R\operatorname{-PProj}))$ is invariant under $\operatorname{FP}$-injective cosyzygies.
\end{lemma}

The proof of the above lemma is a dual argument of that in \cite[Lemma 9.6]{FH}, so we omit it.

\begin{theorem}\label{TFH7}
	Suppose that the ring $R$ satisfies the condition that an $\operatorname{FP}$-injective cosyzygy of a pure projective module is still pure projective. If every $\operatorname{FP}$-projective module $M$ has a resolution
	$$\xymatrix{0\ar[r]&P_{n-1}\ar[r]&\cdots\ar[r]&P_1\ar[r]&P_0\ar[r]&M\ar[r]&0}$$
	with each $P_i$ a pure projective module, and $n\geq 1$, then $n$-{\rm GGH}{\rm ($\Psi$)} holds.
\end{theorem}

\begin{proof} Since $(  \langle \operatorname{add}(R\textmd{-Proj}\star n \mbox{-Filt} (R\textmd{-PProj}) )\rangle,  \Psi^n)$
	is a complete ideal cotorsion pair, and $(\mathcal{FP},\mathcal{FI})$ is a complete cotorsion pair in $R  \operatorname{-Mod}$, to prove this result is equivalent to check that $\operatorname{add}(R\textmd{-Proj}\star n \mbox{-Filt} (R\textmd{-PProj}) )=\mathcal{FP}$.
	The inclusion $\operatorname{add}(R\textmd{-Proj}\star n \mbox{-Filt} (R\textmd{-PProj}) )\subseteq\mathcal{FP}$ is clear.
	For the other inclusion $\mathcal{FP}\subseteq\operatorname{add}(R\textmd{-Proj}\star n \mbox{-Filt} (R\textmd{-PProj}) )$, the case $n=1$ is a tautology.
	
	Let $M$ be an FP-projective module $M$, and suppose that there is an exact sequence $0\to P_1\to P_0\to M\to0$ such that $P_0$ and $P_1$ are pure projective modules. Then there is an FP-injective preenvelope of $P_1$, $f:P_1\to I$, such that the FP-injective cosyzygy $P$ is a pure projective module. Consider the pushout diagram
	$$\xymatrix{&0\ar[d]&0\ar[d]&&\\
		0\ar[r]&P_1\ar[r]\ar[d]_f&P_0\ar[r]\ar[d]&M\ar[r]\ar@{=}[d]&0\\
		0\ar[r]&I\ar[r]\ar[d]&X\ar[r]\ar[d]&M\ar[r]&0\\
		&P\ar@{=}[r]\ar[d]&P\ar[d]&&\\
		&0&0&&}$$
	The middle row is trivial since $M$ is FP-projective and $I$ is FP-injective, and hence $M$ is a direct summand of $X$ which belongs to $2\mbox{-Filt} (R\textmd{-PProj})$. Therefore, $\mathcal{FP}=2\mbox{-Filt} (R\textmd{-PProj})\subseteq\operatorname{add}(R\textmd{-Proj}\star 2\mbox{-Filt} (R\textmd{-PProj}) )$. Now employing Lemma \ref{Lem61}, and by induction, we have that, if $M$ has a resolution $$\xymatrix{0\ar[r]&P_{n-1}\ar[r]&\cdots\ar[r]&P_1\ar[r]&P_0\ar[r]&M\ar[r]&0}$$
	with each $P_i$ a pure projective module, then $M$ belongs to $\operatorname{add}(n\mbox{-Filt} (R\textmd{-PProj}))$, and so the result follows.
\end{proof}

In practice, it may not be easy to check if every FP-projective module $M$ has a resolution
$$\xymatrix{0\ar[r]&P_{n}\ar[r]&\cdots\ar[r]&P_1\ar[r]&P_0\ar[r]&M\ar[r]&0}$$
with each $P_i$ a pure projective module. But since every FP-projective module has a filtration by finitely presented modules, and a module possessing a filtration by projective modules is still projective, one then has the following result (see also \cite[Proposition 3]{auslander}):

\begin{proposition}
	If every finitely presented module has a projective resolution of length no more than $n-1$, where $n$ is a positive integer, then every $\operatorname{FP}$-projective module has a projective resolution of length no more than $n-1$.
\end{proposition}

Recall that a ring $R$ is said to have pure global dimension at most $n$ if every module has pure projective pure resolution of lenght at most $n$.

If $R$ is an artin algebra, then the class of FP-injective modules coincides with the class of injective modules. In this case, the cosyzygy of a finitely presented module is still finitely presented, and so the cosyzygy of a pure projective module is still pure projective. Therefore the following result is an immediate consequence of Theorem \ref{TFH7}.

\begin{corollary}\label{n-GH-FP-ghost1}
	Let $R$ be an artin algebra, and $n$ a positive integer. Then $n$-{\rm GGH}{\rm ($\Psi$)} holds provided one of the following two conditions holds:
	\begin{enumerate}[(i)]
		\item $R$ has global dimension less than $n$;
		\item $R$ has pure global dimension less than $n$.
	\end{enumerate}
\end{corollary}

Note that if  $R$ is a left artinian ring, then $R$ is a semiprimary left noetherian ring, so the Jacobson radical $J=J(R)$ is nilpotent, and every simple module is finitely generated since every maximal left ideal is finitely generated.

\begin{proposition}\label{n-GH-FP-ghost2}
	If $R$ is a left artinian ring with $J^n=0$. then $n$-{\rm GGH}{\rm ($\Psi$)} holds.
\end{proposition}

\begin{proof} Let $M$ be a left $R$-module, and consider the {\rm Loewy series}
	$$\xymatrix{M=M_0\supseteq JM\supseteq J^2M\supseteq\cdots\supseteq J^{n-1}M\supseteq J^nM=0.}$$
	Each of the factors are semisimple. Since every simple module is finitely generated, each factor $J^{i+1}M/J^iM$ is pure projective. Therefore $M\in\operatorname{add}(n\mbox{-Filt} (R\textmd{-PProj}))$, and so $n$-{\rm GGH}{\rm ($\Psi$)} holds.
\end{proof}

\subsection{The dual of Xu's result} Now we are in a position to present the proof of the dual of Xu's result \cite[Theorem 3.5.1]{X}.

\begin{theorem}\label{TFH10}
	If the class $R\operatorname{-PProj}$ of pure projective left $R$-modules is closed under extensions, then
	every FP-projective module is pure projective.
\end{theorem}

\begin{proof} As $R \operatorname{-PProj}$ is closed under
	extensions, we have $2\operatorname{-Filt}(R\textmd{-PProj})=R\textmd{-PProj},$ then $\Psi= \langle R\textmd{-PProj} \rangle^\perp= \langle 2\operatorname{-Filt}(R\textmd{-PProj}) \rangle ^\perp =\Psi^2$. So $\Psi$ is an idempotent ideal in $R \mbox{-Mod}$. By employing  Proposition~\ref{PFH41} and Proposition~\ref{TFH6},  we have
	$\Psi=\Psi^{(\omega)}=\langle\mathcal{FI}\rangle$.
	Since $(\mathcal{FP},\mathcal{FI})$ is a complete cotorsion pair, $( \langle \mathcal{FP} \rangle, \langle \mathcal{FI} \rangle)$ is an ideal cotorsion pair. So $$ \mbox{add} (R\textmd{-Proj}\star R\textmd{-PProj})= \mbox{Ob}( {}^\perp \Psi) =\mbox{Ob}(  {}^\perp\langle\mathcal{FI}\rangle) =\mbox{Ob}( \langle\mathcal{FP}\rangle )=\mathcal{FP}.$$
	Note that any projective left $R$-module is also pure projective and $R\textmd{-PProj}  \subseteq  R\textmd{-Proj}\star R\textmd{-PProj} $. Then the hypothesis implies that
	$$R\textmd{-Proj}\star R\textmd{-PProj} \subseteq R\textmd{-PProj} \subseteq R\textmd{-Proj}\star R\textmd{-PProj},  $$ that is, $R\textmd{-Proj}\star R\textmd{-PProj} =R\textmd{-PProj}$.
	Besides, the class $R\textmd{-PProj}$ is closed under direct summand, and therefore, we have $R\textmd{-PProj}=\mbox{add}(R\textmd{-PProj})=\mathcal{FP}$.
\end{proof}

As an immediate consequence, we have
\begin{corollary}
	Let $R$ be a left noetherian ring. 	If the class $R\operatorname{-PProj}$ of pure projective left $R$-modules is closed under extensions, then $R$ is left pure semisimple.
	\end{corollary}

It is clear that if $R$ is a von Neumann regular ring or a left pure semisimple ring, then the class  $R\operatorname{-PProj}$ of pure projective left $R$-modules  is closed under extensions. The direct product $R=R_1\times R_2$, with $R_1$ a von Neumann regular ring and $R_2$ a left pure semisimple ring (none of them semisimple), gives an easy example of a ring $R$ which is neither von Neumann nor left pure semisimple for which the class the class $R\operatorname{-PProj}$ of pure projective left $R$-modules  is closed under extensions. However, the authors are unaware of a left perfect and non left pure semisimple ring having this property (see Problem \ref{problem3} below).

\subsection{Problems}\label{problems} We conclude this paper with some problems that might be interesting.

 If $\mcS$ is the set of isomorphism class of the indecomposable pure injective modules, then, following Remark \ref{dual.results}, $\mathfrak{Cog}_{\mcS} = \Phi$ is the ideal of phantom morphisms in $\RMod$ \cite{FGHT}, which can be characterized as the morphisms $f$ for which $\Tor_1 (-, f) = 0$; see \cite{H1}.
Then, under the assumptions of Remark \ref{dual.results}, one can prove that
$$(\mathfrak{Cog}_{\mcS}^{\grb}, \grb\mbox{-} \Cofilt(\RInj \star \Prod (\mcS)))$$ is a complete cotorsion pair; the case $\grb =1$ comes from~\cite[Proposition 37]{FGHT}. It follows that every module $M$ admits, for every ordinal $\grb,$ a special $\Phi^{\grb}$-precover whose kernel is a $\grb$-cofiltration with factors from $\RInj \star \RPInj.$ This allows us to pose questions concerning the convergence of the projective powers of the phantom ideal to the object ideal of flat modules.

\begin{problem}{\rm
		We already know that the phantom ideal $\Phi$ is a covering ideal by \cite{H1}. It might be interesting to know if the ideal $\Phi^{\grb}$ is still covering for an ordinal $\grb\geq 2$.}
\end{problem}

\begin{problem}
	{\rm Does equality hold for the last inclusion of
		$$\Phi \supseteq \Phi^2 \cdots \supseteq \Phi^{\grb} \supseteq \cdots \supseteq \bigcap_{\grb \in \On} \Phi^{\grb} \supseteq \langle \RFlat \rangle?$$ If so, given a left $R$-module $M,$ find an ordinal $\grb$ for which $\Phi^{\grb} (-,M) = \langle \RFlat \rangle (-,M).$ Moreover, is there an ordianl $\beta$ such that $\Phi^{\grb}=\langle \RFlat \rangle?$}
\end{problem}

 In \cite[Example 8.3]{HR} it is shown an example of a non-pure semisimple, left and right noetherian ring for which the class of $\operatorname{PInj}$ modules (both to the left and to the right) is closed under extensions. Therefore, in view of Theorem \ref{TFH10}, the class of $\operatorname{PProj}$ of pure projective $R$-modules (right or left) cannot be closed under extensions. Now, for a left perfect ring $R$, we know from Xu's result that the class $R\operatorname{-Pinj}$ is closed under extensions if and only if the ring is left pure semisimple.  This suggests us to ask the following problem.

\begin{problem}\label{problem3}
	{\rm Is there an example of a left perfect and non left pure semisimple ring  for which the class $R\operatorname{-PProj}$ of pure projective left $R$-modules  is closed under extensions?}
\end{problem}

\section*{Acknowledgements}

The authors wish to thank Philipp Rothmaler for helpful comments and suggestions and Yifan Chen for helping us to make the three dimensional diagram in the proof of Theorem \ref{TFH4}.

\end{document}